\documentclass[12pt,reqno]{amsart}
\usepackage{amsmath}
\usepackage{amssymb}
\usepackage{verbatim}
\usepackage{mathrsfs}
\usepackage{graphicx}
\usepackage{caption}
\usepackage{subcaption}
\usepackage{epstopdf}
\usepackage{makecell}
\usepackage{tabularx}
\usepackage{longtable}
\usepackage{tikz}
\usetikzlibrary{external}
\usepackage{pgfplots}
\usepackage{xcolor}
\usepackage{pgfplots}
\usepackage{pgfplotstable}
\usepackage[ruled,vlined]{algorithm2e}
\usepackage[noend]{algpseudocode}
\usepackage[toc]{appendix}
\usepackage[capitalise]{cleveref}
\usepackage[noadjust]{cite}
\usepackage{color}
\usepackage{extarrows}

\usetikzlibrary{arrows}

\usepackage[top=0.6in,bottom=0.5in,left=0.7in,right=0.7in]{geometry}

\newtheorem{lemma}{Lemma}

\newtheorem{theorem}[lemma]{Theorem}

\newtheorem{example}{Example}

\numberwithin{equation}{section}
\numberwithin{lemma}{section}

\newcommand{\N}{\mathbb{N}}    
\newcommand{\R}{\mathbb{R}}    

\newcommand{\ra}{\textsf{r}}
\newcommand{\ca}{\textsf{c}}
\newcommand{\ka}{\textsf{k}}

\newcommand{\Z}{\mathbb{Z}}    
\DeclareMathOperator\supp{supp}

\newcommand{\be}{ \begin{equation} }
	\newcommand{\ee}{ \end{equation} }

\begin{document}
	
	\title[]{A Derivative-Orthogonal Wavelet Multiscale Method for 1D Elliptic Equations with Rough Diffusion Coefficients
	}
	
\author{Qiwei Feng and Bin Han}

\thanks{Research supported in part by  the Mathematics Research Center, Department of Mathematics, University of Pittsburgh, Pittsburgh, PA, USA and Natural Sciences and Engineering Research Council (NSERC) of Canada under grants RGPIN-2024-04991.
}	
	
\address{Department of Mathematics, University of Pittsburgh, Pittsburgh, PA 15260 USA.
	\quad {\tt qif21@pitt.edu, 	 fengq2023@gmail.com}}	
	
\address{Department of Mathematical and Statistical Sciences,
	University of Alberta, Edmonton, Alberta, Canada T6G 2G1.
	\quad {\tt bhan@ualberta.ca}
}

	\makeatletter \@addtoreset{equation}{section} \makeatother

	\begin{abstract}
In this paper, we investigate 1D elliptic equations
$-\nabla\cdot (a\nabla u)=f$ with rough diffusion coefficients  $a$ that satisfy
$0<a_{\min}\le a\le a_{\max}<\infty$
and $f\in L_2(\Omega)$.
To achieve an accurate and robust numerical solution on a coarse mesh of size $H$,
we introduce a derivative-orthogonal wavelet-based framework. This approach incorporates both regular and specialized basis functions constructed through a novel technique, defining a basis function space that enables effective approximation.
We develop a derivative-orthogonal wavelet multiscale method tailored for this framework, proving that the condition number
$\kappa$ of the stiffness matrix satisfies
$\kappa\le a_{\max}/a_{\min}$, independent of $H$. For the error analysis, we establish that the energy and
$L_2$-norm errors of our method converge at first-order and second-order rates, respectively, for any coarse mesh
$H$. Specifically, the energy and
$L_2$-norm errors are bounded by
$2 a_{\min}^{-1/2} \|f\|_{L_2(\Omega)} H$ and $4 a_{\min}^{-1}\|f\|_{L_2(\Omega)} H^2$. Moreover, the numerical approximated solution also possesses the interpolation property at all grid points.
We present a range of challenging test cases with continuous, discontinuous, high-frequency, and high-contrast coefficients $a$ to evaluate errors in
$u, u'$ and $a u'$ in both
$l_2$ and $l_\infty$ norms.
We also provide a numerical example that both coefficient $a$ and source term $f$ contain discontinuous, high-frequency and high-contrast oscillations.
Additionally, we compare our method with the standard second-order finite element method to assess error behaviors and condition numbers when the mesh is not fine enough to resolve coefficient oscillations. Numerical results confirm the bounded condition numbers and convergence rates, affirming the effectiveness of our approach. Thus, our method is capable of handling both the rough diffusion coefficient $a$ and the rough source term $f$, provided that $\|f\|_{L_2(\Omega)}$ is not huge.
\end{abstract}	


\maketitle

\pagenumbering{arabic}

	\section{Introduction and problem formulation}\label{introdu:1:intersect}
The elliptic problems
with rough, discontinuous, high-contrast, and high-frequency diffusion coefficients arise in many real-world applications (\cite{AliMankad2020,GuiraAusas2019,ArbogastXiao2013,ArbogastTao2013,Jaramillo2022,Kippe2008,Tahmasebi2018,Butler2020,TArbHXiao2015,Rasaei2008}), such as the demanding SPE10 (Society of Petroleum Engineers 10) benchmark problem.
The discontinuities, high-contrast jumps, and high frequencies result from the permeability variations in the different geological layers of the SPE10 problem  \cite{Vazqu07}. Usually,  the solution of the elliptic equation
with the rough diffusion coefficient contains severe singularities (see e.g. \cite{CaKi01,Kell75,KiKo06,Blumen85,Petz01,Petz02,Nic90,KCPK07,NS94,FGW73,Kell71,Kell72}).
Since diffusion coefficients typically exhibit high-frequency oscillations and encompass multiple spatial scales, employing the standard finite element and finite volume methods with linear basis functions necessitates the fine grid mesh to capture  small-scale features of  coefficients. This requirement of using very fine meshes leads to the high computational cost and large stiffness matrices. To obtain a reliable and precise solution with the detailed oscillation on the coarse mesh size, various multiscale  methods are proposed (see e.g. \cite{Hellmanvist2017,PietroErn2012,lqvistPeterseim2014,EngwerHenning2019,LiHu2021,FuChungLi2019,HouHwang2017,ZhangCiHou2015,lqvistPersson2018,Allaire2005,ChuGrahamHou2010,HouWu1997,HouWuCai1999,JennyLeeTchelepi2004,Nolen2008,Srinivasan2016}).

			In this paper,  we discuss the following 1D elliptic equation with the rough diffusion coefficient in the domain $\Omega:=(0,1)$:
		\be \label{Model:Problem}
\begin{cases}
-\nabla \cdot (a \nabla  u)=f &\text{in $\Omega$,}\\
u=0  &\text{on $\partial \Omega$.}
\end{cases}
		\ee
		We investigate the model problem \eqref{Model:Problem} only under the following two assumptions:
	\begin{itemize}
		
		\item[(A1)]  There exist positive constants $a_{\min}=\inf_{x\in \Omega} a(x)$ and $a_{\max}=\sup_{x\in \Omega} a(x)$ such that $0<a_{\min}\le a\le  a_{\max}<\infty$ in $\Omega$.
		
		\item[(A2)]  The source term $f\in L_2(\Omega)$.
		
	\end{itemize}
The weak formulation of the equation \eqref{Model:Problem} is to find $u\in H^1_0(\Omega)$ such that
	\be\label{week:form}
	\int_\Omega a \nabla u \cdot \nabla v=\int_\Omega f v,\ \ \ \ \mbox{for all}\ v\in H^1_0(\Omega).
	\ee
The organization of this paper is as follows.

In \cref{sec:regular}, we utilize derivative-orthogonal
wavelets to construct  regular basis functions. In \cref{sec:special}, we propose a novel algorithm to develop special basis functions. In \cref{sec:multiscale}, we define our basis function spaces by  regular and special basis functions, and then
derive our derivative-orthogonal wavelet multiscale method. In \cref{TTY2} of \cref{sec:condition:number}, we prove that the condition number of the stiffness matrix of our derivative-orthogonal wavelet multiscale method is bounded by ${ a_{\max} }/{a_{\min}}$ for any coarse mesh size $H$. In \cref{sec:error:estimates},  we derive \cref{Thm:u:prime:L2} and \cref{Thm:u:L2} to theoretically prove that
the energy and
$L_2$-norm errors of our method converge at first-order and second-order rates, respectively, for any coarse mesh
$H$. Specifically, the energy and
$L_2$-norm errors are bounded by
$C_1 H$ and $C_2 H^2$ with
$C_1=2\|f\|_{L_2(\Omega)}/{\sqrt{a_{\min}}}$ and $C_2=4 {\|f\|_{L_2(\Omega)}}/{a_{\min}}$.
We shall also prove in \cref{Thm:u:infty} that the numerical approximated solution also possesses the interpolation property at all grid points.

In \cref{sec:numerical}, we provide numerical experiments to verify convergence rates measured in $l_2$ and $l_{\infty}$ norms. We test numerical examples in the following 6 cases:
\begin{itemize}
	\item  \cref{Special:ex1}: $u$ is available and $a$ is continuous with high-frequency oscillations in $\Omega$ .
	\item  \cref{Special:ex2}: $u$ is available and $a$ is discontinuous with high-contrast jumps in $\Omega$.
	\item  \cref{Special:ex2:infty}: $u$ is available, $a$ is continuous with high-frequency oscillations, 
$\lim_{x\to 0^+}a(x)=\infty$.
	\item  \cref{Special:ex3}: $u$ is unavailable and $a$ is continuous with high-frequency oscillations in $\Omega$.
	\item  \cref{Special:ex4}: $u$ is unavailable and $a$ is discontinuous with high-contrast jumps in $\Omega$.
	\item  \cref{Special:ex6}: $u$ is unavailable,  $a$ and $f$ are both discontinuous with high-contrast jumps and  high-frequency oscillations in $\Omega$.
\end{itemize}
In \cref{Special:ex1,Special:ex2,Special:ex2:infty}, we compare our proposed derivative-orthogonal wavelet multiscale method with the standard second-order
finite element method.

In \cref{Special:sec:Conclu}, we summarize the main contributions of our paper.
	\section{The Derivative-Orthogonal Wavelet Multiscale Method}
We now describe our proposed derivative-orthogonal wavelet multiscale method to solve \eqref{Model:Problem}.

	\subsection{Regular basis functions}\label{sec:regular}		
	In this subsection, we construct regular basis functions by using derivative-orthogonal wavelets \cite{HanMichelle2019} as follows:
	\be\label{phipsi}
	\phi(x)=\begin{cases}
		x+1, &x\in [-1,0),\\
		1-x, &x\in [0,1),\\
		
		0, &\text{else},
	\end{cases}
	\qquad \psi(x)=\begin{cases}
		2x, &x\in [0,\frac{1}{2}),\\
		2-2x, &x\in [\frac{1}{2},1),\\
		0, &\text{else}.
	\end{cases}
	\ee
	Then
	\be\label{phipsi:shift}
	\phi(2x-1)=\begin{cases}
		2x, &x\in [0,\frac{1}{2}),\\
		2-2x, &x\in [\frac{1}{2},1),\\
		0, &\text{else},
	\end{cases}
	\qquad
		(\phi(2x-1))'=\begin{cases}
		2, &x\in [0,\frac{1}{2}),\\
		-2, &x\in [\frac{1}{2},1),\\
		0, &\text{else},
	\end{cases}
	\ee		
	and
	\be\label{phipsi:shift:prime}
	\psi(2^jx-k)=\begin{cases}
		2^{j+1}x-2k, &x\in [\frac{k}{2^j},\frac{1/2+k}{2^j}),\\
		2-2^{j+1}x+2k, &x\in [\frac{1/2+k}{2^j},\frac{1+k}{2^j}),\\
		0, &\text{else},
	\end{cases}
	\qquad (\psi(2^jx-k))'=\begin{cases}
		2^{j+1}, &x\in [\frac{k}{2^j},\frac{1/2+k}{2^j}),\\
		-2^{j+1}, &x\in [\frac{1/2+k}{2^j},\frac{1+k}{2^j}),\\
		0, &\text{else},
	\end{cases}
	\ee
	where $k=0,\ldots,2^j-1$ and  $j=1,\ldots,n-1$ with $n\in \N$ and $n\ge 2$.
	Now, we define the derivative-orthogonal
	wavelet space $V_H:=\mbox{span}(\mathcal{B}_H)$ with the uniform  coarse mesh size $H$ as follows:
	\be\label{Bh} \mathcal{B}_H:=\{\phi(2x-1)\}\cup\{\psi(2^jx-k):\ k=0,\ldots,2^j-1,\  j=1,\ldots,n-1\},
	\ee
	where $H=1/2^n$ for an integer $n\ge 2$.
	See \cref{regular:base:1,regular:base:2,regular:base:4,regular:base:5} for illustrations of $\{\phi(2x-1)\}$, $\{\psi(2^jx-k):\ k=0,\dots,2^j-1\}$, $\{(\phi(2x-1))'\}$, $\{(\psi(2^jx-k))':\ k=0,\dots,2^j-1\}$ with $j=1,2$.
	Clearly,
	\be\label{VH:0}
	V_H:=\mbox{span}(\mathcal{B}_H)   \subset H_0^1(\Omega).
	\ee
	Let $\supp(f)$ denote the support of the function $f$. By
 \eqref{phipsi:shift:prime}, the elements in the basis $\mathcal{B}_H$ satisfy
	\be\label{regular:support:1}
	\begin{split}
		& \supp((\psi(2^jx-k_0))') \cap \supp((\phi(2x-1))')=\supp((\psi(2^jx-k_0))') \quad \text{or} \quad \emptyset,\\
		& \supp((\psi(2^jx-k_1))')\cap \supp((\psi(2^jx-k_2))')=\emptyset, \quad \text{if } k_1\ne k_2,
	\end{split}
	\ee	
	for any  $k_0,k_1,k_2=0,\ldots,2^j-1$ and $j=1,\ldots,n-1$.
	Furthermore, for any $1\le j_{1}\le j_{2}$ with $j_{1},j_{2}\in \N$,
	\be\label{regular:support:2}
	\supp((\psi(2^{j_{1}}x-k_{1}))') \cap \supp((\psi(2^{j_{2}}x-k_{2}))')= \supp((\psi(2^{j_{2}}x-k_{2}))') \quad \text{or} \quad \emptyset,
	\ee
	where $k_{1}=0,\ldots,2^{j_{1}}-1$ and $k_{2}=0,\ldots,2^{j_{2}}-1$.

	\begin{figure}[htbp]
		\centering	
					\centering	
	\begin{subfigure}[b]{0.33\textwidth}
		 \includegraphics[width=6cm,height=4cm]{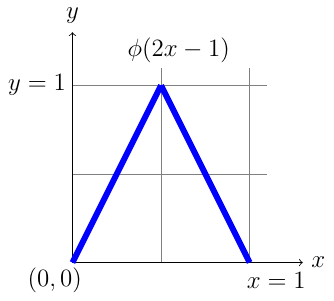}
	\end{subfigure}
	\begin{subfigure}[b]{0.32\textwidth}
		 \includegraphics[width=6cm,height=4cm]{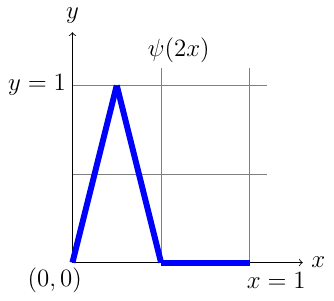}
	\end{subfigure}
	\begin{subfigure}[b]{0.32\textwidth}
		 \includegraphics[width=6cm,height=4cm]{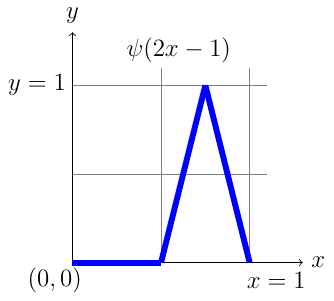}
	\end{subfigure}			
		\caption{Left panel: $\phi(2x-1)$. Middle and right panels: $\{\psi(2^jx-k):\ k=0,\ldots,2^j-1\}$ with $j=1$.}
		\label{regular:base:1}
	\end{figure}
	\begin{figure}[htbp]
		\centering	
				\centering	
		\begin{subfigure}[b]{0.24\textwidth}
			 \includegraphics[width=4.5cm,height=3.5cm]{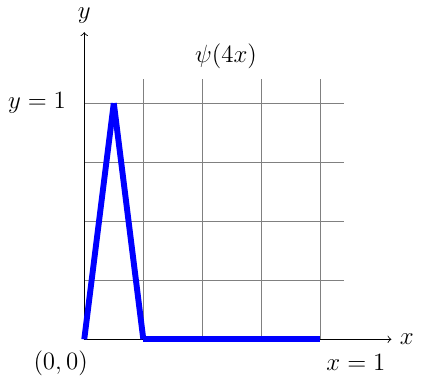}
		\end{subfigure}
		\begin{subfigure}[b]{0.24\textwidth}
			 \includegraphics[width=4.5cm,height=3.5cm]{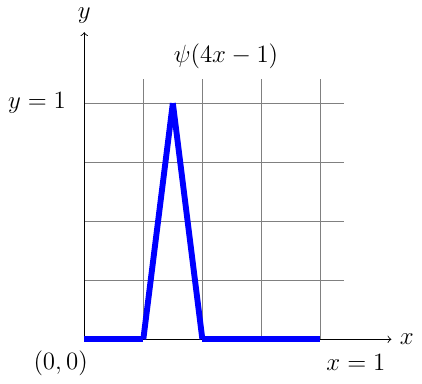}
		\end{subfigure}
		\begin{subfigure}[b]{0.24\textwidth}
			 \includegraphics[width=4.5cm,height=3.5cm]{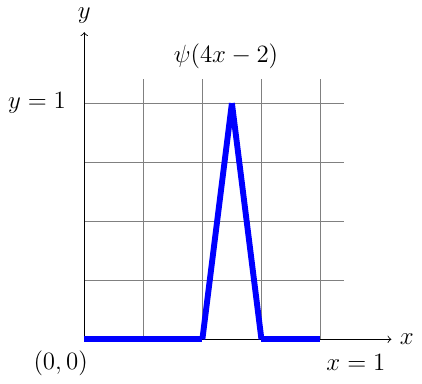}
		\end{subfigure}
		\begin{subfigure}[b]{0.24\textwidth}
			 \includegraphics[width=4.5cm,height=3.5cm]{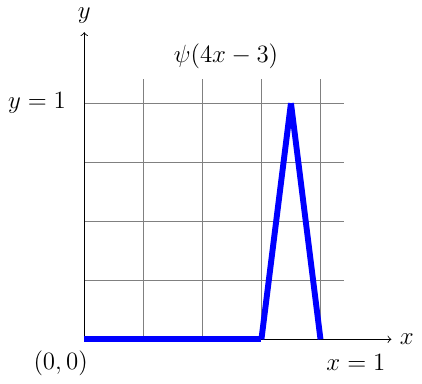}
		\end{subfigure}		
		\caption{$\{\psi(2^jx-k):\ k=0,\ldots,2^j-1\}$ with $j=2$.}
		\label{regular:base:2}
	\end{figure}
	\begin{figure}
				\centering	
		\begin{subfigure}[b]{0.32\textwidth}
			 \includegraphics[width=6cm,height=3.8cm]{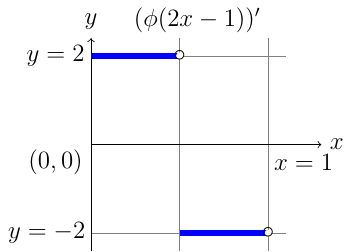}
		\end{subfigure}
		\begin{subfigure}[b]{0.33\textwidth}
			 \includegraphics[width=6cm,height=3.8cm]{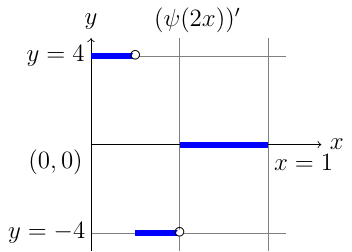}
		\end{subfigure}
		\begin{subfigure}[b]{0.33\textwidth}
			 \includegraphics[width=6cm,height=3.8cm]{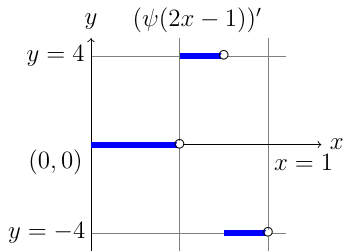}
		\end{subfigure}		
		\caption{Left panel: $(\phi(2x-1))'$. Middle and right panels: $\{(\psi(2^jx-k))':\ k=0,\ldots,2^j-1\}$ with $j=1$.}
		\label{regular:base:4}
	\end{figure}
	\begin{figure}
		\centering	
		\begin{subfigure}[b]{0.24\textwidth}
			 \includegraphics[width=4.5cm,height=3.5cm]{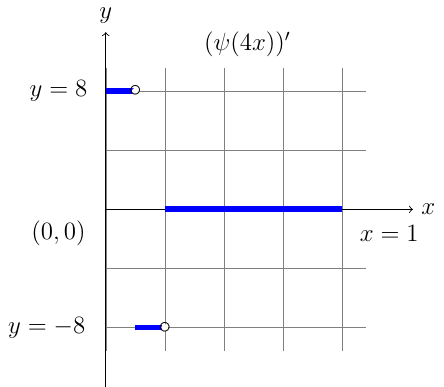}
		\end{subfigure}
		\begin{subfigure}[b]{0.24\textwidth}
			 \includegraphics[width=4.5cm,height=3.5cm]{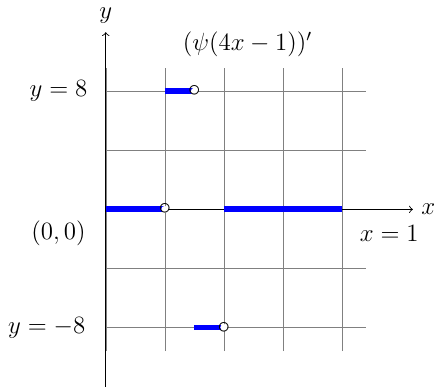}
		\end{subfigure}
		\begin{subfigure}[b]{0.24\textwidth}
			 \includegraphics[width=4.5cm,height=3.5cm]{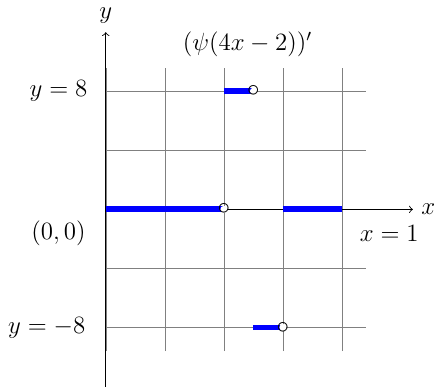}
		\end{subfigure}
		\begin{subfigure}[b]{0.24\textwidth}
			 \includegraphics[width=4.5cm,height=3.5cm]{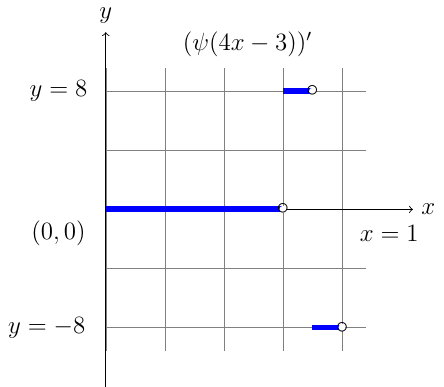}
		\end{subfigure}	
		\caption{$\{(\psi(2^jx-k))':\ k=0,\ldots,2^j-1\}$ with $j=2$.}
		\label{regular:base:5}
	\end{figure}
	\subsection{Special basis functions}\label{sec:special}			
\begin{figure}[htbp]
	\centering
	\begin{subfigure}[b]{0.24\textwidth}
	 \includegraphics[width=4.5cm,height=3.5cm]{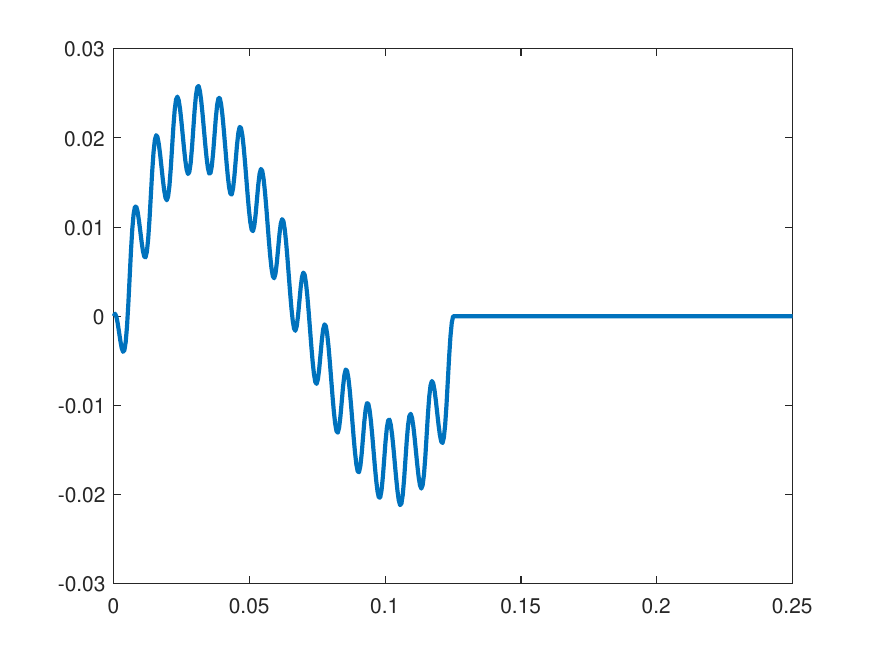}
	\end{subfigure}
	\begin{subfigure}[b]{0.24\textwidth}
	 \includegraphics[width=4.5cm,height=3.5cm]{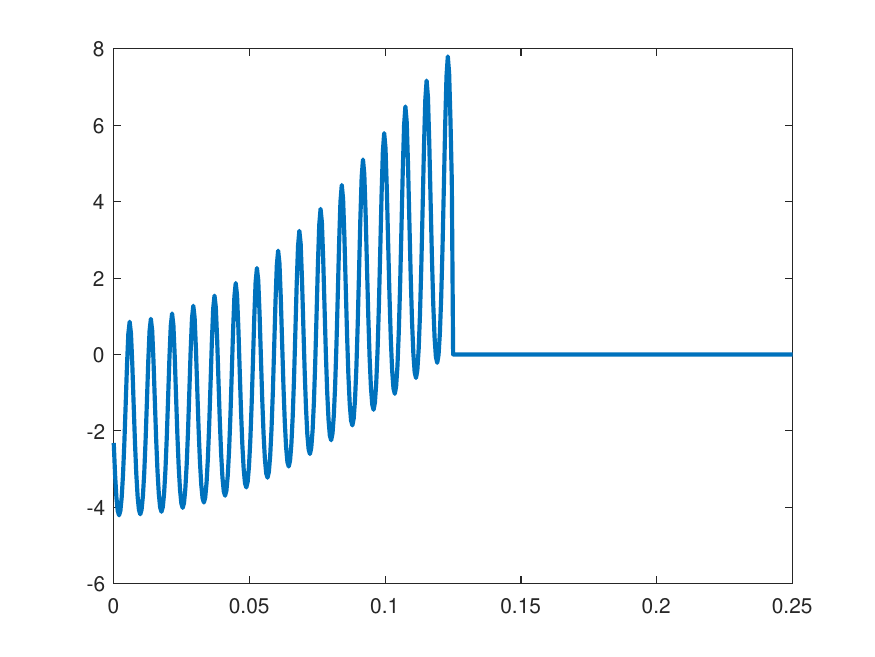}
	\end{subfigure}
	\begin{subfigure}[b]{0.24\textwidth}
	 \includegraphics[width=4.5cm,height=3.5cm]{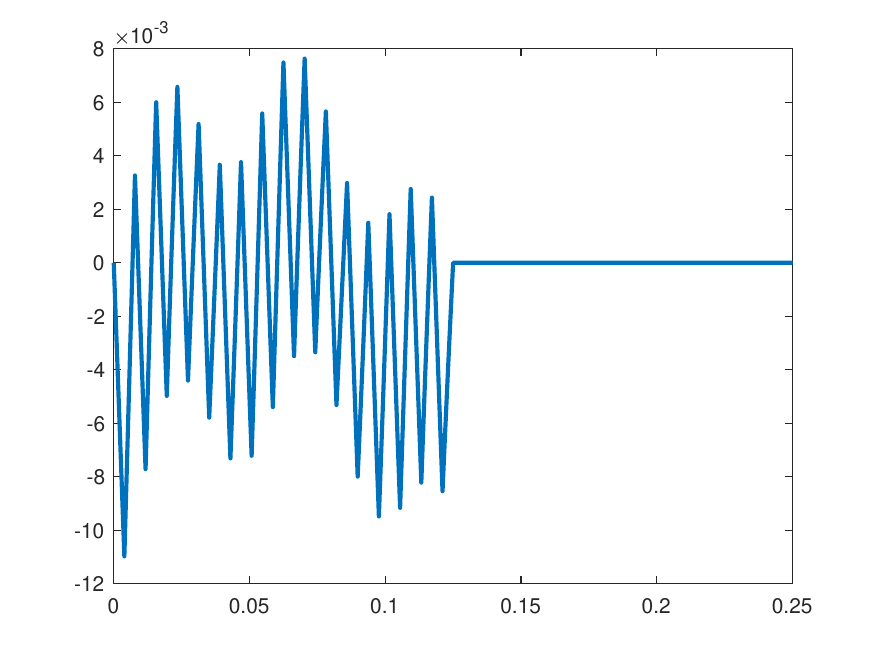}
    \end{subfigure}
    \begin{subfigure}[b]{0.24\textwidth}
	 \includegraphics[width=4.5cm,height=3.5cm]{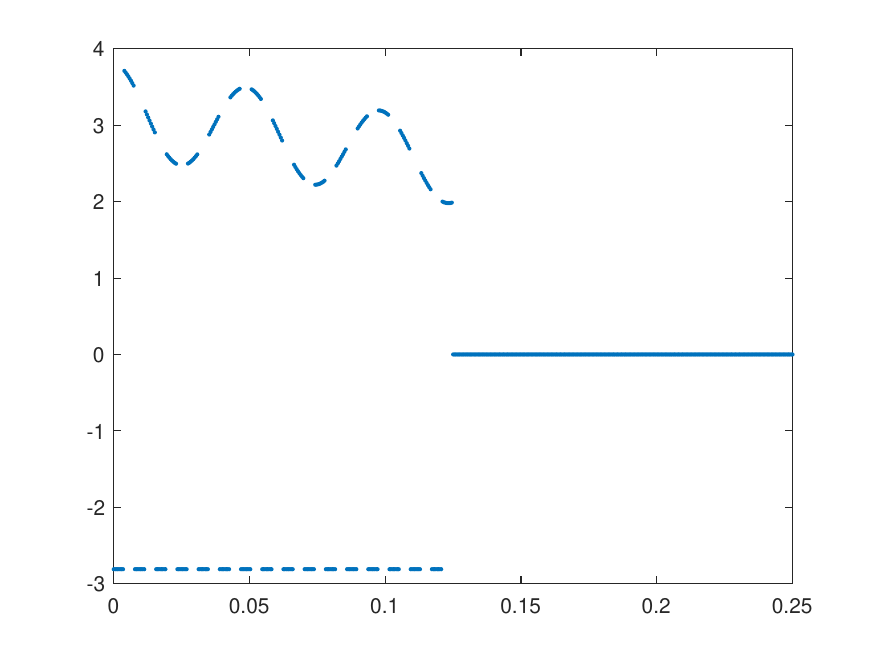}
    \end{subfigure}
	\caption{$S_0$ (first panel) and $(S_0)'$ (second panel) in \eqref{Special:Si} with $n=3$ in $[0,\frac{1}{4}]$ with a continuous coefficient function $a$.  $S_0$ (third panel) and $(S_0)'$ (fourth panel) in \eqref{Special:Si} with $n=3$ in $[0,\frac{1}{4}]$ with a discontinuous coefficient function $a$. Note that $\overline{\supp(S_0)}=\overline{\supp((S_0)')}=[0,\frac{1}{8}]$ in the above two $S_0$.}
	\label{fig:graph:S}
\end{figure}
To capture the high-contrast jump and the high-frequency oscillation  of the coefficient on the coarse mesh size $H$,	 in this subsection, besides regular basis functions, we also construct special basis functions  (see \cref{fig:graph:S} for an illustration) by solving the
	following equations
	\be\label{Spec1}
	\left\{ \begin{array}{lll}
		 (S_i)''&=&\left(\frac{1}{a}\right)',  \qquad x\in (\frac{i}{2^n},\frac{i+1}{2^n}),\\
		\ \ S_i&=&0,  \qquad \hspace{0.6cm} x\in  [0,\frac{i}{2^n}]\cup [\frac{i+1}{2^n},1],
	\end{array}
	\right.
	\ee
	where $i=0,\dots,2^n-1$  for an integer $n\ge 2$.
	For the sake of brevity, we construct special basis functions by using the following equivalent formula of \eqref{Spec1}:
	\be\label{Special:Si}	
	S_i=\begin{cases}
	 \displaystyle{\int_{\frac{i}{2^n}}^{x}
		}\left(\frac{1}{a(t)}-w_i\right)\ dt, &\quad  x\in (\frac{i}{2^n},\frac{i+1}{2^n}),\\
		0, & \quad x\in  [0,\frac{i}{2^n}]\cup [\frac{i+1}{2^n},1],
	\end{cases}	
	\ee
	where  $i=0,\dots,2^n-1$ for an integer $n\ge 2$ and
	\be \label{wi} w_i=\frac{\int_{\frac{i}{2^n}}^{\frac{i+1}{2^n}}\frac{1}{a(t)}dt}{\frac{i+1}{2^n}-\frac{i}{2^n}}=2^n\int_{\frac{i}{2^n}}^{\frac{i+1}{2^n}}\frac{1}{a(t)}dt.
	\ee
	Then \eqref{Special:Si}	 and \eqref{wi} yield
	\be\label{Si:xi1}
	S_{i}\left(\tfrac{i}{2^n}\right)=0,
	\ee
	\be\label{Si:xi2}
	\begin{split}
	S_i\left(\frac{i+1}{2^n}\right)= \int_{\frac{i}{2^n}}^{\frac{i+1}{2^n}}
		\left(\frac{1}{a(t)}-w_i\right) dt= \int_{\frac{i}{2^n}}^{\frac{i+1}{2^n}}
	\frac{1}{a(t)}dt-\frac{w_i}{2^n}=0,
	\end{split}
	\ee
and
	\be\label{Si:prime}
	(S_i)'=\begin{cases}
		\displaystyle{\tfrac{1}{a}}-w_i, &\quad  x\in[\frac{i}{2^n},\frac{i+1}{2^n}),\\
		0, &\quad \text{else},
	\end{cases}	
	\ee
    where  $i=0,\dots,2^n-1$ for an integer $n\ge 2$.
	So
	\be\label{supp:Si}
	\supp((S_i)')\cap \supp((S_j)')=\emptyset, \qquad \quad \text{if } i\ne j.
	\ee
	According to $a>0$, if $a$ is not a constant function on $(\frac{i}{2^n},\frac{i+1}{2^n})$,
we observe
	\[
 \|(S_i)'\|^2_{L_2(\Omega)}=\int_{\frac{i}{2^n}}^{\frac{i+1}{2^n}} \left(\frac{1}{a}-w_i\right)^2 dx>0.
	\]
Note that we throw away $S_i$ if $a$ is a constant on $(\frac{i}{2^n},\frac{i+1}{2^n})$ because $S_i=0$. Clearly,
	  \[
	  \|(\phi(2x-1))'\|^2_{L_2(\Omega)}=4
\quad \mbox{and}\quad
 \|(\psi(2^jx-k))'\|^2_{L_2(\Omega)}=\frac{2^{2j+3}}{2^{j+1}}=2^{j+2}.
	  \]

	\subsection{The proposed derivative-orthogonal wavelet multiscale method}\label{sec:multiscale}			
 Now, we formulate the  derivative-orthogonal wavelet multiscale method space with defined regular  (\cref{sec:regular})  and special  (\cref{sec:special})  basis functions in the  uniform coarse mesh size $H$ as follows:
%
%
%
%
%
\be\label{Btildeh}
\begin{split}
	\tilde{\mathcal{B}}_H=&\bigg\{ \frac{\phi(2x-1)}{\|(\phi(2x-1))'\|_{L_2(\Omega)}}\bigg\}\cup\bigg\{\frac{\psi(2^jx-k)}{
\|(\psi(2^jx-k))'\|_{L_2(\Omega)}} \; :\; k=0,\ldots,2^j-1,\  j=1,\ldots,n-1 \bigg\}\\ &\cup\bigg\{\frac{S_i}{\|(S_i)'\|_{L_2(\Omega)}}:\ i=0,\ldots,2^n-1\bigg\},
\end{split}
\ee
where $H=1/2^n$ for an integer $n\ge 2$.
Note that $\|g'\|_{L_2(\Omega)}=1$ for every $g\in \tilde{\mathcal{B}}_H$ and
\be\label{Btildeh:H01}
\begin{split}
	 \tilde{V}_H:=\mbox{span}(\tilde{\mathcal{B}}_H) \subset H_0^1(\Omega).
\end{split}
\ee	
By \eqref{phipsi:shift} and \eqref{Si:prime},
\be	\label{supp:Si:2}
\begin{split}
	&\left[0,\tfrac{1}{2}\right]  \cap \supp((S_i)')= \supp((S_i)') \quad \text{or} \quad \emptyset,\\
	&\left[\tfrac{1}{2},1\right]  \cap \supp((S_i)')= \supp((S_i)') \quad \text{or} \quad \emptyset,
\end{split}
\ee
where   $i=0,\dots,2^n-1$ and $n\ge 2$.	   \eqref{phipsi:shift:prime} and \eqref{Si:prime} imply
\be	\label{supp:Si:3}
\begin{split} &\left[\tfrac{k}{2^j},\tfrac{k+1/2}{2^j}\right] \cap \supp((S_i)')= \supp((S_i)') \quad \text{or} \quad \emptyset,\\
	 &\left[\tfrac{k+1/2}{2^j},\tfrac{k+1}{2^j}\right] \cap \supp((S_i)')= \supp((S_i)') \quad \text{or} \quad \emptyset,
\end{split}
\ee
where $k=0,\dots,2^j-1$, $j=1,\dots,n-1$, $i=0,\dots,2^n-1$ and $n\ge 2$.	
For the sake of brevity, we rewrite $\tilde{\mathcal{B}}_H$  in \eqref{Btildeh} as
\be\label{Btildeh:2}
\tilde{V}_H=\mbox{span}(\tilde{\mathcal{B}}_N)
\quad \mbox{and}\quad\
\tilde{\mathcal{B}}_N= \big\{ \varphi_1,  \varphi_2,  \dots,  \varphi_M \big\},
\ee
where 		
\be\label{number:M}
M=1+\left(\sum_{j=1}^{n-1}\sum_{k=0}^{2^j-1}1\right)+2^n=1+\left(\sum_{j=1}^{n-1}2^j\right)+2^n=1+\frac{2-2^n}{1-2}+2^n
=2^{n+1}-1.
\ee
	Now, the numerical solution of our proposed derivative-orthogonal wavelet multiscale method is computed by the following procedure: Find $\tilde{u}_H\in  \tilde{V}_H:=\mbox{span}(\tilde{\mathcal{B}}_H)$ such that
	\be\label{M:FEM}
	\int_\Omega a \nabla \tilde{u}_H \cdot \nabla \tilde{v}_H=\int_\Omega f \tilde{v}_H,\ \ \ \ \mbox{for all}\ \tilde{v}_H\in \tilde{\mathcal{B}}_H,
	\ee
	where $\tilde{\mathcal{B}}_H$ is defined in \eqref{Btildeh} or equivalently \eqref{Btildeh:2} and is the basis of $\tilde{V}_H$ with the derivative-orthogonality property (\cite{HanMichelle2019}). See \cite{Hanbook,HanMichelle2019} and many references therein for wavelet-based applications for numerical PDEs.
Note that $\mathcal{B}_{H} \subset \tilde{\mathcal{B}}_{H}$, while $\mathcal{B}_{H/2}$ and $\tilde{\mathcal{B}}_H$ have the same cardinality $2^{n+1}-1$.
    %
	%
	%

%
%

%
%
	%
	%
	\subsection{Bounded condition numbers}\label{sec:condition:number}				 
	In this subsection, we first state the following \cref{lem:ortho} which is used to prove that the condition number of our proposed derivative-orthogonal wavelet multiscale method is bounded by $a_{\max}/a_{\min}$.
	\begin{lemma}\label{lem:ortho}
		For any $\varphi_i$ and $\varphi_j\in \tilde{\mathcal{B}}_H$,
		\be\label{orthorgonal:is:0}
		\int_{\Omega} (\varphi_i)' (\varphi_j)'\ dx=\delta_{ij},
		\ee
	where $\tilde{\mathcal{B}}_H$ is defined in \eqref{Btildeh} and \eqref{Btildeh:2}.		 Hence, $\tilde{\mathcal{B}}_H$ forms a derivative-orthogonal basis of $\tilde{V}_H$.
	\end{lemma}
	\begin{proof}
		
		By   \eqref{phipsi:shift}--\eqref{regular:support:2}, we deduce
		\[
		\int_{\Omega} (\varphi_i)' (\varphi_j)'=0, \quad \text{if } \varphi_i\ne \varphi_j, \quad \text{for any }  \varphi_i,\varphi_j\in \mathcal{B}_H,
		\]
		where $\mathcal{B}_H$ is defined in \eqref{Bh}.
		Using \eqref{phipsi:shift}, we obtain that for any $S_i$ in \eqref{Special:Si}
		\[
		\int_{\Omega} (\phi(2x-1))' (S_i)'dx
		=	2\int_{0}^{1/2}  (S_i)'dx-2\int_{1/2}^{1}  (S_i)'dx.		
		\]
		 Note that \eqref{Si:prime} implies $\supp((S_i)')\subset$ $[0,1/2]$ or $[1/2,1]$.
		Then
		\[
		\begin{split}
			&\int_{0}^{1/2}  (S_i)'dx=\int_{\frac{i}{2^n}}^{\frac{i+1}{2^n}}  (S_i)'dx \quad \text{or} \quad 0,\\
			&\int_{1/2}^{1}  (S_i)'dx=\int_{\frac{i}{2^n}}^{\frac{i+1}{2^n}}  (S_i)'dx \quad \text{or} \quad 0.
		\end{split}
		\]
		Note that \eqref{Si:xi1} and \eqref{Si:xi2} result in
		\be\label{Si:is:0}
		\begin{split}
			 \int_{\frac{i}{2^n}}^{\frac{i+1}{2^n}}  (S_i)'dx=S_i\left(\frac{i+1}{2^n}\right)-S_i\left(\frac{i}{2^n}\right)=0.
		\end{split}
		\ee
		So
		\[
		\int_{\Omega} (\phi(2x-1))' (S_i)'dx=0, \quad \text{for any }  i=0,\dots, 2^n-1.
		\]
		On the other hand, by \eqref{phipsi:shift:prime},
		\[
		 \int_{\Omega}(\psi(2^jx-k))'(S_i)'dx=2^{j+1}\int_{\frac{k}{2^j}}^{\frac{1/2+k}{2^j}} (S_i)'dx-2^{j+1}\int_{\frac{1/2+k}{2^j}}^{\frac{1+k}{2^j}} (S_i)'dx.
		\]
		 Now \eqref{supp:Si:3} yields
		\[
		 \int_{\Omega}(\psi(2^jx-k))'(S_i)'dx=\begin{cases}
			 2^{j+1}\int_{\frac{k}{2^j}}^{\frac{1/2+k}{2^j}}  (S_i)'dx=2^{j+1}\int_{\frac{i}{2^n}}^{\frac{i+1}{2^n}}  (S_i)'dx, &\quad  \supp((S_i)')\subset [\frac{k}{2^j},\frac{1/2+k}{2^j}],\\
			 -2^{j+1}\int_{\frac{1/2+k}{2^j}}^{\frac{1+k}{2^j}}  (S_i)'dx=-2^{j+1}\int_{\frac{i}{2^n}}^{\frac{i+1}{2^n}}  (S_i)'dx, & \quad  \supp((S_i)')\subset [\frac{1/2+k}{2^j},\frac{1+k}{2^j}],\\
			0, &\quad \text{else}.
		\end{cases}	
		\]
	 From \eqref{Si:is:0}, we observe
		\[
		 \int_{\Omega}(\psi(2^jx-k))'(S_i)'dx=0.
		\]
		Recall  \eqref{supp:Si} that
		$\supp((S_i)')\cap \supp((S_j)')=\emptyset$ if $i\ne j$.
		Now, we can conclude that
		\[
		\int_{\Omega} (\varphi_i)' (\varphi_j)'\ dx=0, \quad \text{if }  \varphi_i\ne \varphi_j\quad \text{for any } \varphi_i,\varphi_j\in \tilde{\mathcal{B}}_H.
		\]
	By the definition of $\tilde{\mathcal{B}}_H$ in \eqref{Btildeh}, it is straightforward to verify
		\[
		\int_{\Omega} (\varphi_i)' (\varphi_i)'\ dx=1, \quad \text{for any } \varphi_i\in \tilde{\mathcal{B}}_H.
		\]
		Thus, \eqref{orthorgonal:is:0} is proved.
	\end{proof}
	Applying \eqref{orthorgonal:is:0} in	 \cref{lem:ortho},  the following theorem  presents that the  condition number $\kappa$ of our proposed derivative-orthogonal wavelet multiscale method satisfies $\kappa \le a_{\max}/a_{\min}$.
	\begin{theorem}\label{TTY2}
    Define $a_{\max}:=\sup_{x\in \Omega} a(x)$ and $a_{\min}:=\inf_{x\in \Omega} a(x)$. Suppose that $0<a_{\min}\le a(x)\le a_{\max}<\infty$ for all $x\in \Omega$.
		Let $A$ be the corresponding stiffness matrix of \eqref{M:FEM} and $\kappa$ be its condition number by using  the basis $\tilde{\mathcal{B}}_H$ defined in \eqref{Btildeh} of $\tilde{V}_H$.
		Then
		$A$ is  an $M\times M$ symmetric positive definite matrix   and
		\be
		\kappa\le \frac{ a_{\max} }{a_{\min}},
		\ee
		where $M:=2^{n+1}-1$ and $n$ is an integer such that the uniform mesh size $H=1/2^n$ and  $n\ge 2$.
	\end{theorem}
	\begin{proof}
Note that $\tilde{\mathcal{B}}_H=\{\varphi_1,\ldots,\varphi_M\}$ in \eqref{Btildeh:2}.
		By \eqref{Btildeh}, 	 \eqref{Btildeh:H01}, and  \eqref{M:FEM}, we have
		\[
A_{ij}=	\int_{\Omega} a(x)(\varphi_i)' (\varphi_j)'\ dx,\qquad i,j=1,\ldots,M.
\]
		  Let $c=(c_1,\ldots,c_{M})^T\in \R^M$ be a column vector with each $c_i\in \R$. Then
		\[
c^T A c=	\int_{\Omega} a(x)\Big(\sum_{i=1}^{M} c_i (\varphi_i)'\Big) \Big(\sum_{j=1}^{M} c_j (\varphi_j)'\Big) dx=	 \int_{\Omega} a(x) \Big(\sum_{i=1}^{M} c_i (\varphi_i)'\Big)^2 dx.\]
		By $0<a_{\min}\le a(x)\le a_{\max}<\infty$ for all $x\in \Omega$ and $(\sum_{i=1}^{M} c_i (\varphi_i)')^2\ge 0$, we conclude that
		\[
a_{\min}
\int_{\Omega} \Big(\sum_{i=1}^{M} c_i (\varphi_i)'\Big)^2 dx \le c^T A c\le
a_{\max}
	\int_{\Omega} \Big(\sum_{i=1}^{M} c_i (\varphi_i)'\Big)^2  dx.
\]
As each $c_i\in \R$ and
note that $\|(\varphi_i)'\|^2_{L_2(\Omega)}=1$ for any $\varphi_i\in \tilde{\mathcal{B}}_H$ in \eqref{Btildeh}, we conclude from \eqref{orthorgonal:is:0} in \cref{lem:ortho} that
\[
\int_{\Omega} \Big(\sum_{i=1}^{M} c_i (\varphi_i)'\Big)^2 dx
=\int_{\Omega} \sum_{i=1}^{M} \Big(c_i(\varphi_i)'\Big)^2  dx
=\sum_{i=1}^{M} (c_i)^2.
\]
%
		%
		%
Consequently, we proved
		\[
a_{\min}
\sum_{i=1}^{M} (c_i)^2 \le c^T A c\le
a_{\max}
\sum_{i=1}^{M} (c_i)^2, \qquad \forall\; c=(c_1,\ldots, c_M)^T\in \R^M.
\]
		%
		%
		%
		Now, we can see that the stiffness matrix $A$ is a symmetric positive definite matrix, which implies the unique expression of $\tilde{u}_H\in  \text{span} (\tilde{\mathcal{B}}_H )$ by solving \eqref{M:FEM}.
		%
		%
		%
		%
		%
Now we conclude from the inequalities that
the maximum and minimum eigenvalues of the matrix $A$ satisfy that
		$\lambda_{\max}\le a_{\max}$ and $\lambda_{\min}\ge a_{\min}$.
		Thus
		$\kappa=\lambda_{\max}/ \lambda_{\min}\le a_{\max}/a_{\min}$.
	\end{proof}

	\subsection{Error estimates}\label{sec:error:estimates}			 
	In this subsection, we theoretically prove that the energy and $L_2$ norms of the errors  of our proposed method achieve first-order and second-order
	convergence rates in \cref{Thm:u:prime:L2} and \cref{Thm:u:L2}, respectively  for any coarse mesh size $H$.
To do so, we need the following auxiliary results in \cref{lemma:piecewise} and \cref{lemma:space}.

	\begin{lemma}\label{lemma:piecewise}
Let $p$ be any piecewise constant function on $\Omega$ given by
\[
p(x)=p_i\in \R, \mbox{ for all }  x\in[\tfrac{i}{2^n},\tfrac{i+1}{2^n})
\mbox{ and } i=0,\dots,2^n-1,
\]
where $n$ is a positive integer with $n\ge 2$. Then the function $p(x)$ can be uniquely expressed as
		\be\label{px} p(x)=c_0+c_{1}(\phi(2x-1))'+\sum_{j=1}^{n-1}\sum_{k=0}^{2^j-1}c_{j,k}(\psi(2^jx-k))',
		\ee
		where $(\phi(2x-1))'$ and $(\psi(2^jx-k))'$ are defined in \eqref{phipsi:shift} and \eqref{phipsi:shift:prime} respectively, and
		\[
		\begin{split}
		c_0=\frac{p_i}{2^n},\qquad c_1=	 \frac{1}{4}\int_{0}^1	 p(x)(\phi(2x-1))'dx,\qquad
		c_{j,k}=	 \frac{1}{2^{j+2}}\int_{0}^1	 p(x)(\psi(2^jx-k))'dx.
		\end{split}
		\]
	\end{lemma}
	
\begin{proof}
The claim follows immediately from the facts that $\left \{ (\psi(2^jx-k))' : j\in \N, 0\le k \le 2^j-1 \right \} \cup \left \{(\phi(2x-1))'\right \}$ $\cup\{\chi_{[0,1]} \}$ form an orthogonal basis in $L_2(\Omega)$ and $\int_{\Omega} p(x) (\psi(2^jx-k))' dx=0$ for all $j \ge n$ and $0\le k \le 2^j-1$.
\end{proof}

	In order to theoretically prove the convergence rate of the error in the energy norm of our proposed  derivative-orthogonal wavelet multiscale method in \cref{Thm:u:prime:L2}, we  define
	%
	%
	%

%
\begin{equation}\label{Bh:prime}
\mathcal{B}'_H:=\{\chi_{\Omega}\}\cup \{g' \; : \; g\in \mathcal{B}_H\}
\quad \mbox{and}\quad
\tilde{\mathcal{B}}'_H:=\{\chi_{\Omega}\}\cup \{g' \; : \; g\in \tilde{\mathcal{B}}_H\},
\end{equation}
where $\mathcal{B}_H$ and $\tilde{\mathcal{B}}_H$ are defined in
\eqref{Bh} and \eqref{Btildeh}, respectively, and $H=1/2^n$ for an integer $n\ge 2$.

	%
	%
	%
	%
	%
	%
	%
	\begin{lemma}\label{lemma:space}	 For any  $f_1,f_2\in  \text{span}(\mathcal{B}'_H)$  and $f_3\in  \text{span}(\tilde{\mathcal{B}}'_H)$,
		we have
		\be\label{f1f2f3}
		f_1f_2\in  \text{span}(\mathcal{B}'_H),\quad  f_1f_3\in \text{span}( \tilde{\mathcal{B}}'_H), \quad \text{and} \quad  \text{span}(\mathcal{B}'_H) \subset  \text{span}(\tilde{\mathcal{B}}'_H).
		\ee
	\end{lemma}
	\begin{proof}
		\eqref{f1f2f3} can be directly derived from \eqref{phipsi:shift}--\eqref{regular:support:2}, \eqref{Btildeh}--\eqref{supp:Si:3},  and \eqref{Bh:prime}.
	\end{proof}

With the aid of  \cref{lemma:piecewise}, \cref{lemma:space}, and \eqref{Bh:prime}, we prove that the first-order convergence rate of the error in the energy  norm of our proposed method can be achieved for any coarse mesh size $H$
in the following \cref{Thm:u:prime:L2}.
	 \begin{theorem}\label{Thm:u:prime:L2}
Define $a_{\max}:=\sup_{x\in \Omega} a(x)$ and $a_{\min}:=\inf_{x\in \Omega} a(x)$. Let $u$ be the exact solution of the model problem \eqref{Model:Problem} under the assumptions that $0<a_{\min}\le a\le a_{\max}<\infty$ and $f\in L_2(\Omega)$. Let $\tilde{u}_H$ be the numerical solution of our proposed derivative-orthogonal wavelet multiscale method through numerically solving \eqref{M:FEM} in the finite-dimensional space $\tilde{V}_H:=\mbox{span}(\tilde{\mathcal{B}}_H)$ with the basis $\tilde{\mathcal{B}}_H$ defined in \eqref{Btildeh}.
		Then
		\be	\label{error:estimate:1} \|\sqrt{a}(u'-(\tilde{u}_H)')\|_{L_2(\Omega)} \leq CH\quad \mbox{with}\quad
 C:=\frac{2}{\sqrt{a_{\min}}}\|f\|_{L_2(\Omega)}.
		\ee
		where $H$ is the coarse mesh size such that $H=2^{-n}$ for each integer $n\ge2$.
%
		%
	\end{theorem}  	
	\begin{proof}
		Recall that \eqref{week:form}  and \eqref{M:FEM} indicate that
the exact solution $u\in H^1_0(\Omega)$ and the numerical solution $\tilde{u}_H \in  \mbox{span}(\tilde{\mathcal{B}}_H)\subset H_0^1(\Omega)$ satisfy
		\[
		\begin{split}
			& \int_\Omega a  u' v'dx=\int_\Omega f vdx,\ \ \
			\ \quad \mbox{for all}\ v\in H_0^1(\Omega),\\
			& \int_\Omega a(x)  (\tilde{u}_H)' (\tilde{v}_H)'dx=\int_\Omega f \tilde{v}_Hdx,\ \ \
			\ \quad \mbox{for all}\ \tilde{v}_H\in\mbox{span}(\tilde{\mathcal{B}}_H) \subset H_0^1(\Omega).
		\end{split}
		\]
Take $v=\tilde{v}_H\in \text{span}( \tilde{\mathcal{B}}_H )$ in the first identity.
		Then we have the Galerkin orthogonality:
		\be\label{orthorgonal}
		\int_\Omega a(x)  (u-\tilde{u}_H)' (\tilde{v}_H)'dx=0,\ \ \
		\ \quad \mbox{for all}\ \tilde{v}_H\in \text{span}( \tilde{\mathcal{B}}_H ).
		\ee   	
		For any $\tilde{v}_H\in  \text{span}( \tilde{\mathcal{B}}_H )$, we also have
		\be
		\begin{split}
			\int_{\Omega} a(x)  \left(u'-(\tilde{v}_H)'\right)^2dx=&	 \int_{\Omega} a(x)  \left(u'-	 (\tilde{u}_H)'+(\tilde{u}_H)'-(\tilde{v}_H)'\right)^2dx,\\
			=&	\int_{\Omega} a(x)  \left(u'-	 (\tilde{u}_H)'\right)^2dx+2\int_{\Omega} a(x)  \left(u'-	 (\tilde{u}_H)'\right)\left((\tilde{u}_H)'-(\tilde{v}_H)'\right)dx\\
			&+ \int_{\Omega} a(x)  \left((\tilde{u}_H)'-(\tilde{v}_H)'\right)^2dx.
		\end{split}
		\ee
		Since $\tilde{u}_H-\tilde{v}_H\in \text{span}(\tilde{\mathcal{B}}_H)$, \eqref{orthorgonal} implies
		\be
		\begin{split}
			\int_{\Omega} a(x)  \left(u'-(\tilde{v}_H)'\right)^2dx
			&=	\int_{\Omega} a(x)  \left(u'-	 (\tilde{u}_H)'\right)^2dx+ \int_{\Omega} a(x) \left((\tilde{u}_H)'-(\tilde{v}_H)'\right)^2dx,\\
			&\ge \int_{\Omega} a(x)  \left(u'-	 (\tilde{u}_H)'\right)^2dx.
		\end{split}
		\ee	
		That is, we proved
		\be\label{optimal:choice}
		\int_{\Omega} a(x)  \left(u'-	 (\tilde{u}_H)'\right)^2dx \le 	 \int_{\Omega} a(x)  \left(u'-(\tilde{v}_H)'\right)^2dx, \quad \text{for any} \quad \tilde{v}_H \in \text{span}(\tilde{\mathcal{B}}_H).
		\ee
		On the other hand, we define
		\be\label{xi:proof}
		\xi_i=
			w_i\chi_{[\frac{i}{2^n},\ \frac{i+1}{2^n})}
		\qquad \text{and} \qquad 		 \eta_i=S_i',
		\ee
		where $i=0,\dots,2^n-1$, $w_i$ and $S_i'$ are defined in \eqref{wi}	 and \eqref{Si:prime} respectively.
		Let
		\be\label{xi:and:eta}
		\xi=\sum_{i=0}^{2^n-1}\xi_i, \qquad 	 \eta=\sum_{i=0}^{2^n-1}\eta_i.
		\ee
		Then we conclude from \eqref{Si:prime} that
		\be\label{sum:xi:eta}
		 \xi+\eta=\frac{1}{a}, \quad \text{for all } x\in \Omega.
		\ee
		Furthermore,   \eqref{Bh:prime}, \eqref{xi:proof}, \eqref{xi:and:eta},  \cref{lemma:piecewise,lemma:space} together imply that
		\be\label{xi:space}
		\xi \in \text{span}(\mathcal{B}'_H), \qquad \eta\in \text{span}(\tilde{\mathcal{B}}'_H).
		\ee
		 From the model problem \eqref{Model:Problem}, we have
		 \[
		 	-(au')'=f \Longrightarrow a(x)u'(x)=-\int_0^xf(t)dt+K,
		 \]
		 where $K$ is a constant to be determined later. Define
		 \be\label{function:F}
		 F(x):=-\int_0^xf(t)\ dt,\qquad x \in \Omega.
		 \ee
        Note that we are considering the 1D problem, i.e., $\Omega=(0,1)$. By our assumption $f\in L_2(\Omega)$, we must have $F\in H^1(\Omega) \subset C(\Omega)$ and $F(0)=0$.
        Using the definition of the function $F$,
         we can rewrite the identity $a(x)u'(x)=-\int_0^xf(t)dt+K$ as
		\be\label{equation:problem}
		\begin{split}
		  u'(x)=\frac{F(x)+K}{a(x)}.
		\end{split}
		\ee
Because $F\in H^1(\Omega) \subset C(\Omega)$, we have $au'\in C(\Omega)$. Furthermore, $u\in H_0^1(\Omega) \subset C(\Omega)$.
		We choose the constant $K$ to be
		\be\label{The:K}
	K:=	\frac{-\int_0^1 F(t)/a(t)dt}{\int_0^1 1/a(t)dt  }.
		\ee
		Consequently, we obtain from the above choice of the constant $K$ that
		\be\label{int:F:is0}
		\int_0^1 \frac{F(t)+K}{a(t)}dt=	 \int_0^1 \frac{F(t)}{a(t)}dt+K\int_0^1 \frac{1}{a(t)}dt=0.
		\ee
		Now, \eqref{equation:problem}--\eqref{int:F:is0} imply that the exact solution $u$ of \eqref{Model:Problem} can be written as
		\be\label{exact:u:express}
		u(x)=\int_0^x \frac{F(t)+K}{a(t)}dt,
		\ee
		with
		\[
		u(0)=\int_0^0 \frac{F(t)+K}{a(t)}dt=0, \quad \text{and} \quad u(1)=\int_0^1 \frac{F(t)+K}{a(t)}dt=0.
		\]
		Define a piecewise constant function $g$ on $\Omega$ by
		\be\label{function:g}
		g(x)=g_i, \quad \text{for} \quad x\in \left[\frac{i}{2^n},\ \frac{i+1}{2^n}\right),\quad i=0,\ldots, 2^n-1 \quad \text{with} \quad g_i:=\frac{\int_{\frac{i}{2^n}}^\frac{i+1}{2^n} \frac{F(x)+K}{a(x)}dx }{ \int_{\frac{i}{2^n}}^\frac{i+1}{2^n} \frac{1}{a(x)}dx}.
		\ee
		%
		%
		By the definition of the function $g$ in \eqref{function:g}, we have
		\be\label{seperate:form}
		\begin{split}
	\|F+K-g\|_{L_2(\Omega)}^2&=\int_{0}^1 (F+K-g)^2dx= \sum_{i=0}^{2^n-1}  \int_{\frac{i}{2^n}}^\frac{i+1}{2^n} (F+K-g_i)^2dx.
		\end{split}
		\ee
Note that both $K$ and $g_i$ are constant in $(\frac{i}{2^n},\frac{i+1}{2^n})$.
	We observe from the definition of $g_i$ in \eqref{function:g} that
	\[
\int_{\frac{i}{2^n}}^\frac{i+1}{2^n} \frac{F+K-g_i}{a}dx=\int_{\frac{i}{2^n}}^\frac{i+1}{2^n} \frac{F+K}{a}dx-g_i\int_{\frac{i}{2^n}}^\frac{i+1}{2^n} \frac{1}{a}dx=0.
	\]
	Since $F(x)+K-g_i$ is continuous on $(\frac{i}{2^n},\frac{i+1}{2^n})$ and $a(x)>0$ in $\Omega$, the above identity implies that there must exist $x_0\in (\frac{i}{2^n},\frac{i+1}{2^n})$ such that
	\be\label{zero:point}
	F(x_0)+K-g_i=0.
	\ee
	Hence, by $x_0\in (\frac{i}{2^n},\frac{i+1}{2^n})$, we have
	\be
	\int_{\frac{i}{2^n}}^\frac{i+1}{2^n} (F+K-g_i)^2dx=	 \int_{\frac{i}{2^n}}^{x_0} (F+K-g_i)^2dx+	 \int_{x_0}^\frac{i+1}{2^n} (F+K-g_i)^2dx.
	\ee
Note that $F\in H^1(\Omega) \subset C(\Omega)$.
	Using the integration by parts and \eqref{function:F}, we deduce from \eqref{zero:point} that
	\be\label{int:part}
	\begin{split}
	&\int_{\frac{i}{2^n}}^{x_0} (F+K-g_i)^2dx=\int_{\frac{i}{2^n}}^{x_0} (F+K-g_i)^2d(x-{i}/{2^n})\\
	 &=(F+K-g_i)^2\left(x-\frac{i}{2^n}\right)
\bigg|_{i/2^n}^{x_0}-2\int_{\frac{i}{2^n}}^{x_0}(x-{i}/{2^n})(F+K-g_i) (-f(x)) dx,\\
	 &=2\int_{\frac{i}{2^n}}^{x_0}(x-{i}/{2^n})(F+K-g_i) f dx\le2\int_{\frac{i}{2^n}}^{x_0}|x-{i}/{2^n}\|(F+K-g_i) f| dx,\\
	&\le 2\frac{1}{2^n}  \int_{\frac{i}{2^n}}^{x_0}|(F+K-g_i) f| dx.
	\end{split}
	\ee
	Recall that $H=\frac{1}{2^n}$. Using the  Cauchy-Schwarz inequality, \eqref{int:part} leads to
	\be
	\left(\int_{\frac{i}{2^n}}^{x_0} (F+K-g_i)^2dx\right)^2 \le 4H^2 	 \left(\int_{\frac{i}{2^n}}^{x_0} (F+K-g_i)^2dx\right) \left( \int_{\frac{i}{2^n}}^{x_0} f^2dx\right),
	\ee
	from which we deduce that
$\int_{\frac{i}{2^n}}^{x_0} (F+K-g_i)^2dx\le 4H^2   \int_{\frac{i}{2^n}}^{x_0} f^2dx$.
	Similarly, we have
	$\int_{x_0}^{\frac{i+1}{2^n}} (F+K-g_i)^2dx\le 4H^2  \int_{x_0}^{\frac{i+1}{2^n}} f^2dx$.
	Consequently, we conclude that
\[
\int_{\frac{i}{2^n}}^{\frac{i+1}{2^n}} (F+K-g_i)^2dx\le
4H^2   \int_{\frac{i}{2^n}}^{x_0} f^2dx+
4H^2  \int_{x_0}^{\frac{i+1}{2^n}} f^2dx
=4H^2   \int_{\frac{i}{2^n}}^{\frac{i+1}{2^n}} f^2dx.
\]
	According to \eqref{seperate:form} and the above inequality, we conclude that
	\be\label{f:and:g:H2:final}
		\|F+K-g\|_{L_2(\Omega)}^2\le 4H^2 \|f\|_{L_2(\Omega)}^2.
	\ee	
		Furthermore, \eqref{Bh:prime}, \eqref{function:g}, \cref{lemma:piecewise,lemma:space} together imply that
		\be\label{g:space}
		g \in \text{span}(\mathcal{B}'_H),
		\ee
		and by \eqref{function:g}, we have
		\[
		\int_{0}^1 \frac{g(x)}{a(x)} dx=\sum_{i=0}^{2^n-1} g_i \int_{\frac{i}{2^n}}^\frac{i+1}{2^n} \frac{1}{a(x)}dx =\sum_{i=0}^{2^n-1} \int_{\frac{i}{2^n}}^\frac{i+1}{2^n} \frac{F(x)+K}{a(x)}dx=\int_{0}^1 \frac{F(x)+K}{a(x)}dx.
		\]
		By \eqref{int:F:is0}, we conclude from the above identity that
		\be\label{g:o:a:0}
		\int_{0}^1 \frac{g(x)}{a(x)} dx=0.
		\ee
		Note that \eqref{sum:xi:eta} holds. Define
		\be\label{tildewhx}
		w(x):=\int_{0}^x (\xi(t)+\eta(t))g(t)dt\xlongequal{\text{by } \eqref{sum:xi:eta}} \int_{0}^x \frac{g(t)}{a(t)} dt.	
		\ee
		Based on \eqref{Bh:prime},	 \eqref{xi:space}, \eqref{g:space},
		 \cref{lemma:piecewise,lemma:space}, it follows that
		\[
		\frac{d}{dx} w(x)= \frac{g(x)}{a(x)}\in \text{span}(\tilde{\mathcal{B}}'_H).
		\]
		Using the definition of $\tilde{\mathcal{B}}_H'$ in \eqref{Bh:prime}, we conclude from the above identity that
		\be\label{wH:space}
		w(x)\in \text{span}(\tilde{\mathcal{B}}_H \cup \{x+b,1\}),
		\ee
		where  $b$ is a constant.
		So \eqref{Btildeh:2}  implies that
		\be\label{tilde:w:H:linear}
		 w(x)=\sum_{i=1}^{M}d_i\varphi_i(x)+d_{M+1}(x+b)+d_{M+2},
		\ee
		where each $d_i$ is a constant.
		By \eqref{g:o:a:0} and \eqref{tildewhx},
		\be\label{bd:is:0}
		w(0)= w(1)=0.
		\ee
		So, \eqref{Btildeh:H01}, \eqref{tilde:w:H:linear} and \eqref{bd:is:0} yield
		\[
		d_{M+1}(0+b)+d_{M+2}=0,\quad d_{M+1}(1+b)+d_{M+2}=0 \Longrightarrow d_{M+1}=d_{M+2}=0.
		\]
		Consequently, we conclude from \eqref{tilde:w:H:linear} that
$w(x)=\sum_{i=1}^{M}d_i\varphi_i(x)$, that is,
		\be\label{wH:space:2}
	w\in \text{span}(\tilde{\mathcal{B}}_H)\subset H^1_0(\Omega).
	\ee
		Using \eqref{equation:problem} and \eqref{tildewhx}, due to $\frac{1}{a(x)}\le \frac{1}{a_{\min}}$ for all $x\in \Omega$,
we have
		\be\label{error:order:2:proof}
		\begin{split}
			\int_{\Omega} a(x) \left(u'- w'\right)^2 dx&=	\int_{\Omega}  a(x) \left(\frac{F+K}{a}-\frac{g}{a}\right)^2 dx=\int_{\Omega} \frac{1}{a(x)} \left(F+K-g\right)^2 dx,\\
&\le  \frac{1}{a_{\min}} \int_{\Omega} \left(F+K-g\right)^2 dx\le
\frac{4}{a_{\min}} H^2 \|f\|_{L_2(\Omega)}^2,
		\end{split}
		\ee
where we used \eqref{f:and:g:H2:final} in the last inequality.
		From  \eqref{optimal:choice} and \eqref{wH:space:2}, we deduce from the fact $w\in \mbox{span}(\tilde{\mathcal{B}}_H)$
that
		\be
		\int_{\Omega} a(x)  \left(u'-	 (\tilde{u}_H)'\right)^2dx \le	 \int_{\Omega} a(x)\left(u'-(w )'\right)^2 dx \le  C^2H^2,
		\ee
		where we used \eqref{error:order:2:proof} in the last inequality and $C:=\frac{2}{\sqrt{a_{\min}}} \|f\|_{L_2(\Omega)}$.
	\end{proof}
Utilizing \eqref{error:estimate:1} in \cref{Thm:u:prime:L2} and the well-known Aubin--Nitsche argument,  we now theoretically prove that the $L_2$ norm of the error of our proposed method achieves the second-order
convergence rate  for any coarse mesh size $H$ in the following \cref{Thm:u:L2}.
	\begin{theorem}\label{Thm:u:L2}
Define $a_{\max}:=\sup_{x\in \Omega} a(x)$ and $a_{\min}:=\inf_{x\in \Omega} a(x)$. Let $u$ be the exact solution of the model problem \eqref{Model:Problem} under the assumptions that $0<a_{\min}\le a\le a_{\max}<\infty$ and $f\in L_2(\Omega)$. Let $\tilde{u}_H$ be the numerical solution of our proposed derivative-orthogonal wavelet multiscale method through numerically solving \eqref{M:FEM} in the finite-dimensional space $\tilde{V}_H:=\mbox{span}(\tilde{\mathcal{B}}_H)$ with the basis $\tilde{\mathcal{B}}_H$ defined in \eqref{Btildeh}.
Then
	\be	
	\|u-\tilde{u}_H\|_{L_2(\Omega)} \leq CH^2\quad \mbox{with}\quad
C:=\frac{4}{a_{\min}}\|f\|_{L_2(\Omega)},
	\ee
	where $H$ is the coarse mesh size such that $H=2^{-n}$ for each integer $n\ge2$.
\end{theorem}  	
\begin{proof}
We use the Aubin--Nitsche argument.
Let $z\in H^1_0(\Omega)$ be the solution to
\be\label{equation:with:z}
	-\nabla \cdot (a \nabla  z)=u-\tilde{u}_H,
\ee
with the boundary conditions $z(0)=z(1)=0$.
Then
\[
\|u-\tilde{u}_H\|^2_{L_2(\Omega)}=\int_{\Omega} (u-\tilde{u}_H)^2 dx=\int_{\Omega} 	 -\nabla \cdot (a \nabla  z) (u-\tilde{u}_H) dx
=\int_{\Omega} a(x) z'(u-\tilde{u}_H)' dx.
\]
Let $\tilde{z}_H\in \mbox{span}(\tilde{\mathcal{B}}_H)\subseteq H^1_0(\Omega)$ be the unique solution to $\int_\Omega a\nabla \tilde{z}_H\cdot \nabla \tilde{v}_H dx=\int_\Omega (u-\tilde{u}_H) \tilde{v}_Hdx$ for all $\tilde{v}_H\in \mbox{span}(\tilde{\mathcal{B}}_H)$.
Applying the Galerkin orthogonality in \eqref{orthorgonal} with $\tilde{v}_H=\tilde{z}_H \in \mbox{span}(\tilde{\mathcal{B}}_H)$ and using the Cauchy-Schwarz inequality, we derive from the above identity  that
\be\label{L2:norm:eq1}
\begin{split} \|&u-\tilde{u}_H\|^2_{L_2(\Omega)}
=\int_{\Omega} a(x) z'(u-\tilde{u}_H)' dx=\int_{\Omega} a(x) (z-\tilde{z}_H)'(u-\tilde{u}_H)' dx,\\
	& \le \|\sqrt{a} (z-\tilde{z}_H)'\|_{L_2(\Omega)} \|\sqrt{a} (u-\tilde{u}_H)'\|_{L_2(\Omega)}
	\le  \|\sqrt{a} (z-\tilde{z}_H)'\|_{L_2(\Omega)} \frac{2}{\sqrt{a_{\min}}}\|f\|_{L_2(\Omega)}H,
\end{split}
\ee
where we used the inequality \eqref{error:estimate:1} of \cref{Thm:u:prime:L2} in the last inequality. Note that $u-\tilde{u}_H\in L_2(\Omega)$.
Applying \eqref{error:estimate:1} of \cref{Thm:u:prime:L2} and \eqref{equation:with:z}, we must have
\be\label{L2:norm:eq2}
 \|\sqrt{a} (z-\tilde{z}_H)'\|_{L_2(\Omega)}\le \frac{2}{\sqrt{a_{\min}}}\|u-\tilde{u}_H\|_{L_2(\Omega)}H.
\ee
Now, we conclude from \eqref{L2:norm:eq1} and \eqref{L2:norm:eq2} that
\[
	\|u-\tilde{u}_H\|^2_{L_2(\Omega)} \le \frac{2}{\sqrt{a_{\min}}}\|u-\tilde{u}_H\|_{L_2(\Omega)}H \frac{2}{\sqrt{a_{\min}}}\|f\|_{L_2(\Omega)}H
=\frac{4}{a_{\min}} \|f\|_{L_2(\Omega)} \|u-\tilde{u}_H\|_{L_2(\Omega)} H^2,
\]
from which we obtain
$\|u-\tilde{u}_H\|_{L_2(\Omega)} \le \frac{4}{a_{\min}} \|f\|_{L_2(\Omega)}H^2$.
\end{proof}
	Finally, we introduce the following \cref{same:space}	to show the interpolation property that
$\tilde{u}_H\left(\frac{i}{2^n}\right)=u\left(\frac{i}{2^n}\right)$, for $i=0,1,\dots,2^n$, where $u$ is the exact solution of  \eqref{Model:Problem}, $\tilde{u}_H$ is the numerical solution of our proposed derivative-orthogonal wavelet multiscale method, and $H$ is the coarse mesh size such that $H=2^{-n}$ for each integer $n\ge2$.
	\begin{lemma}\label{same:space}		 Define the continuous piecewise linear interpolations
		\be\label{alpha:i}
		\ell_i:=\begin{cases}
		2^n(	x-\frac{i}{2^n}), &\quad x\in[\frac{i}{2^n},\frac{i+1}{2^n}),\\
		2^n(	-x+\frac{i+2}{2^n}), &\quad x\in[\frac{i+1}{2^n},\frac{i+2}{2^n}),\\
			0, &\quad \text{else},
		\end{cases}	
		\ee
		where $i=0,1,\dots,2^n-2$  and  $n$ is an integer with $n\ge 2$. Then
		\be\label{same:space:formula}
	\text{span}\{\ell_i : i=0,1,\dots,2^n-2\}	 =V_H:=\text{span}(\mathcal{B}_H),
		\ee
		where $\mathcal{B}_H$ is defined in \eqref{Bh}.
	\end{lemma}
\begin{proof}	
\eqref{same:space:formula} is a direct consequence	due to the way we construct the wavelet basis $\mathcal{B}_H$ of $V_H$ in \eqref{Bh}. See \cite{Hanbook,HanMichelle2019} for details on wavelets and their refinable structure.
\end{proof}

	Now,  by \cref{same:space}, we can prove the following result on the interpolation property.
	\begin{theorem}\label{Thm:u:infty}
Define $a_{\max}:=\sup_{x\in \Omega} a(x)$ and $a_{\min}:=\inf_{x\in \Omega} a(x)$. Let $u$ be the exact solution of the model problem \eqref{Model:Problem} under the assumptions that $0<a_{\min}\le a\le a_{\max}<\infty$ and $f\in L_2(\Omega)$. Let $\tilde{u}_H$ be the numerical solution of our proposed derivative-orthogonal wavelet multiscale method through numerically solving \eqref{M:FEM} in the finite-dimensional space $\tilde{V}_H:=\mbox{span}(\tilde{\mathcal{B}}_H)$ with the basis $\tilde{\mathcal{B}}_H$ defined in \eqref{Btildeh}.
Then
    \be\label{interpolation:1} \tilde{u}_H\left(\frac{i}{2^n}\right)=u\left(\frac{i}{2^n}\right), \qquad \text{for }\quad i=0,1,\dots,2^n,
	\ee
	where $H$ is the coarse mesh size such that $H=2^{-n}$ for each integer $n\ge2$.
\end{theorem}  	
\begin{proof}
We define
\be\label{beta}
\beta_i:=\frac{S_i}{w_i}-\frac{S_{i+1}}{w_{i+1}}, \quad \text{for } \quad i=0,1,\dots,2^n-2,
\ee
where $S_i$ is defined in \eqref{Special:Si} and each constant $w_i$ is defined in \eqref{wi}. By \eqref{wi},  each constant $w_i>0$.
\eqref{Special:Si} and \eqref{wi} also imply that 	
\be\label{special:a:constant}
S_i=0, \quad \text{and} \quad w_i=\frac{1}{\ka_i}, \qquad  \text{if } a  \text{ is equal to a positive constant $\ka_i$ in } \left(\frac{i}{2^n},\frac{i+1}{2^n}\right).
\ee
From \eqref{Si:prime} and \eqref{beta}, it follows that:
	\be\label{beta:prime}
	(\beta_i)'=\begin{cases}
		\displaystyle\frac{1}{w_ia}-1, &\quad x\in[\frac{i}{2^n},\frac{i+1}{2^n}),\\
		 \displaystyle\frac{-1}{w_{i+1}a}+1, &\quad x\in[\frac{i+1}{2^n},\frac{i+2}{2^n}),\\
		0, &\quad \text{else},
	\end{cases}	
	\ee
	where  $i=0,1,\dots,2^n-2$.
	Clearly, if $a$ is equal to a positive constant $\ka_i$ in $(\frac{i}{2^n},\frac{i+1}{2^n})$, then \eqref{special:a:constant} yields
		\be\label{beta:prime:special}
\beta_i=-\frac{S_{i+1}}{w_{i+1}}, \qquad
(\beta_i)'=\begin{cases}
	\displaystyle\frac{-1}{w_{i+1}a}+1, &\quad x\in[\frac{i+1}{2^n},\frac{i+2}{2^n}),\\
	0, &\quad \text{else}.
\end{cases}
	\ee
	By the definition of $ \tilde{\mathcal{B}}_H$ in \eqref{Btildeh} and  the definition of $ \beta_i$ in \eqref{beta}, we identify that
		\be\label{beta:space}
	\beta_i\in \text{span}( \tilde{\mathcal{B}}_H), \quad \text{for } \quad i=0,1,\dots,2^n-2.
	\ee
Define $\ell_i$ as in \eqref{alpha:i} with $i=0,1,\dots,2^n-2$. It is clear that
	\be\label{alpha:prime}
	(\ell_i)'=\begin{cases}
		2^n, &\quad x\in[\frac{i}{2^n},\frac{i+1}{2^n}),\\
		-2^n, &\quad x\in[\frac{i+1}{2^n},\frac{i+2}{2^n}),\\
		0, &\quad \text{else}.
	\end{cases}	
	\ee
By \eqref{same:space:formula} in \cref{same:space}, definitions of $\mathcal{B}_H$  and $ \tilde{\mathcal{B}}_H$ in \eqref{Bh} and \eqref{Btildeh}, we observe
	\be\label{alpha:space}
\ell_i\in \text{span}(\mathcal{B}_H ) \subset \text{span}( \tilde{\mathcal{B}}_H), \quad \text{for } \quad i=0,1,\dots,2^n-2.
\ee
Define the error
\be\label{def:E}
E:=u-\tilde{u}_H.
\ee
Then,  \eqref{orthorgonal} implies
	\be\label{orthor}
\int_\Omega a(x)  E' (\tilde{v}_H)'dx=\int_\Omega a(x)  (u-\tilde{u}_H)' (\tilde{v}_H)'dx=0,\ \ \
\ \mbox{for all}\ \tilde{v}_H\in \text{span}( \tilde{\mathcal{B}}_H).
\ee
We deduce from \eqref{alpha:prime}, \eqref{alpha:space}, and \eqref{orthor} that
\be\label{seperate:L:infty:alpha}
\begin{split}
	0&=\frac{1}{2^n} \int_\Omega a(x)  E'(x) (\ell_i(x))'dx=\int_{\frac{i}{2^n}}^{\frac{i+1}{2^n}} a(x)E'(x) dx+\int_{\frac{i+1}{2^n}}^{\frac{i+2}{2^n}} a(x)E'(x) (-1) dx.
\end{split}
\ee
We claim that
	\be\label{seperate:L:infty:together}
\begin{split}
	0= \frac{1}{w_{i}}\int_{\frac{i}{2^n}}^{\frac{i+1}{2^n}} E' dx-\frac{1}{w_{i+1}}\int_{\frac{i+1}{2^n}}^{\frac{i+2}{2^n}} E' dx,
\end{split}
\ee
%
where each constant $w_i$ is defined in \eqref{wi}. We prove \eqref{seperate:L:infty:together} by following \eqref{seperate:L:infty:beta}--\eqref{seperate:L:infty:alpha:3} in the following 3 cases.

Case 1. If  $a$ is not equal to a positive constant $\ka_i$ in each $(\frac{i}{2^n},\frac{i+1}{2^n})$ with $i=0,1,\dots,2^n-2$.
We conclude from \eqref{beta:prime}, \eqref{beta:space}, and \eqref{orthor} that
	\be\label{seperate:L:infty:beta}
\begin{split}
	0&= \int_\Omega a  E' (\beta_i)'dx=\int_{\frac{i}{2^n}}^{\frac{i+1}{2^n}} a E' \left(\frac{1}{w_ia}-1\right) dx+\int_{\frac{i+1}{2^n}}^{\frac{i+2}{2^n}} a E' \left(\frac{-1}{w_{i+1}a}+1\right) dx,\\
	& =\frac{1}{w_i}\int_{\frac{i}{2^n}}^{\frac{i+1}{2^n}} E' dx-\frac{1}{w_{i+1}}\int_{\frac{i+1}{2^n}}^{\frac{i+2}{2^n}} E' dx-\int_{\frac{i}{2^n}}^{\frac{i+1}{2^n}} aE' dx+\int_{\frac{i+1}{2^n}}^{\frac{i+2}{2^n}} aE' dx.
\end{split}
\ee
Then \eqref{seperate:L:infty:beta} and \eqref{seperate:L:infty:alpha}  lead to \eqref{seperate:L:infty:together}.

Case 2. If $a=\ka_i$ for $x\in (\frac{i}{2^n},\frac{i+1}{2^n})$, then \eqref{beta:prime:special} implies
	\be\label{seperate:L:infty:beta:2}
\begin{split}
	0&= \int_\Omega a  E' (\beta_i)'dx=\int_{\frac{i+1}{2^n}}^{\frac{i+2}{2^n}} a E' \left(\frac{-1}{w_{i+1}a}+1\right) dx =-\frac{1}{w_{i+1}}\int_{\frac{i+1}{2^n}}^{\frac{i+2}{2^n}} E' dx+\int_{\frac{i+1}{2^n}}^{\frac{i+2}{2^n}} aE' dx.
\end{split}
\ee
Now,  \eqref{seperate:L:infty:alpha} indicates
\be\label{seperate:L:infty:alpha:2}
\begin{split}
	0&=\int_{\frac{i}{2^n}}^{\frac{i+1}{2^n}} aE' dx-\int_{\frac{i+1}{2^n}}^{\frac{i+2}{2^n}} aE'  dx=\ka_i\int_{\frac{i}{2^n}}^{\frac{i+1}{2^n}} E' dx-\int_{\frac{i+1}{2^n}}^{\frac{i+2}{2^n}} aE'  dx.
\end{split}
\ee
\eqref{seperate:L:infty:beta:2} and \eqref{seperate:L:infty:alpha:2} together with $\ka_i=\frac{1}{w_i}$ in \eqref{special:a:constant} leads to \eqref{seperate:L:infty:together}.

Case 3. If $a$ is a positive constant function $\ka$ for $x\in (0,1)$, then \eqref{wi} suggests
\[
w_{0}=w_{1}=\dots=w_{2^n-1}=\frac{1}{\ka}.
\]
By  \eqref{seperate:L:infty:alpha},
\be\label{seperate:L:infty:alpha:3}
\begin{split}
	 0=\ka\int_{\frac{i}{2^n}}^{\frac{i+1}{2^n}} E' dx-\ka \int_{\frac{i+1}{2^n}}^{\frac{i+2}{2^n}} E'  dx=\frac{1}{w_i}\int_{\frac{i}{2^n}}^{\frac{i+1}{2^n}} E' dx-\frac{1}{w_{i+1}} \int_{\frac{i+1}{2^n}}^{\frac{i+2}{2^n}} E'  dx.
\end{split}
\ee
Thus, \eqref{seperate:L:infty:together} is true for any constant $a>0$ on $\Omega$.

Since $E=u-\tilde{u}_H$ is continuous in $\Omega$, we have
\[
\int_{\frac{i}{2^n}}^{\frac{i+1}{2^n}} E' dx=E\left(\frac{i+1}{2^n}\right)-E\left(\frac{i}{2^n}\right), \quad \text{and} \quad \int_{\frac{i+1}{2^n}}^{\frac{i+2}{2^n}} E' dx=E\left(\frac{i+2}{2^n}\right)-E\left(\frac{i+1}{2^n}\right).
\]
Now, \eqref{seperate:L:infty:together} yields
\be\label{zero:expression}
-\frac{1}{w_{i}}E\left(\frac{i}{2^n}\right)
+\left(\frac{1}{w_{i}}+\frac{1}{w_{i+1}}
\right) E\left(\frac{i+1}{2^n}\right)
-\frac{1}{w_{i+1}} E\left(\frac{i+2}{2^n}\right)=0,
\ee
for  $i=0,1,\dots,2^n-2$.
We define 
\be\label{vector:E}
\vec{E}:=\left(E(0),E\left(\frac{1}{2^n}\right),E\left(\frac{2}{2^n}\right),
\dots,E\left(\frac{2^n-1}{2^n}\right),E(1)\right)^T.
\ee
By the boundary condition $u(0)=u(1)=0$, it is obvious that $E(0)=E(1)=0$. So \eqref{zero:expression} and \eqref{vector:E} indicate
 \be\label{AE:is:0}
 A\vec{E}=\vec{0},
 \ee
 where
	\[
A=	\begin{pmatrix}
	1 &0 &0 & 0  & 0 & 0 & \cdots  & 0 & 0 \\
	-\frac{1}{w_{0}}  &	 \frac{1}{w_{0}}+\frac{1}{w_{1}} & -\frac{1}{w_{1}} & 0  & 0 & 0 & \cdots  & 0& 0 \\
	0  &	-\frac{1}{w_{1}} & \frac{1}{w_{1}}+\frac{1}{w_{2}} & -\frac{1}{w_{2}}  & 0 & 0 & \cdots  & 0 & 0\\
	0  &	0 & -\frac{1}{w_{2}} & \frac{1}{w_{2}}+\frac{1}{w_{3}} & -\frac{1}{w_{3}}  & 0& \cdots  & 0  & 0\\
	0  &	0 & 0& -\frac{1}{w_{3}} & \frac{1}{w_{3}}+\frac{1}{w_{4}} & -\frac{1}{w_{4}}& \ddots  & 0   & 0\\
	\vdots  &	\vdots & 	\vdots & 	 \ddots & \ddots & \ddots& \ddots  & 0    & 0\\
	0  &    0&  0&\cdots & 0& -\frac{1}{w_{2^n-3}} & \frac{1}{w_{2^n-3}}+\frac{1}{w_{2^n-2}} & -\frac{1}{w_{2^n-2}}    & 0\\
	0  &	0 & 0 & 0 & \cdots & 0& -\frac{1}{w_{2^n-2}}  & \frac{1}{w_{2^n-2}}+\frac{1}{w_{2^n-1}}    &-\frac{1}{w_{2^n-1}} \\
	0 &0 &0 & 0  & 0 & \cdots & 0  & 0 & 1 \\
	\end{pmatrix}.
	\]
	Since the above $A$ is a  $(2^n+1) \times (2^n+1)$ M-matrix, we can say that $A$ is invertiable. Now \eqref{AE:is:0} leads to $\vec{E}=\vec{0}$. From \eqref{def:E} and \eqref{vector:E}, we observe that
	\be\label{interpolation:1:proof}
	 \tilde{u}_H\left(\frac{i}{2^n}\right)=u\left(\frac{i}{2^n}\right), \qquad \text{where }\quad i=0,1,\dots,2^n, \quad \text{and} \quad H=\frac{1}{2^n}.
	\ee
	Thus, \eqref{interpolation:1} is proved.
\end{proof}	
\section{Numerical experiments}
\label{sec:numerical}
Recall that $\Omega=(0,1)$, $H=\frac{1}{2^n}$ for every integer $n\ge 2$,  $u$ is the exact solution of the model problem \eqref{Model:Problem},  $\tilde{u}_H$ is the numerical solution of our proposed derivative-orthogonal wavelet multiscale method through solving \eqref{M:FEM} by using the basis $\tilde{\mathcal{B}}_H$ in \eqref{Btildeh}.
We shall measure errors in $l_2$ and $l_\infty$ norms on an extremely fine mesh $\frac{1}{N} \Z \cap \overline{\Omega}$ with very large integers $N$. That is,
we define
\be\label{fine:xi}
x_i=i/N, \qquad i=0,\dots, N, \qquad \text{for very large integers } N.
\ee
For all our numerical examples, we shall use either $N=2^{14}$ or $2^{15}$.
To verify the convergence rate,  we define the following $l_2$ and $l_\infty$ norms if  the exact solution $u$ is available
\be\label{Error:1}
\begin{split}
& \frac{\|\tilde{u}_H-u\|_2}{\|u\|_2}
:=\sqrt{\frac{\sum_{i=0}^N \left|\tilde{u}_H(x_i)-u(x_i)\right|^2}{\sum_{i=0}^N \left|u(x_i)\right|^2}},  \qquad \frac{\|\tilde{u}'_H-u'\|_2}{\|u'\|_2}
:=\sqrt{\frac{\sum_{i=0}^N \left|\tilde{u}'_H(x_i)-u'(x_i)\right|^2}{\sum_{i=0}^N \left|u'(x_i)\right|^2}},\\
& \frac{\|a\tilde{u}'_H-au'\|_2}{\|au'\|_2}
:=\sqrt{\frac{\sum_{i=0}^N \left|a(x_i)\tilde{u}'_H(x_i)-a(x_i)u'(x_i)\right|^2}{\sum_{i=0}^N \left|a(x_i)u'(x_i)\right|^2}},
\end{split}
\ee
for measuring the errors in $l_2$ norms,
and
\be\label{Error:2}
\begin{split}
& \|\tilde{u}_H-u\|_\infty
:=\max_{0\le i\le N} \left|\tilde{u}_H(x_i)-u(x_i)\right|, \qquad \|\tilde{u}'_H-u'\|_\infty
:=\max_{0\le i\le N} \left|\tilde{u}_H'(x_i)-u'(x_i)\right|,\\
& \|a\tilde{u}'_H-au'\|_\infty
:=\max_{0\le i\le N} \left|a(x_i)\tilde{u}_H'(x_i)-a(x_i)u'(x_i)\right|,
\end{split}
\ee
for measuring the errors in $l_\infty$ norms.

If the exact solution $u$ is not available, we  replace $u$ by $\tilde{u}_{H/2}$ in
\eqref{Error:1}--\eqref{Error:2} to define following  $l_2$ and $l_\infty$ norms
\[
\begin{split}
& \frac{\|\tilde{u}_H-\tilde{u}_{H/2}\|_2}{\|\tilde{u}_{H/2}\|_2},  \qquad \frac{\|\tilde{u}'_H-\tilde{u}'_{H/2}\|_2}{\|\tilde{u}'_{H/2}\|_2}, \qquad \frac{\|a\tilde{u}'_H-a\tilde{u}'_{H/2}\|_2}{\|a\tilde{u}'_{H/2}\|_2}, \qquad l_2 \text{ norms},\\
& \|\tilde{u}_H-\tilde{u}_{H/2}\|_\infty, \qquad \|\tilde{u}'_H-\tilde{u}'_{H/2}\|_\infty, \qquad \|a\tilde{u}'_H-a\tilde{u}'_{H/2}\|_\infty, \qquad l_\infty \text{ norms}.
\end{split}
\]

To highlight benefits of our proposed  multiscale method \eqref{M:FEM}, in \cref{Special:ex1,Special:ex2,Special:ex2:infty}, we shall compare the $l_2$ and $l_\infty$ norms of errors of the above
numerical solution $\tilde{u}_H$ with these of the linear finite element solution $u_{H/2}\in 	 \text{span}\{\ell_i : i=0,\dots,2^{n+1}-2\}$ satisfying
\begin{equation}\label{uHhalf}
\int_{\Omega} a\nabla u_{H/2} \cdot \nabla v=\int_{\Omega} f v,\quad \mbox{for all } v\in \{\ell_i : i=0,\dots,2^{n+1}-2\},
\end{equation}
where $\{\ell_i : i=0,\dots,2^{n+1}-2\}$ is obtained by replacing $n$ with $n+1$ in \eqref{alpha:i}.
Precisely, we present the errors in \cref{table3:example1} for  \cref{Special:ex1},  \cref{table3:example2} for  \cref{Special:ex2}, and \cref{table3:example2:infty} for  \cref{Special:ex2:infty} by replacing
the above numerical solution $\tilde{u}_H$ with the solution $u_{H/2}$ in \eqref{Error:1} and \eqref{Error:2}.
By \eqref{same:space:formula}, $\text{span}\{\ell_i : i=0,1,\dots,2^{n+1}-2\}=\text{span}(\mathcal{B}_{H/2})$.
Due to $H=2^{-n}$ and the refinability structure of $\mathcal{B}_{H}$, we observe (\cite{Hanbook,HanMichelle2019}) that $\mathcal{B}_{H/2} \subset \tilde{\mathcal{B}}_{H/2}$, while the three bases $\{\ell_i : i=0,1,\dots,2^{n+1}-2\}$, $\mathcal{B}_{H/2}$ and $\tilde{\mathcal{B}}_H$ have the same cardinality $2^{n+1}-1$.
%
%
%
%
%
In \cref{table1:example1}--\cref{table2:example6} of \cref{Special:ex1}--\cref{Special:ex6},  $\kappa$ is the condition number of the  stiffness matrix of our proposed derivative-orthogonal wavelet multiscale method or the standard finite element method (FEM). We test 6 numerical examples with $\Omega=(0,1)$ in this section.

	%
	%
%
\subsection{The exact solution $u$ is available}
\begin{example}\label{Special:ex1}
	\normalfont
	Consider $f(x)=1000x$ and
	a continuous diffusion coefficient $a(x)=\frac{1}{1.05+\sin(2^9\pi x)}$ having the high-frequency oscillation.
	Then the exact solution $u$ of \eqref{Model:Problem} can be expressed as
	\[
	u(x)=1000\left( - \frac{1.05x^3}{6}+ \frac{\ca_1\cos(2^9\pi x)}{2^9\pi} + \frac{\cos(2^9\pi x)}{2^{10}\pi}x^2 -\frac{\cos(2^9\pi x)}{(2^9\pi)^3}- \frac{\sin(2^9\pi x)}{2^{9}\pi}x -1.05\ca_1 x+\ca_2\right),
	\]
	where
	\[
	\begin{split}
&\ca_1 = \frac{1.05(2^{9}\pi)^3 - 3\cos(2^{9}\pi)(2^{9}\pi)^2 + 6\sin(2^{9}\pi)(2^{9}\pi) + 6\cos(2^{9}\pi) - 6}{6(2^{9}\pi)^2(-1.05(2^{9}\pi) + \cos(2^{9}\pi) - 1)},\\
&\ca_2 = \frac{-1.05(2^{9}\pi)^2 + 3\cos(2^{9}\pi)(2^{9}\pi) - 6\sin(2^{9}\pi) - 6.3} {6(2^{9}\pi)^2(-1.05(2^{9}\pi) + \cos(2^{9}\pi) - 1)}.
	\end{split}
	\]
The results on numerical solutions $\tilde{u}_H$ and $u_{H/2}$ are presented in \cref{table1:example1,table2:example1,table3:example1,fig:exam1}.
\end{example}

\begin{table}[htbp]
	\caption{Performance of the $l_2$ norm of the errors in \cref{Special:ex1} of our proposed derivative-orthogonal wavelet multiscale method  on the uniform Cartesian mesh $H$. Note that $N=2^{14}$ in \eqref{fine:xi} for estimating the $l_2$ norm. }
	\centering
		\renewcommand\arraystretch{0.7}
	\setlength{\tabcolsep}{3mm}{
		 \begin{tabular}{c|c|c|c|c|c|c|c|c}
			\hline
			$H$
			&   $\frac{\|\tilde{u}_H-u\|_2}{\|u\|_2}$
			& order &  $\frac{\|\tilde{u}'_H-u'\|_2}{\|u'\|_2}$ & order &   $\frac{\|a\tilde{u}'_H-au'\|_2}{\|au'\|_2}$
			&order & $\kappa$ & $\frac{a_{\max}}{a_{\min}}$ \\
			\hline	
$1/2$   &2.7782E-01   &   &5.4501E-01   &   &5.4490E-01   &   &11.64   &41\\
$1/2^2$   &7.1084E-02   &1.97   &2.7783E-01   &0.97   &2.7779E-01   &0.97   &11.64   &41\\
$1/2^3$   &1.7870E-02   &1.99   &1.3955E-01   &0.99   &1.3957E-01   &0.99   &11.64   &41\\
$1/2^4$   &4.4716E-03   &2.00   &6.9793E-02   &1.00   &6.9889E-02   &1.00   &11.64   &41\\
$1/2^5$   &1.1162E-03   &2.00   &3.4779E-02   &1.00   &3.5001E-02   &1.00   &11.64   &41\\
$1/2^6$   &2.7680E-04   &2.01   &1.7133E-02   &1.02   &1.7593E-02   &0.99   &11.64   &41\\			 
			\hline
	\end{tabular}}
	\label{table1:example1}
\end{table}	

\begin{table}[htbp]
	\caption{Performance of the $l_\infty$ norm of the error in \cref{Special:ex1} of our proposed derivative-orthogonal wavelet multiscale method on the uniform Cartesian mesh $H$. Note that $N=2^{14}$ in \eqref{fine:xi} for estimating the $l_\infty$ norm.  }
	\centering
		\renewcommand\arraystretch{0.7}
	\setlength{\tabcolsep}{2mm}{
		 \begin{tabular}{c|c|c|c|c|c|c|c|c}
			\hline
			$H$
			&   $\|\tilde{u}_H-u\|_\infty$
			& order &  $\|\tilde{u}'_H-u'\|_\infty$ & order &   $\|a\tilde{u}'_H-au'\|_\infty$ & order  & $\kappa$ & $\frac{a_{\max}}{a_{\min}}$ \\
			\hline	
$1/2$   &2.4665E+01   &   &4.2219E+02   &   &2.0878E+02   &   &11.64   &41\\
$1/2^2$   &7.1769E+00   &1.78   &2.3019E+02   &0.88   &1.1510E+02   &0.86   &11.64   &41\\
$1/2^3$   &1.9214E+00   &1.90   &1.1817E+02   &0.96   &6.0451E+01   &0.93   &11.64   &41\\
$1/2^4$   &4.9642E-01   &1.95   &5.8158E+01   &1.02   &3.1656E+01   &0.93   &11.64   &41\\
$1/2^5$   &1.2618E-01   &1.98   &2.8029E+01   &1.05   &1.6988E+01   &0.90   &11.64   &41\\
$1/2^6$   &3.1864E-02   &1.99   &1.2546E+01   &1.16   &9.4095E+00   &0.85   &11.64   &41\\		 
			\hline
	\end{tabular}}
	\label{table2:example1}
\end{table}	

\begin{table}[htbp]
	\caption{Performance of $l_2$ and $l_\infty$ norms of the error in \cref{Special:ex1} of the standard second-order FEM  on the uniform Cartesian mesh $H$. Note that $N=2^{14}$ in \eqref{fine:xi} for  estimating $l_2$ and $l_\infty$ norms.  }
	\centering
		\renewcommand\arraystretch{1}
	\setlength{\tabcolsep}{0.1mm}{
		 \begin{tabular}{c|c|c|c|c|c|c|c|c}
			\hline
			$H$ 	&   $\frac{\|u_{H/2}-u\|_2}{\|u\|_2}$
			&   $\frac{\|u'_{H/2}-u'\|_2}{\|u'\|_2}$ &   $\frac{\|au'_{H/2}-au'\|_2}{\|au'\|_2}$
			&   $\|u_{H/2}-u\|_\infty$
			&   $\|u'_{H/2}-u'\|_\infty$ &   $\|au'_{H/2}-au'\|_\infty$   & $\kappa$ & $\frac{a_{\max}}{a_{\min}}$ \\
			\hline	
$1/2$   &7.1E-01   &8.2E-01   &1.5E+00   &4.8E+01   &6.1E+02   &1.3E+03   &5.8E+00   &41\\
$1/2^2$   &7.0E-01   &8.1E-01   &1.5E+00   &4.7E+01   &5.9E+02   &1.5E+03   &2.5E+01   &41\\
$1/2^3$   &7.0E-01   &8.0E-01   &1.5E+00   &4.7E+01   &5.8E+02   &1.7E+03   &1.0E+02   &41\\
$1/2^4$   &6.9E-01   &8.0E-01   &1.5E+00   &4.7E+01   &5.8E+02   &1.7E+03   &4.1E+02   &41\\
$1/2^5$   &6.9E-01   &8.0E-01   &1.5E+00   &4.7E+01   &5.7E+02   &1.8E+03   &1.7E+03   &41\\
$1/2^6$   &6.9E-01   &8.0E-01   &1.5E+00   &4.7E+01   &5.7E+02   &1.8E+03   &6.6E+03   &41\\
			\hline
	\end{tabular}}
	\label{table3:example1}
\end{table}	

\begin{figure}[htbp]
	\centering
	\begin{subfigure}[b]{0.23\textwidth}
	 \includegraphics[width=4.5cm,height=3.cm]{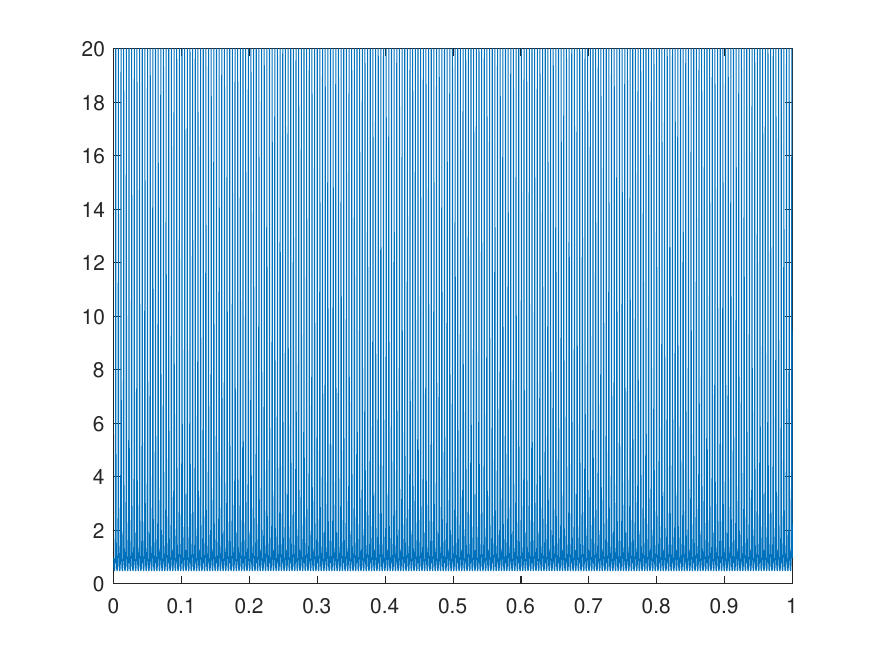}
\end{subfigure}	
	\begin{subfigure}[b]{0.23\textwidth}
		 \includegraphics[width=4.5cm,height=3.cm]{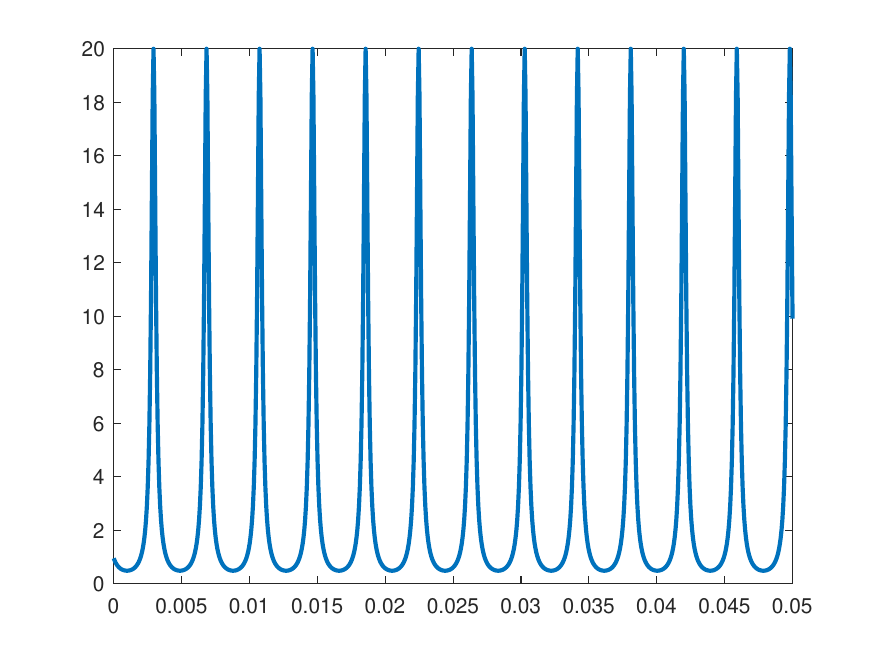}
	\end{subfigure}
	\begin{subfigure}[b]{0.23\textwidth}
		 \includegraphics[width=4.5cm,height=3.cm]{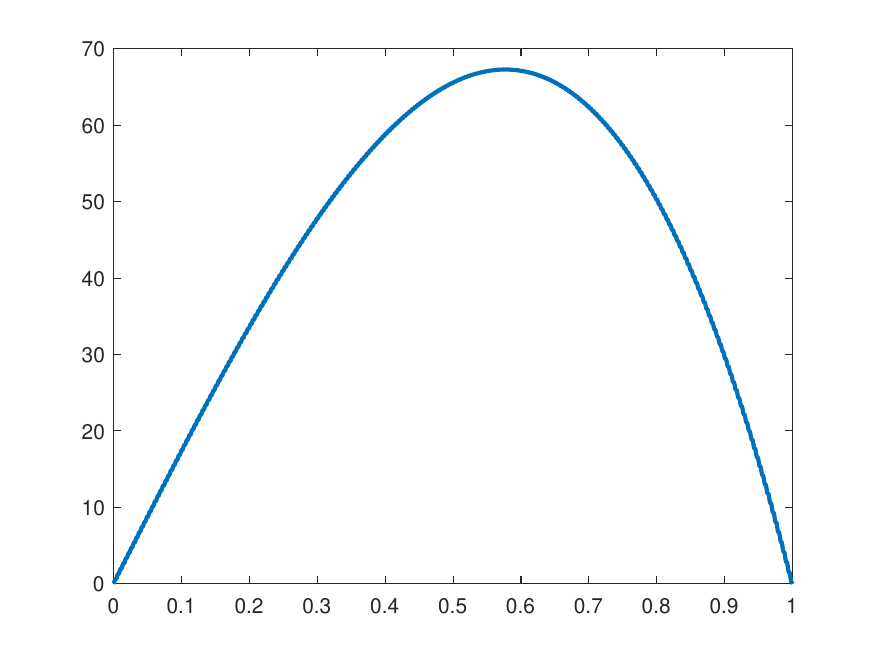}
	\end{subfigure}
	\begin{subfigure}[b]{0.23\textwidth}
		 \includegraphics[width=4.5cm,height=3.cm]{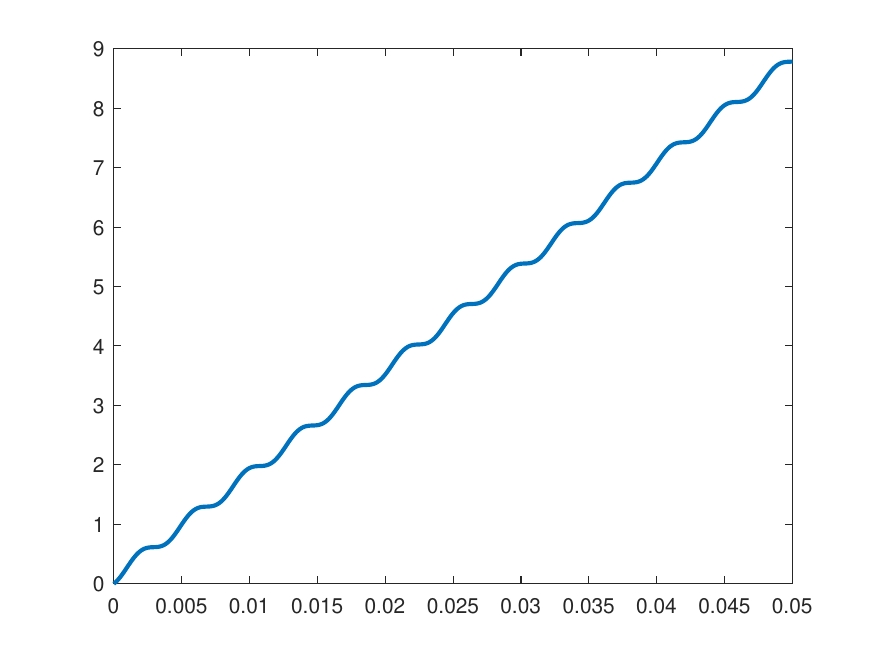}
	\end{subfigure}
	\begin{subfigure}[b]{0.23\textwidth}
		 \includegraphics[width=4.5cm,height=3.cm]{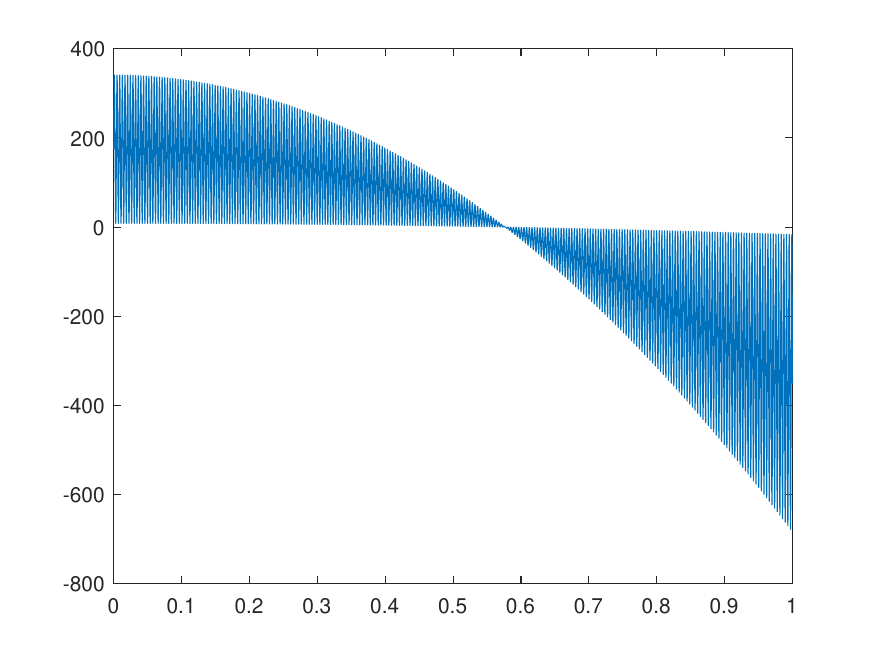}
	\end{subfigure}
	\begin{subfigure}[b]{0.23\textwidth}
		 \includegraphics[width=4.5cm,height=3.cm]{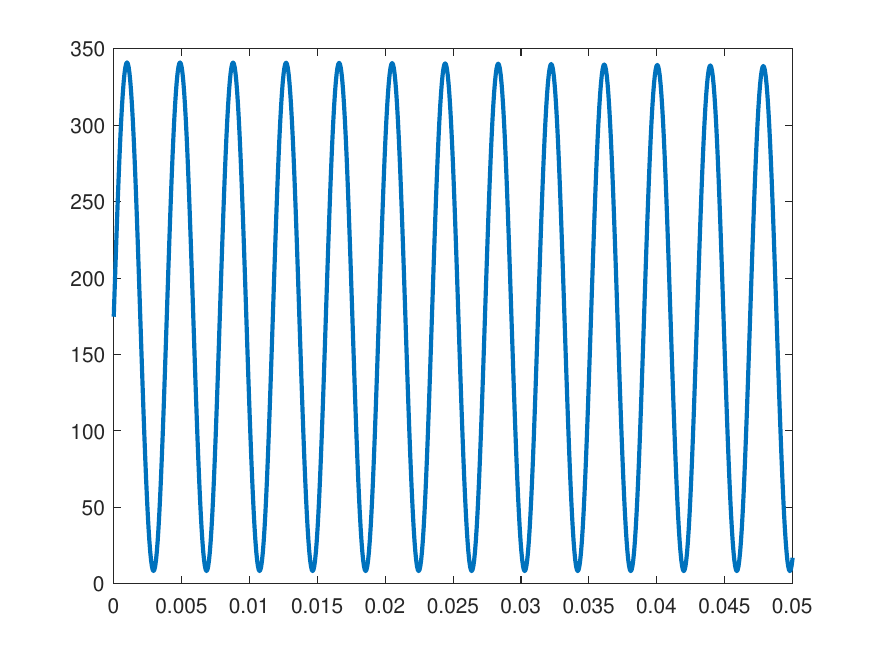}
	\end{subfigure}
	\begin{subfigure}[b]{0.23\textwidth}
		 \includegraphics[width=4.5cm,height=3.cm]{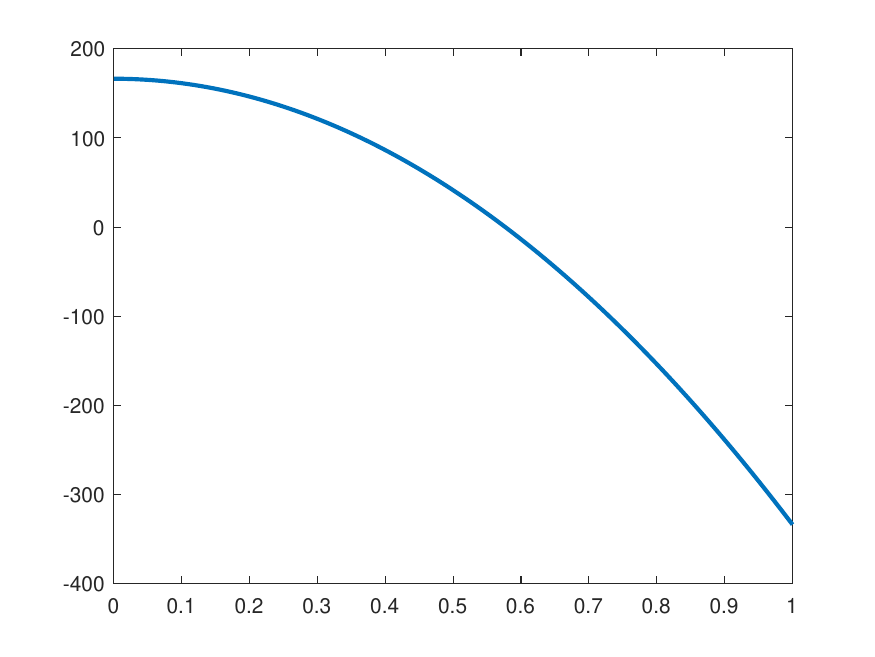}
	\end{subfigure}
	\begin{subfigure}[b]{0.23\textwidth}
		 \includegraphics[width=4.5cm,height=3.5cm]{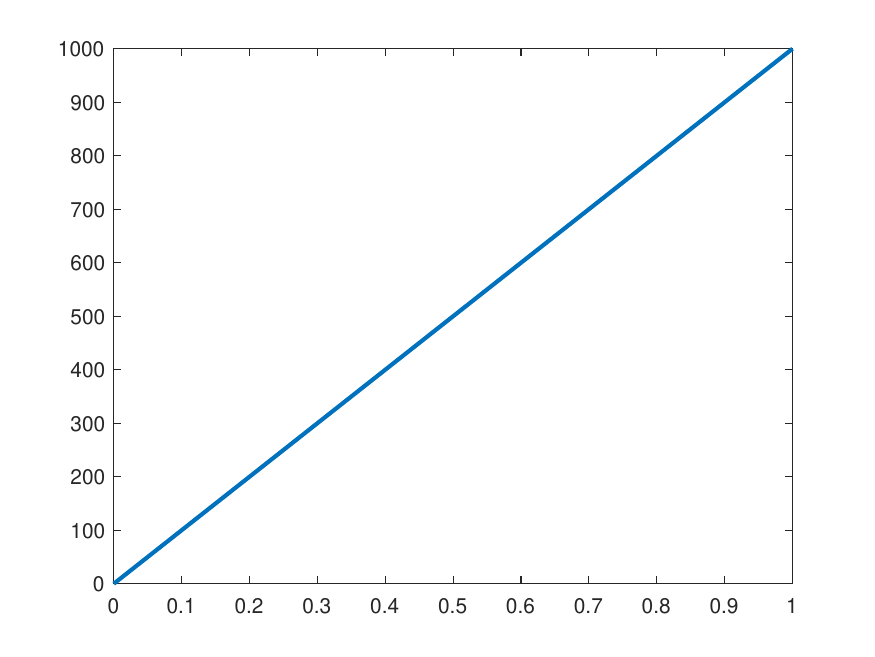}
	\end{subfigure}	
	\begin{subfigure}[b]{0.3\textwidth}
		 \includegraphics[width=6cm,height=3cm]{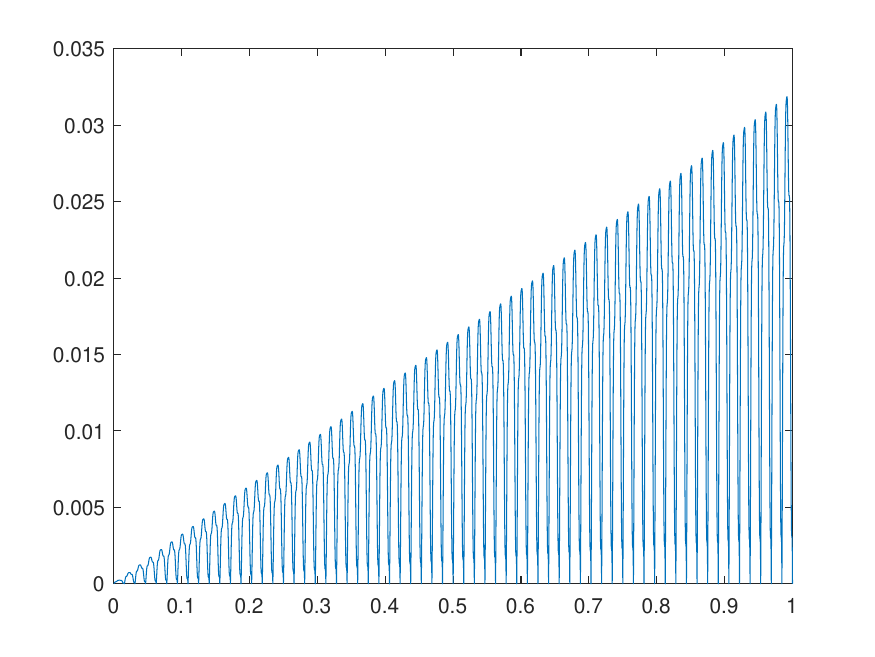}
	\end{subfigure}
	\begin{subfigure}[b]{0.3\textwidth}
		 \includegraphics[width=6cm,height=3cm]{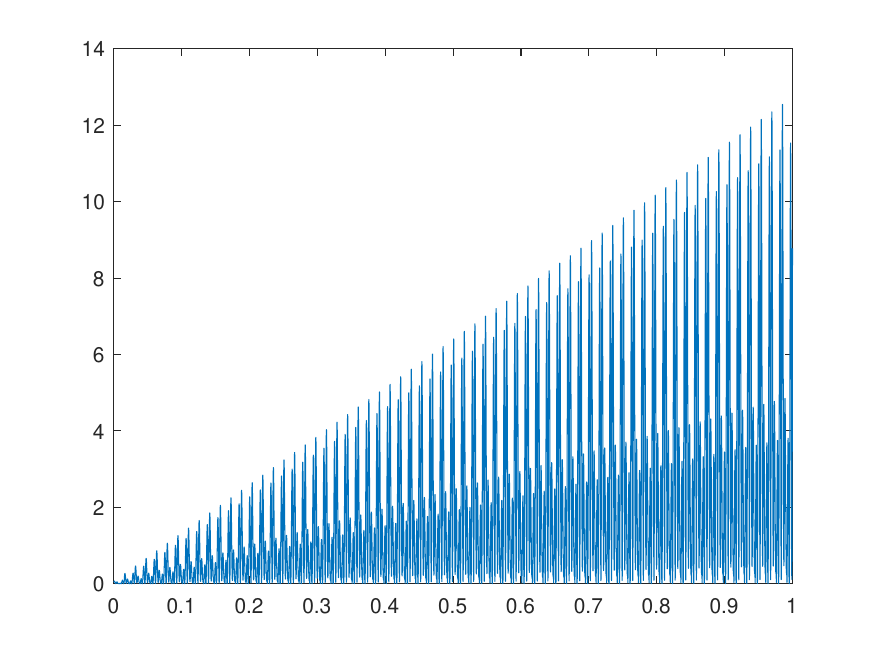}
	\end{subfigure}
	\begin{subfigure}[b]{0.3\textwidth}
		 \includegraphics[width=6cm,height=3cm]{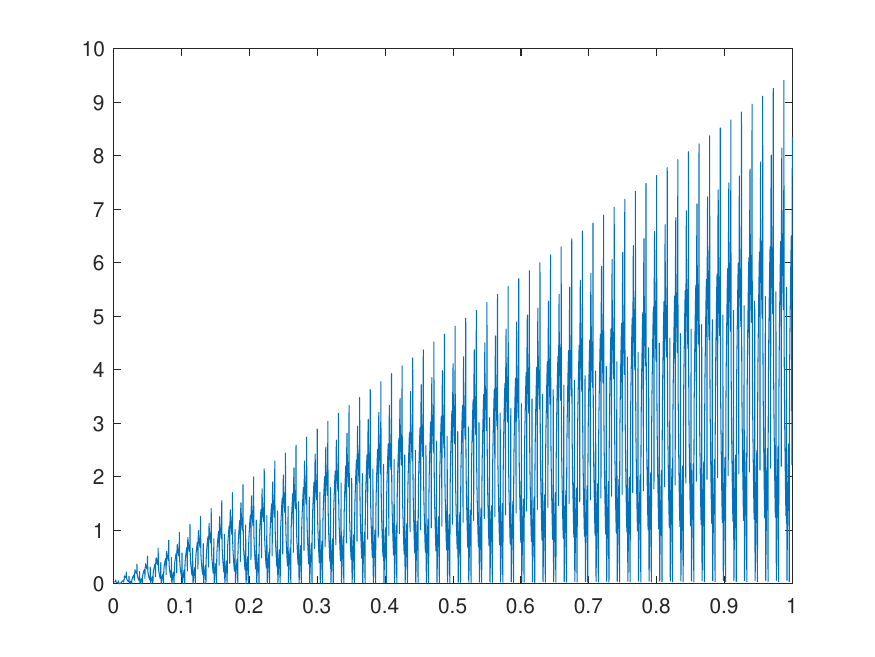}
	\end{subfigure}
	\caption{\cref{Special:ex1}: The top row: $a$ in [0,1] (first panel), $a$ in [0,0.05] (second panel), $u$ in [0,1] (third panel), and $u$ in [0,0.05] (fourth panel). The middle row: $u'$ in [0,1] (first panel), $u'$ in [0,0.05] (second panel), $au'$ in [0,1] (third panel), and $f$ in [0,1] (fourth panel). The bottom row: $|\tilde{u}_H-u|$ (first panel), $|\tilde{u}'_H-u'|$ (second panel), and  $|a\tilde{u}'_H-au'|$ (third panel) in [0,1] with $H=\frac{1}{2^6}$ at all $x_i=i/N$ for $i=0,\dots, N$ with $N=2^{14}$ in \eqref{fine:xi}.}
	\label{fig:exam1}
\end{figure}	

\begin{example}\label{Special:ex2}
	\normalfont
	Consider $f(x)=x$ and a diffusion discontinuous coefficient with high-contract jumps below
	\[ a(x)=10^4\chi_{[\frac{i-1}{2^8},\frac{i}{2^8})}
	+10^{-4}\chi_{[\frac{i}{2^8}, \frac{i+1}{2^8})}, \quad \text{with} \quad i=1,3,5,\dots,2^8-1.
	\]
Then the exact solution $u$ of \eqref{Model:Problem} can be expressed as
	\[
	 u(x)=-\frac{1}{6}\ca_{i,3}x^3+\ca_{i,1}x+\ca_{i,0} \quad \text{for } x\in \left[\frac{i-1}{2^8}, \frac{i}{2^8}\right] \text{ with } i=1,2,3,\dots,2^8,
	\]
	where
	\[
	\ca_{i,3}=\begin{cases}
		10^{-4}, &\text{ if } i \text{ is odd},\\
		10^4, &\text{ if } i \text{ is even},
	\end{cases}
	\]
	 and  the constants $\ca_{i,1},\ca_{i,0}$  can be uniquely determined by imposing the continuity of $u$ and $au'$  at $x=\frac{i-1}{2^8}$ for $i=2,3,4,\dots,2^8$, along with the boundary condition $u(0)=u(1)=0$.
The results on numerical solutions $\tilde{u}_H$ and $u_{H/2}$ are presented in \cref{table1:example2,table2:example2,table3:example2,fig:exam2}.
\end{example}

\begin{table}[htbp]
	\caption{Performance of the $l_2$ norm of the error in \cref{Special:ex2} of our proposed derivative-orthogonal wavelet multiscale method  on the uniform Cartesian mesh $H$. Note that $N=2^{14}$ in \eqref{fine:xi} for estimating the $l_2$ norm. }
	\centering
		\renewcommand\arraystretch{0.7}
	\setlength{\tabcolsep}{1mm}{
		 \begin{tabular}{c|c|c|c|c|c|c|c|c}
			\hline
			$H$
			&   $\frac{\|\tilde{u}_H-u\|_2}{\|u\|_2}$
			& order &  $\frac{\|\tilde{u}'_H-u'\|_2}{\|u'\|_2}$ & order &   $\frac{\|a\tilde{u}'_H-au'\|_2}{\|au'\|_2}$
			&order & $\kappa$ & $\frac{a_{\max}}{a_{\min}}$ \\
			\hline	
$1/2$   &2.7754E-01   &   &5.4457E-01   &   &5.4483E-01   &   &1.00E+08   &1.00E+08\\
$1/2^2$   &7.0993E-02   &1.97   &2.7751E-01   &0.97   &2.7767E-01   &0.97   &1.00E+08   &1.00E+08\\
$1/2^3$   &1.7837E-02   &1.99   &1.3926E-01   &0.99   &1.3934E-01   &0.99   &1.00E+08   &1.00E+08\\
$1/2^4$   &4.4544E-03   &2.00   &6.9403E-02   &1.00   &6.9445E-02   &1.00   &1.00E+08   &1.00E+08\\
$1/2^5$   &1.1021E-03   &2.01   &3.4087E-02   &1.03   &3.4109E-02   &1.03   &1.00E+08   &1.00E+08\\
$1/2^6$   &2.6042E-04   &2.08   &1.5737E-02   &1.12   &1.5747E-02   &1.12   &1.00E+08   &1.00E+08\\	
			\hline
	\end{tabular}}
	\label{table1:example2}
\end{table}	

\begin{table}[htbp]
	\caption{Performance of the $l_\infty$ norm of the error in \cref{Special:ex2} of our proposed derivative-orthogonal wavelet multiscale method  on the uniform Cartesian mesh $H$. Note that $N=2^{14}$ in \eqref{fine:xi} for estimating the $l_\infty$ norm.}
	\centering
		\renewcommand\arraystretch{0.7}
	\setlength{\tabcolsep}{1mm}{
		 \begin{tabular}{c|c|c|c|c|c|c|c|c}
			\hline
			$H$
			&   $\|\tilde{u}_H-u\|_\infty$
			& order &  $\|\tilde{u}'_H-u'\|_\infty$ & order &   $\|a\tilde{u}'_H-au'\|_\infty$ & order  & $\kappa$ & $\frac{a_{\max}}{a_{\min}}$ \\
			\hline	
	$1/2$   &1.1784E+02   &   &2.0687E+03   &   &2.0687E-01   &   &1.00E+08   &1.00E+08\\
	$1/2^2$   &3.4260E+01   &1.78   &1.1287E+03   &0.87   &1.1287E-01   &0.87   &1.00E+08   &1.00E+08\\
	$1/2^3$   &9.1744E+00   &1.90   &5.8065E+02   &0.96   &5.8065E-02   &0.96   &1.00E+08   &1.00E+08\\
	$1/2^4$   &2.3699E+00   &1.95   &2.8707E+02   &1.02   &2.8707E-02   &1.02   &1.00E+08   &1.00E+08\\
	$1/2^5$   &6.0201E-01   &1.98   &1.3540E+02   &1.08   &1.3540E-02   &1.08   &1.00E+08   &1.00E+08\\
	$1/2^6$   &1.5169E-01   &1.99   &5.8340E+01   &1.21   &5.8340E-03   &1.21   &1.00E+08   &1.00E+08\\
			\hline
	\end{tabular}}
	\label{table2:example2}
\end{table}	

\begin{table}[htbp]
	\caption{Performance of $l_2$ and $l_\infty$ norms of the error in \cref{Special:ex2} of the standard second-order FEM  on the uniform Cartesian mesh $H$. Note that $N=2^{14}$ in \eqref{fine:xi} for estimating $l_2$ and $l_\infty$ norms.  }
	\centering
		\renewcommand\arraystretch{0.7}
	\setlength{\tabcolsep}{0.1mm}{
		 \begin{tabular}{c|c|c|c|c|c|c|c|c}
			\hline
			$H$ 	&   $\frac{\|u_{H/2}-u\|_2}{\|u\|_2}$
			&   $\frac{\|u'_{H/2}-u'\|_2}{\|u'\|_2}$ &   $\frac{\|au'_{H/2}-au'\|_2}{\|au'\|_2}$
			&   $\|u_{H/2}-u\|_\infty$
			&   $\|u'_{H/2}-u'\|_\infty$ &   $\|au'_{H/2}-au'\|_\infty$   & $\kappa$ & $\frac{a_{\max}}{a_{\min}}$ \\
			\hline	
			$1/2$   &1.0E+00   &1.0E+00   &1.0E+00   &3.2E+02   &3.3E+03   &3.3E-01   &5.8E+00   &1.0E+08\\
			$1/2^2$   &1.0E+00   &1.0E+00   &1.0E+00   &3.2E+02   &3.3E+03   &3.3E-01   &2.5E+01   &1.0E+08\\
			$1/2^3$   &1.0E+00   &1.0E+00   &1.0E+00   &3.2E+02   &3.3E+03   &3.3E-01   &1.0E+02   &1.0E+08\\
			$1/2^4$   &1.0E+00   &1.0E+00   &1.0E+00   &3.2E+02   &3.3E+03   &3.3E-01   &4.1E+02   &1.0E+08\\
			$1/2^5$   &1.0E+00   &1.0E+00   &1.0E+00   &3.2E+02   &3.3E+03   &3.3E-01   &1.7E+03   &1.0E+08\\
			$1/2^6$   &1.0E+00   &1.0E+00   &1.0E+00   &3.2E+02   &3.3E+03   &3.3E-01   &6.6E+03   &1.0E+08\\
			\hline
	\end{tabular}}
	\label{table3:example2}
\end{table}

\begin{figure}[htbp]
	\centering
	\begin{subfigure}[b]{0.23\textwidth}
		 \includegraphics[width=4.5cm,height=3.cm]{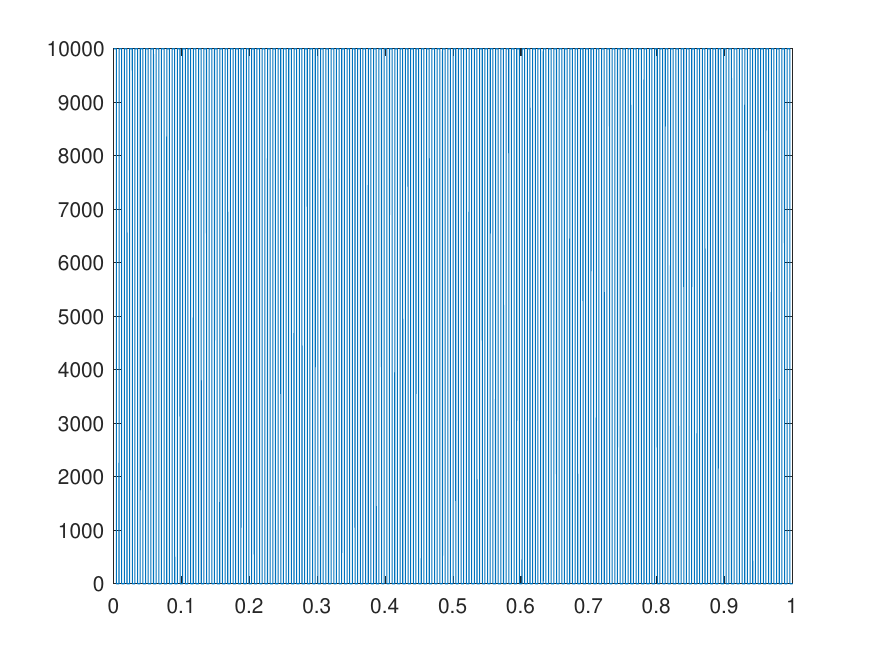}
	\end{subfigure}	
	\begin{subfigure}[b]{0.23\textwidth}
		 \includegraphics[width=4.5cm,height=3.cm]{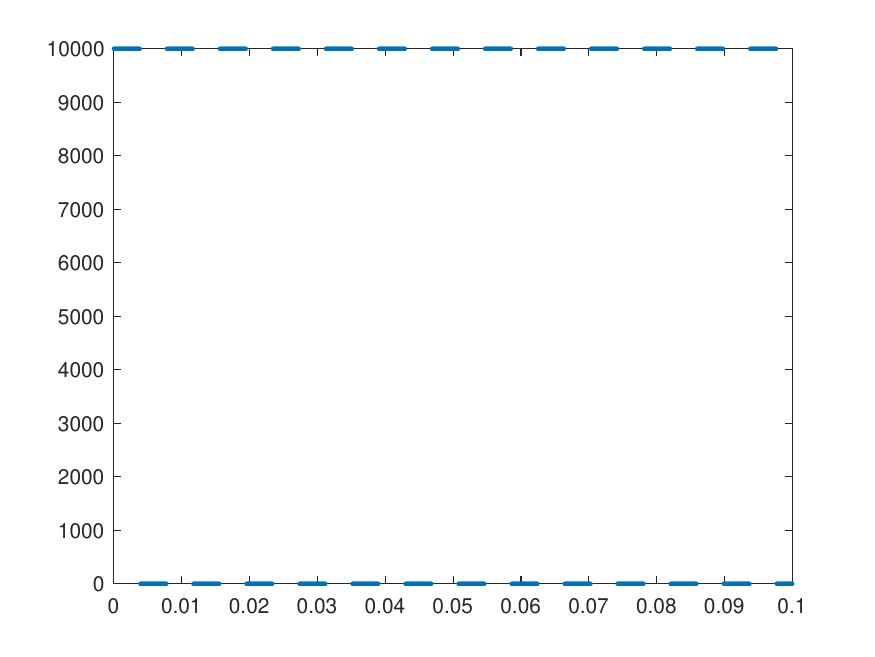}
	\end{subfigure}
	\begin{subfigure}[b]{0.23\textwidth}
		 \includegraphics[width=4.5cm,height=3.cm]{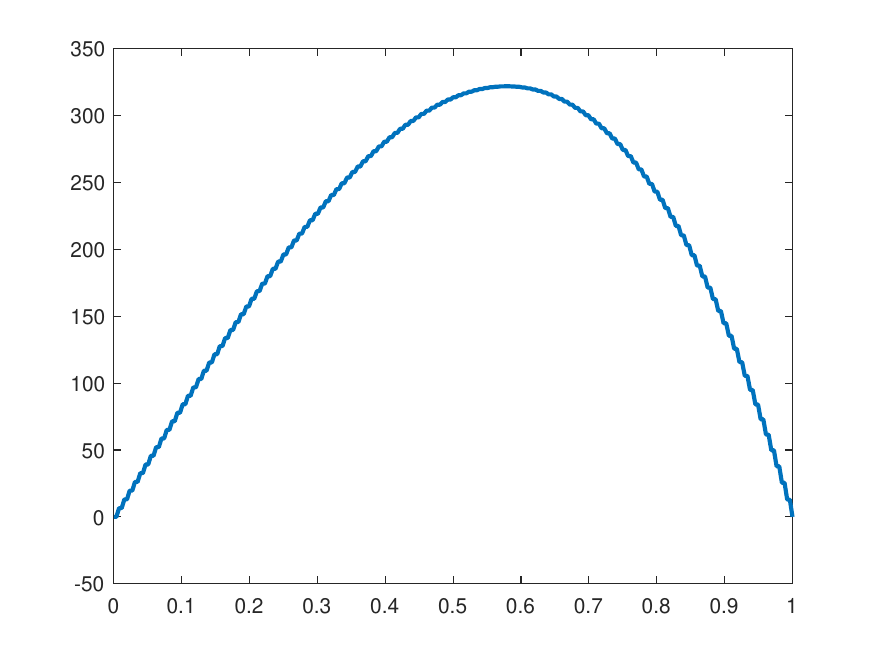}
	\end{subfigure}
	\begin{subfigure}[b]{0.23\textwidth}
		 \includegraphics[width=4.5cm,height=3.cm]{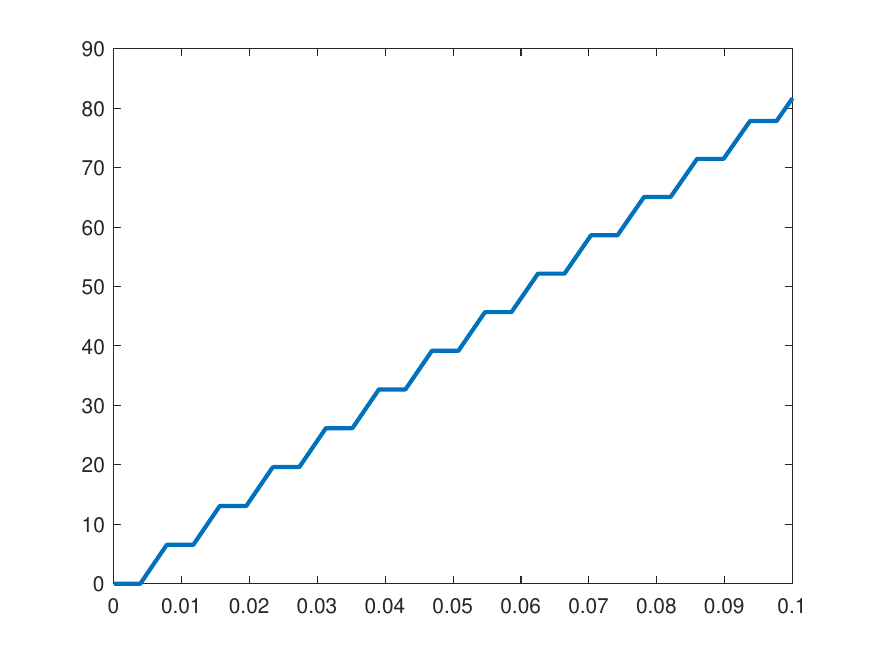}
	\end{subfigure}
	\begin{subfigure}[b]{0.23\textwidth}
		 \includegraphics[width=4.5cm,height=3.cm]{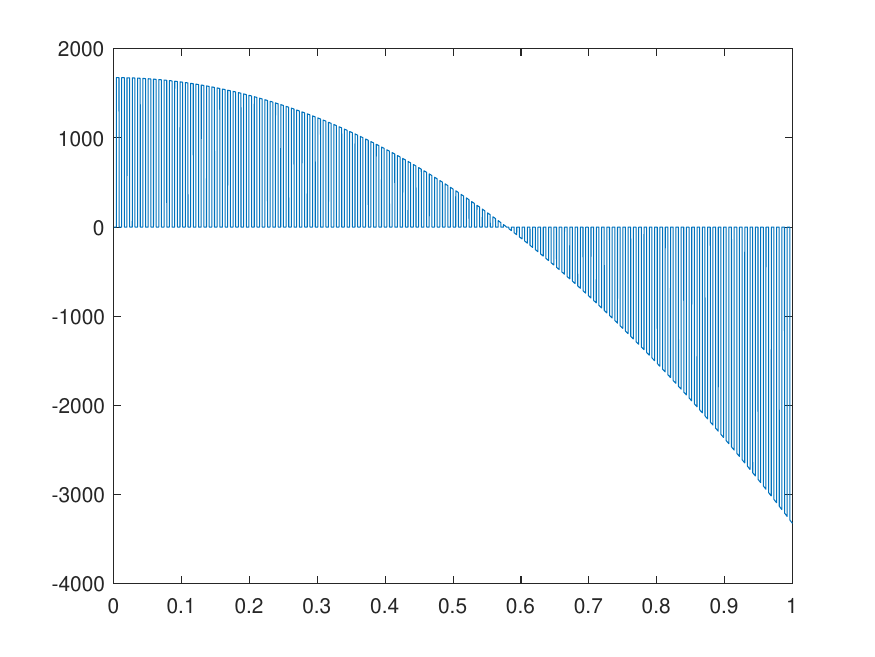}
	\end{subfigure}
	\begin{subfigure}[b]{0.23\textwidth}
		 \includegraphics[width=4.5cm,height=3.cm]{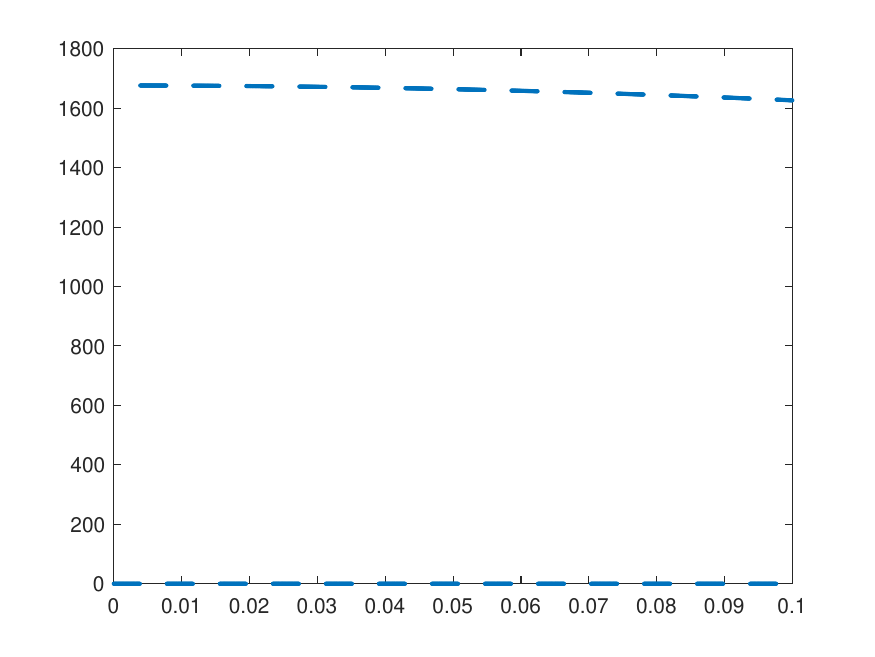}
	\end{subfigure}
	\begin{subfigure}[b]{0.23\textwidth}
		 \includegraphics[width=4.5cm,height=3.cm]{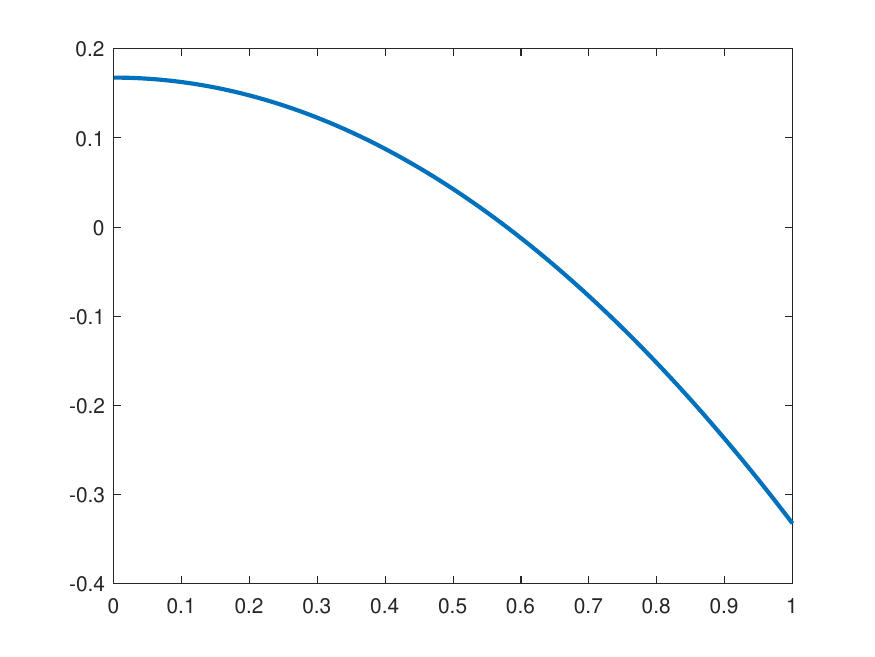}
	\end{subfigure}
	\begin{subfigure}[b]{0.23\textwidth}
		 \includegraphics[width=4.5cm,height=3.cm]{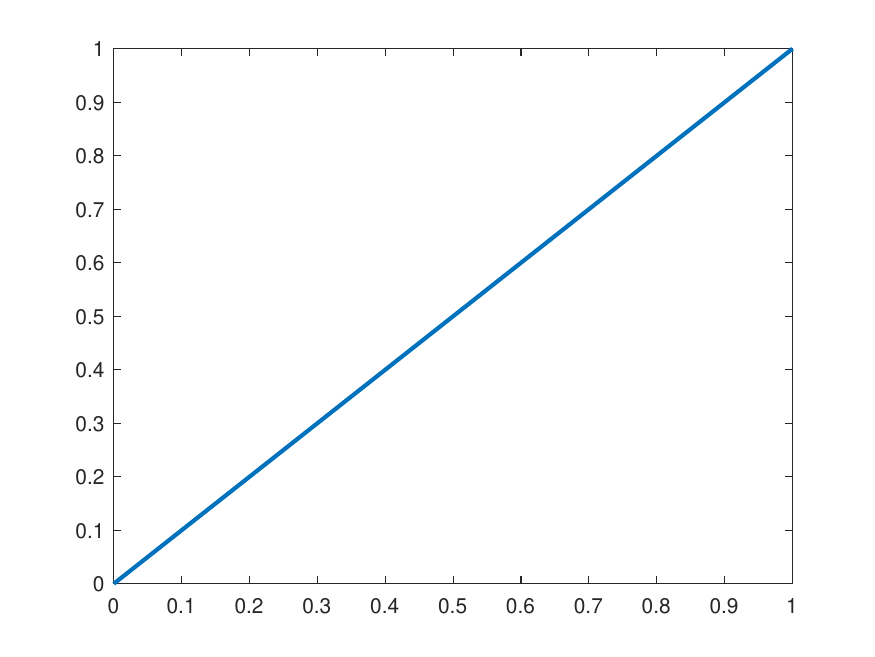}
	\end{subfigure}	
	\begin{subfigure}[b]{0.3\textwidth}
		 \includegraphics[width=6cm,height=3cm]{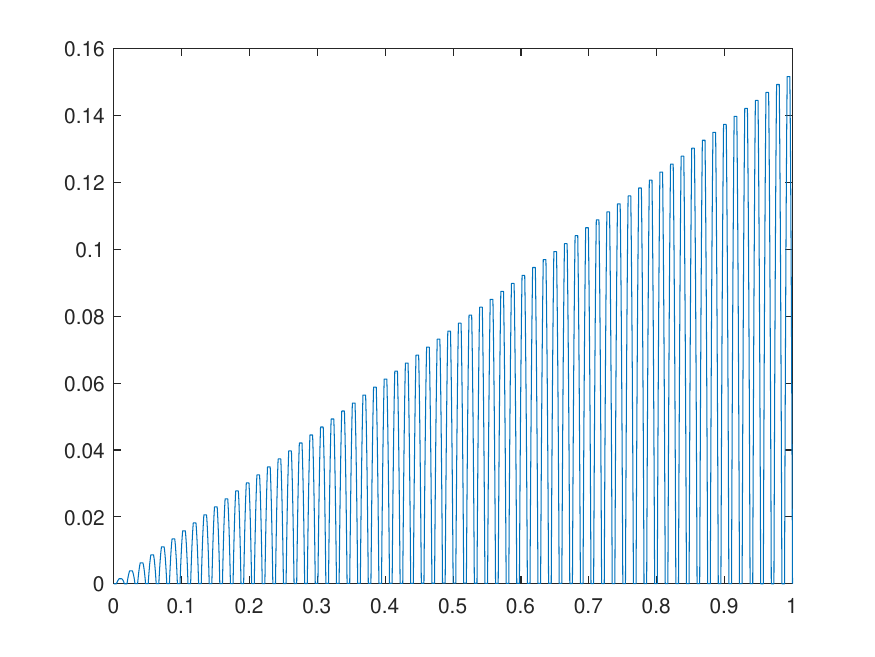}
	\end{subfigure}
	\begin{subfigure}[b]{0.3\textwidth}
		 \includegraphics[width=6cm,height=3cm]{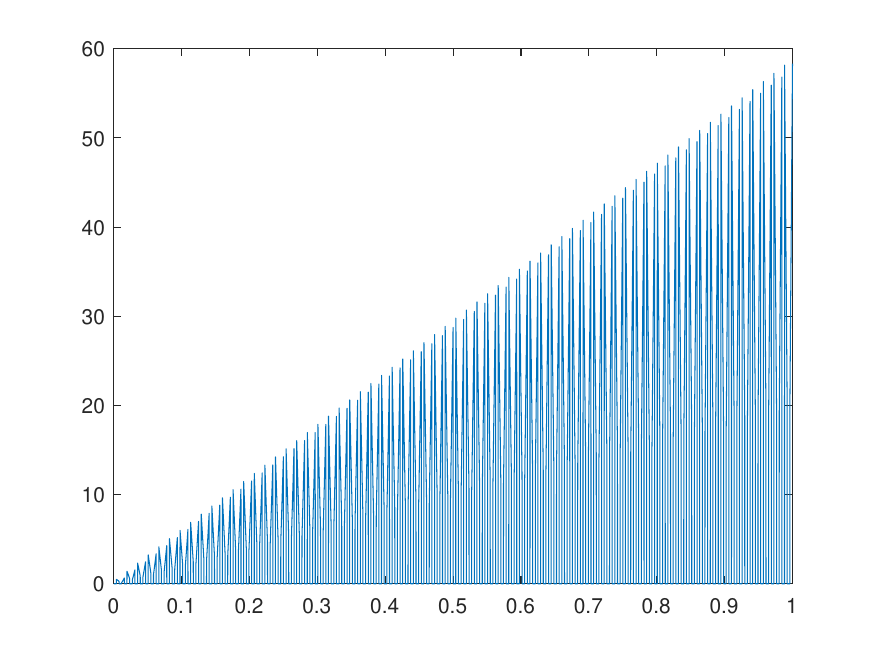}
	\end{subfigure}
	\begin{subfigure}[b]{0.3\textwidth}
		 \includegraphics[width=6cm,height=3cm]{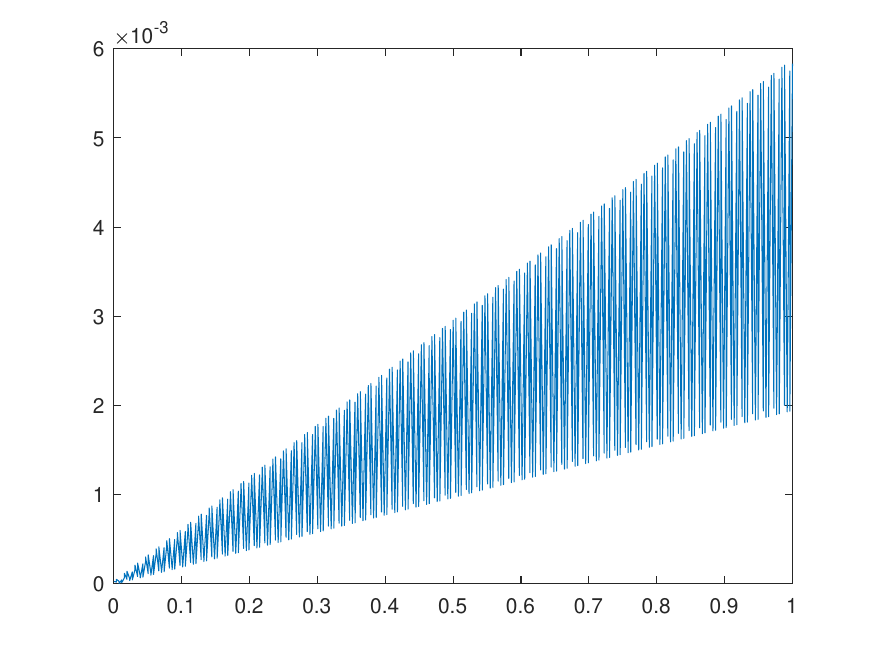}
	\end{subfigure}
	\caption{\cref{Special:ex2}: The top row: $a$ in [0,1] (first panel), $a$ in [0,0.1] (second panel), $u$ in [0,1] (third panel), and $u$ in [0,0.1] (fourth panel). The middle row: $u'$ in [0,1] (first panel), $u'$ in [0,0.1] (second panel), $au'$ in [0,1] (third panel), and $f$ in [0,1] (fourth panel). The bottom row: $|\tilde{u}_H-u|$ (first panel), $|\tilde{u}'_H-u'|$ (second panel), and  $|a\tilde{u}'_H-au'|$ (third panel) in [0,1] with $H=\frac{1}{2^6}$ at all $x_i=i/N$ for $i=0,\dots, N$ with $N=2^{14}$ in \eqref{fine:xi}.}
	\label{fig:exam2}
\end{figure}	

\begin{example}\label{Special:ex2:infty}
	\normalfont
	Consider $f(x)=1$ and a continuous diffusion coefficient $a(x)=\frac{1}{x^2}\frac{1}{1.05+\sin(2^{10}\pi x)}$ with the high-frequency oscillation.
Note that $\lim_{x\to 0^+} a(x)=\infty$.
	Then the exact solution $u$ of \eqref{Model:Problem} can be expressed as
	\be
	\begin{split}
	u(x)=&\frac{-\ca_1\ra_1x^3}{3} - \frac{\ra_1x^4}{4} + \frac{\ca_1x^2\cos(\ra_2x)}{\ra_2} - \frac{2\ca_1\cos(\ra_2x)}{\ra_2^3}
	- \frac{2\ca_1x\sin(\ra_2x)}{\ra_2^2} \\
	&+ \frac{x^3\cos(\ra_2x)}{\ra_2} - \frac{3x^2\sin(\ra_2x)}{\ra_2^2} + \frac{6\sin(\ra_2x)}{\ra_2^4} - \frac{6x\cos(\ra_2x)}{\ra_2^3} + \ca_2,
	\end{split}
	\ee
	where $\ra_1=1.05$, $\ra_2=2^{10}\pi$,
	and  the constants $\ca_{1},\ca_{2}$  can be uniquely determined by imposing the boundary condition $u(0)=u(1)=0$. The results on numerical solutions $\tilde{u}_H$ and $u_{H/2}$ are presented in \cref{table1:example2:infty,table2:example2:infty,table3:example2:infty,fig:exam2:infty}.
\end{example}

\begin{table}[htbp]
	\caption{Performance of the $l_2$ norm of the error in \cref{Special:ex2:infty} of our proposed derivative-orthogonal wavelet multiscale method  on the uniform Cartesian mesh $H$. Note that $N=2^{15}$ in \eqref{fine:xi} for estimating the $l_2$ norm. Since $\lim_{x\to 0^+}a(x)=\infty$, we replace $\sum_{i=0}^N$ by $\sum_{i=1}^N$ in $\frac{\|a\tilde{u}'_H-au'\|_2}{\|au'\|_2}$ in \eqref{Error:1}, and $a_{\max}=\max\limits_{2^{-15}\le x \le 1} a(x)$.}
	\centering
		\renewcommand\arraystretch{0.7}
	\setlength{\tabcolsep}{1mm}{
		 \begin{tabular}{c|c|c|c|c|c|c|c|c}
			\hline
			$H$
			&   $\frac{\|\tilde{u}_H-u\|_2}{\|u\|_2}$
			& order &  $\frac{\|\tilde{u}'_H-u'\|_2}{\|u'\|_2}$ & order &   $\frac{\|a\tilde{u}'_H-au'\|_2}{\|au'\|_2}$
			&order & $\kappa$ & $\frac{a_{\max}}{a_{\min}}$ \\
			\hline	
$1/2$   &4.7851E-01   &   &7.3506E-01   &   &4.0353E+00   &   &1.2E+05   &1.9E+09 \\
$1/2^{2}$   &1.3535E-01   &1.82   &4.2175E-01   &0.80   &5.3416E-01   &2.92   &7.2E+05   &1.9E+09 \\
$1/2^{3}$   &3.5206E-02   &1.94   &2.1820E-01   &0.95   &1.0874E-01   &2.30   &2.2E+06   &1.9E+09 \\
$1/2^{4}$   &8.9328E-03   &1.98   &1.1000E-01   &0.99   &4.7332E-02   &1.20   &5.2E+06   &1.9E+09 \\
$1/2^{5}$   &2.2461E-03   &1.99   &5.5063E-02   &1.00   &2.3678E-02   &1.00   &1.1E+07   &1.9E+09 \\
$1/2^{6}$   &5.6249E-04   &2.00   &2.7445E-02   &1.00   &1.1889E-02   &0.99   &2.4E+07   &1.9E+09 \\
$1/2^{7}$   &1.4022E-04   &2.00   &1.3521E-02   &1.02   &5.9820E-03   &0.99   &4.9E+07   &1.9E+09 \\
			\hline
	\end{tabular}}
	\label{table1:example2:infty}
\end{table}

\begin{table}[htbp]
	\caption{Performance of the $l_\infty$ norm of the error in \cref{Special:ex2:infty} of our proposed derivative-orthogonal wavelet multiscale method  on the uniform Cartesian mesh $H$. Note that $N=2^{15}$ in \eqref{fine:xi} for estimating the $l_\infty$ norm. Since $\lim_{x\to 0^+}a(x)=\infty$, we replace $\max_{0\le i\le N}$ by $\max_{1\le i\le N}$ in $\|a\tilde{u}'_H-au'\|_\infty$ in  \eqref{Error:2}, and $a_{\max}=\max\limits_{ 2^{-15}\le x \le 1} a(x)$.}
	\centering
		\renewcommand\arraystretch{0.7}	
	\setlength{\tabcolsep}{1mm}{
		 \begin{tabular}{c|c|c|c|c|c|c|c|c}
			\hline
			$H$
			&   $\|\tilde{u}_H-u\|_\infty$
			& order &  $\|\tilde{u}'_H-u'\|_\infty$ & order &   $\|a\tilde{u}'_H-au'\|_\infty$ & order  & $\kappa$ & $\frac{a_{\max}}{a_{\min}}$ \\
			\hline	
$1/2$   &1.5972E-02   &   &3.6152E-01   &   &2.6852E+02   &   &1.2E+05   &1.9E+09 \\
$1/2^{2}$   &5.8117E-03   &1.46   &2.1609E-01   &0.74   &3.3392E+01   &3.01   &7.2E+05   &1.9E+09 \\
$1/2^{3}$   &1.7303E-03   &1.75   &1.1623E-01   &0.89   &3.6502E+00   &3.19   &2.2E+06   &1.9E+09 \\
$1/2^{4}$   &4.7107E-04   &1.88   &5.9318E-02   &0.97   &2.5125E-01   &3.86   &5.2E+06   &1.9E+09 \\
$1/2^{5}$   &1.2284E-04   &1.94   &2.9109E-02   &1.03   &6.4687E-02   &1.96   &1.1E+07   &1.9E+09 \\
$1/2^{6}$   &3.1387E-05   &1.97   &1.4016E-02   &1.05   &5.4947E-02   &0.24   &2.4E+07   &1.9E+09 \\
$1/2^{7}$   &7.9477E-06   &1.98   &6.2779E-03   &1.16   &2.9363E-02   &0.90   &4.9E+07   &1.9E+09 \\
			\hline
	\end{tabular}}
	\label{table2:example2:infty}
\end{table}	

\begin{table}[htbp]
	\caption{Performance of  $l_2$ and $l_\infty$ norms of the error in \cref{Special:ex2:infty} of the standard second-order FEM  on the uniform Cartesian mesh $H$. Note that $N=2^{15}$ in \eqref{fine:xi} for estimating $l_2$ and $l_\infty$ norms.  Since $\lim_{x\to 0^+}a(x)=\infty$, we replace  $\sum_{i=0}^N$ by $\sum_{i=1}^N$ and $\max_{0\le i\le N}$ by $\max_{1\le i\le N}$  in the following  $\frac{\|au'_{H/2}-au'\|_2}{\|au'\|_2}$  and $\|au'_{H/2}-au'\|_\infty$ errors, and $a_{\max}=\max\limits_{2^{-15}\le x \le 1} a(x)$. }
	\centering
	\renewcommand\arraystretch{0.7}	
	\setlength{\tabcolsep}{0.1mm}{
		 \begin{tabular}{c|c|c|c|c|c|c|c|c}
			\hline
			$H$ 	&   $\frac{\|u_{H/2}-u\|_2}{\|u\|_2}$
			&   $\frac{\|u'_{H/2}-u'\|_2}{\|u'\|_2}$ &   $\frac{\|au'_{H/2}-au'\|_2}{\|au'\|_2}$
			&   $\|u_{H/2}-u\|_\infty$
			&   $\|u'_{H/2}-u'\|_\infty$ &   $\|au'_{H/2}-au'\|_\infty$   & $\kappa$ & $\frac{a_{\max}}{a_{\min}}$ \\
			\hline	
$1/2$   &7.8E-01   &8.7E-01   &1.5E+01   &2.1E+01   &4.8E+02   &1.0E+06   &6.1E+04   &1.9E+09\\
$1/2^{2}$   &7.1E-01   &8.2E-01   &8.3E+00   &2.0E+01   &4.6E+02   &5.5E+05   &4.9E+05   &1.9E+09\\
$1/2^{3}$   &7.0E-01   &8.1E-01   &4.5E+00   &1.9E+01   &4.4E+02   &2.9E+05   &3.9E+06   &1.9E+09\\
$1/2^{4}$   &7.0E-01   &8.0E-01   &2.6E+00   &1.9E+01   &4.4E+02   &1.5E+05   &3.1E+07   &1.9E+09\\
$1/2^{5}$   &6.9E-01   &8.0E-01   &1.9E+00   &1.9E+01   &4.3E+02   &7.4E+04   &2.5E+08   &1.9E+09\\
$1/2^{6}$   &6.9E-01   &8.0E-01   &1.6E+00   &1.9E+01   &4.3E+02   &3.7E+04   &2.0E+09   &1.9E+09\\
$1/2^{7}$   &6.9E-01   &8.0E-01   &1.5E+00   &1.9E+01   &4.3E+02   &1.8E+04   &1.6E+10   &1.9E+09\\
			\hline
	\end{tabular}}
	\label{table3:example2:infty}
\end{table}	

\begin{figure}[htbp]
	\centering
	\begin{subfigure}[b]{0.23\textwidth}
		 \includegraphics[width=4.5cm,height=3.cm]{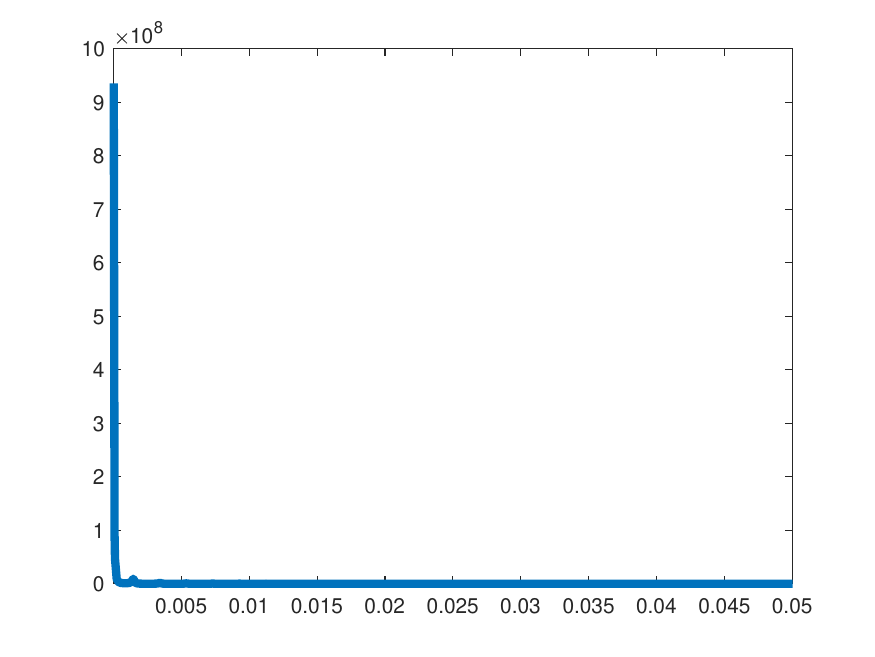}
	\end{subfigure}	
	\begin{subfigure}[b]{0.23\textwidth}
		 \includegraphics[width=4.5cm,height=3.cm]{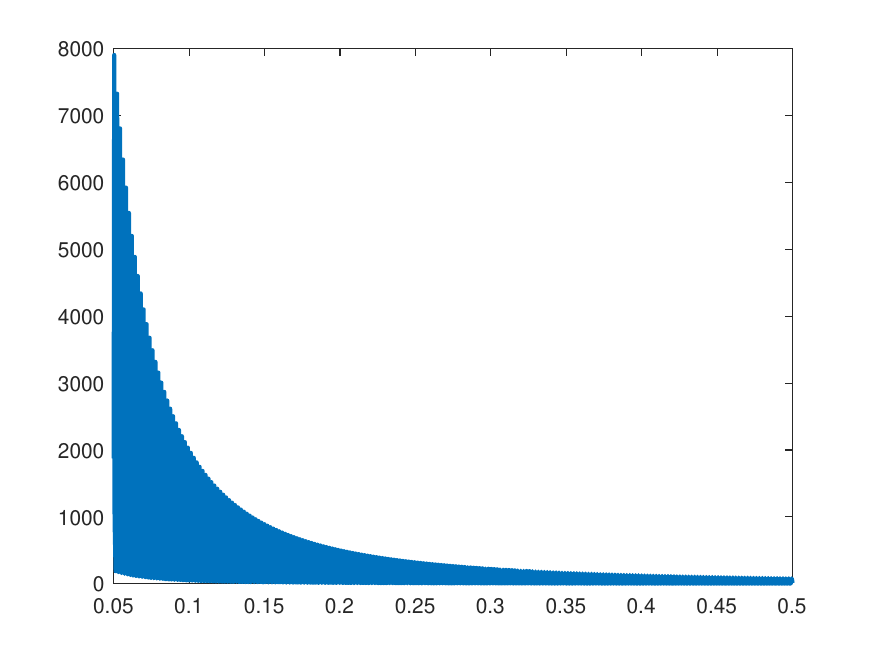}
	\end{subfigure}
	\begin{subfigure}[b]{0.23\textwidth}
		 \includegraphics[width=4.5cm,height=3.cm]{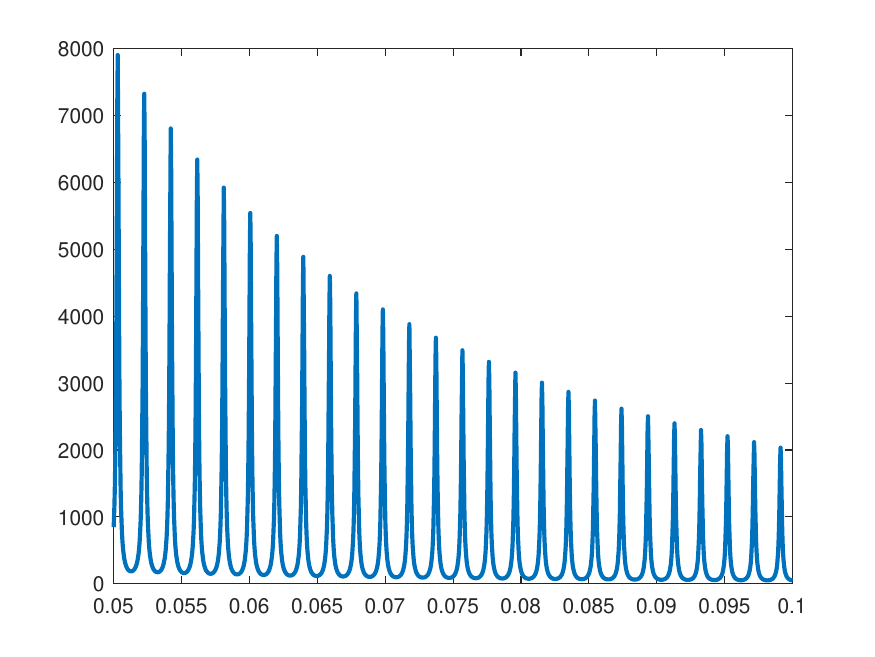}
	\end{subfigure}
	\begin{subfigure}[b]{0.23\textwidth}
		 \includegraphics[width=4.5cm,height=3.cm]{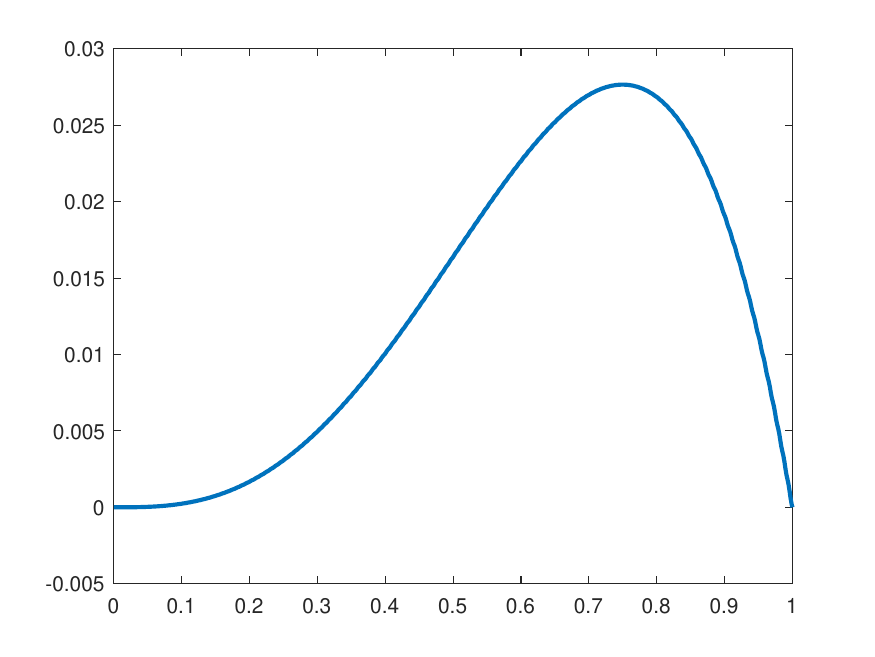}
	\end{subfigure}
	\begin{subfigure}[b]{0.23\textwidth}
		 \includegraphics[width=4.5cm,height=3.cm]{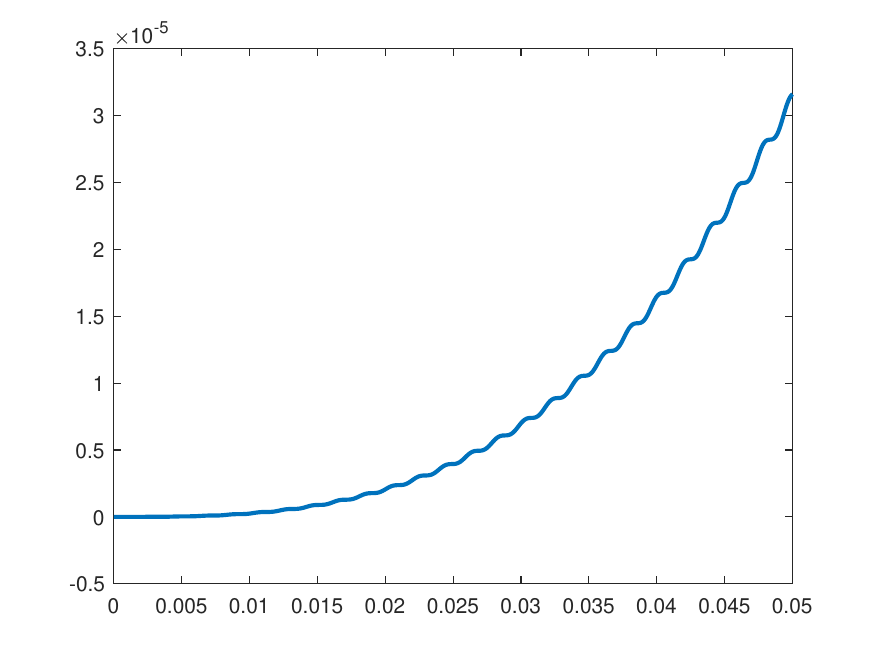}
	\end{subfigure}
	\begin{subfigure}[b]{0.23\textwidth}
		 \includegraphics[width=4.5cm,height=3.cm]{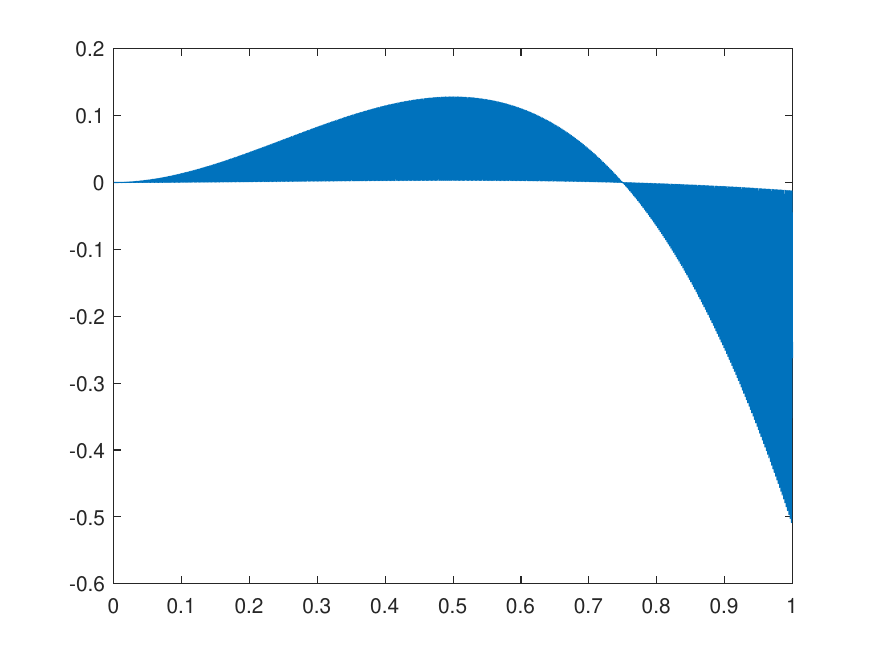}
	\end{subfigure}
	\begin{subfigure}[b]{0.23\textwidth}
		 \includegraphics[width=4.5cm,height=3.cm]{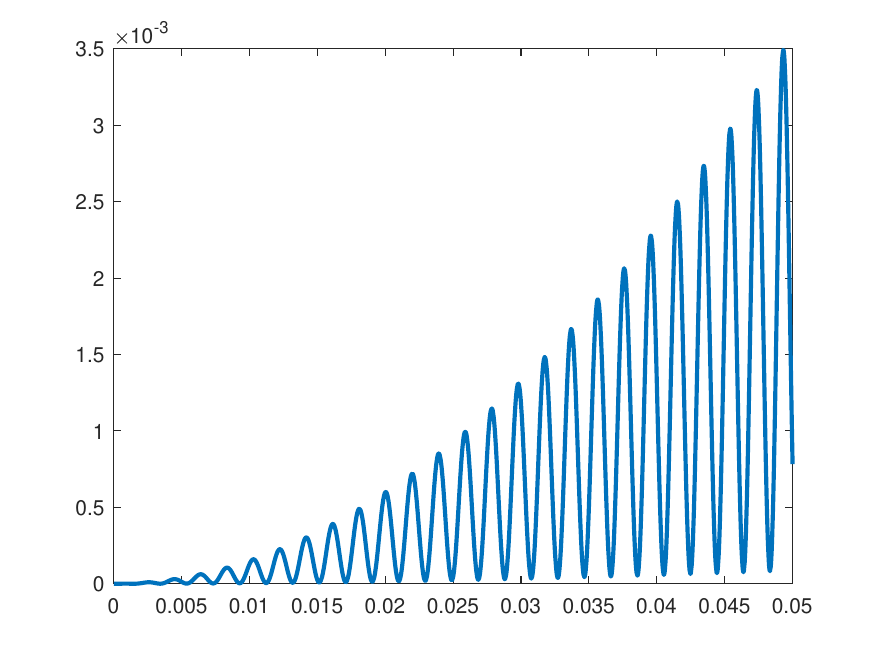}
	\end{subfigure}
	\begin{subfigure}[b]{0.23\textwidth}
		 \includegraphics[width=4.5cm,height=3.cm]{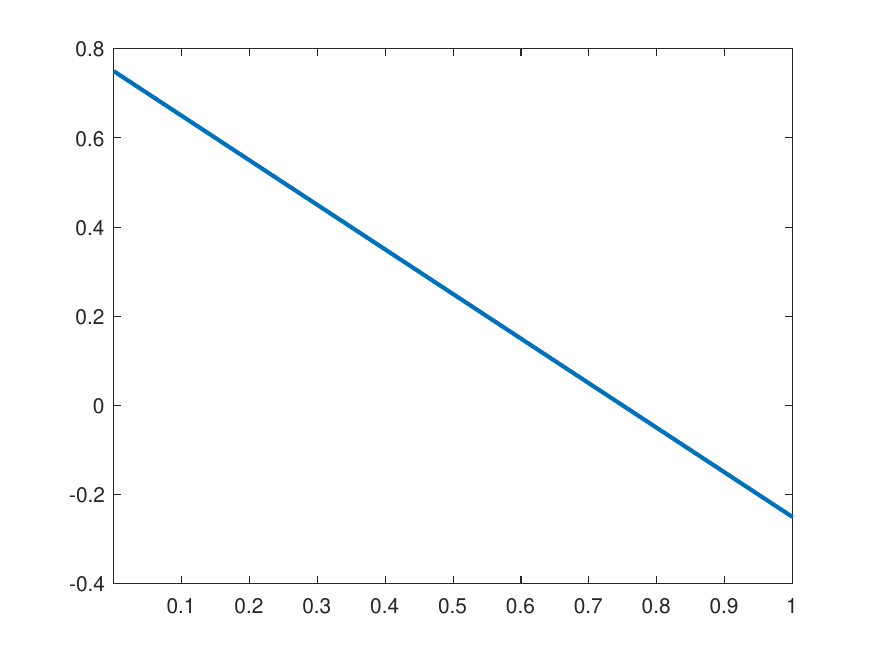}
	\end{subfigure}	
	\begin{subfigure}[b]{0.3\textwidth}
		 \includegraphics[width=6cm,height=3cm]{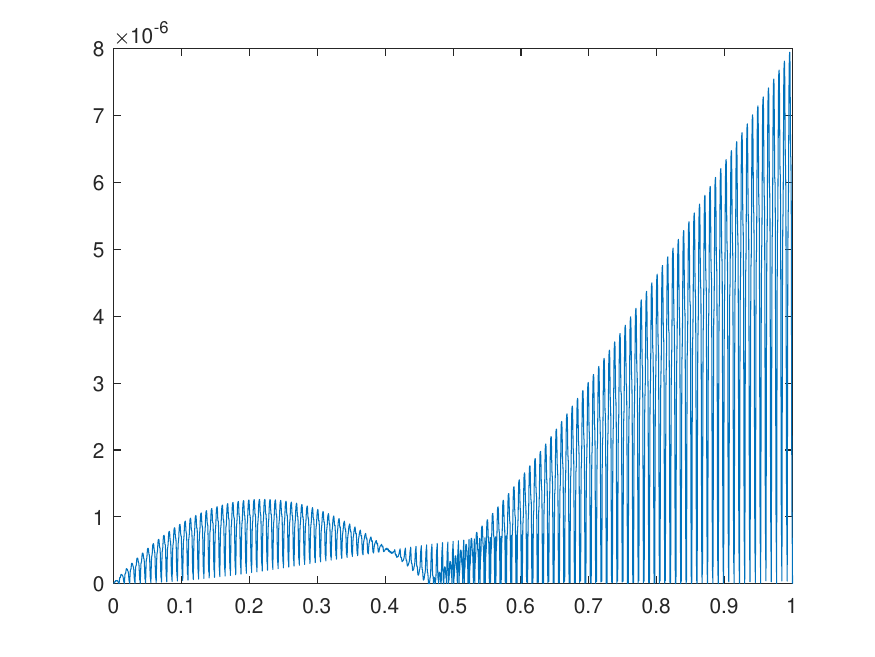}
	\end{subfigure}
	\begin{subfigure}[b]{0.3\textwidth}
		 \includegraphics[width=6cm,height=3cm]{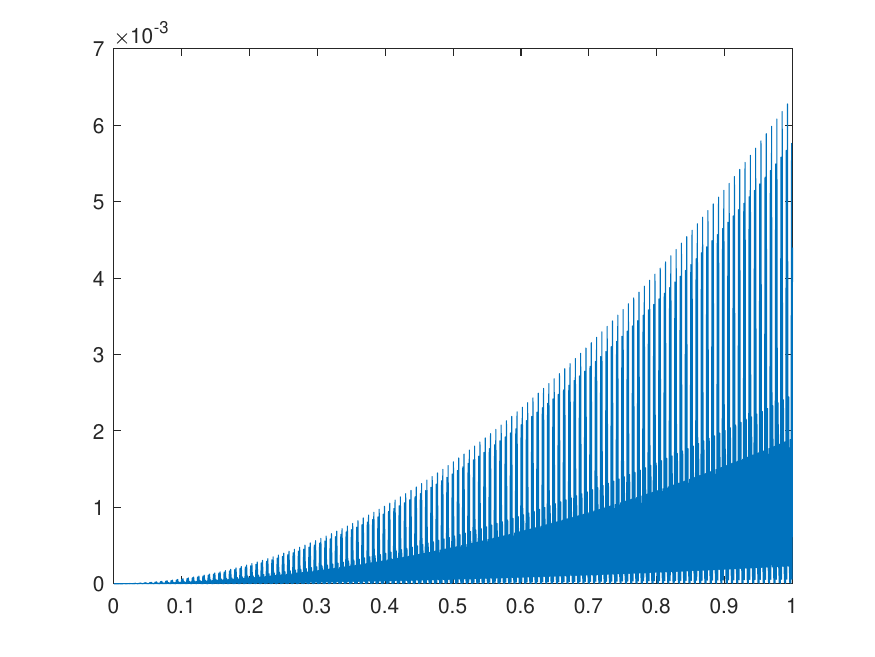}
	\end{subfigure}
	\begin{subfigure}[b]{0.3\textwidth}
		 \includegraphics[width=6cm,height=3cm]{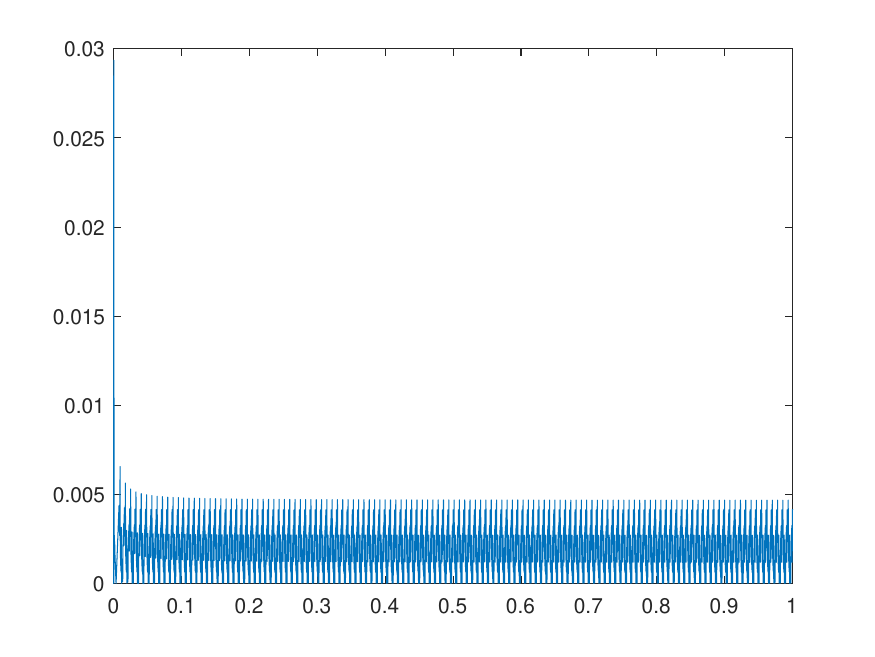}
	\end{subfigure}
	\caption{\cref{Special:ex2:infty}: The top row: $a$ in $[2^{-15},0.05]$ (first panel), $a$ in [0.05,0.5] (second panel), $a$ in [0.05,0.1] (third panel), and $u$ in [0,1] (fourth panel). The middle row: $u$ in [0,0.05] (first panel), $u'$ in [0,1] (second panel), $u'$ in [0,0.05] (third panel), and $au'$ in $[2^{-15},1]$ (fourth panel). The bottom row: $|\tilde{u}_H-u|$ (first panel)  in [0,1], $|\tilde{u}'_H-u'|$ (second panel)  in [0,1], and  $|a\tilde{u}'_H-au'|$ (third panel) in $[2^{-15},1]$ with $H=\frac{1}{2^7}$ at all $x_i=i/N$,  $i=0, 1,\dots, N$ with $N=2^{15}$ in \eqref{fine:xi}.}
	\label{fig:exam2:infty}
\end{figure}	

	\subsection{The exact solution $u$ is not available}
\begin{example}\label{Special:ex3}
	\normalfont
	Consider $f(x)=\sin(2x)\cos(x)$ and a continuous coefficient
\[
a(x)=	2\exp(1) +\sin(2^9\pi x)\exp(x^2)\cos(10x),
\]
with the high-frequency oscillation. Note that the exact solution $u$ of \eqref{Model:Problem} is not available.
The results on the numerical solution $\tilde{u}_H$ are presented in \cref{table1:example3,table2:example3,fig:exam3}.
\end{example}
\begin{table}[htbp]
\caption{Performance of the $l_2$ norm of the error in \cref{Special:ex3} of our proposed derivative-orthogonal wavelet multiscale method  on the uniform Cartesian mesh $H$. Note that $N=2^{14}$ in \eqref{fine:xi} for estimating the $l_2$ norm.}
\centering
	\renewcommand\arraystretch{0.7}
\setlength{\tabcolsep}{3mm}{
	\begin{tabular}{c|c|c|c|c|c|c|c|c}
		\hline
		$H$
		&   $\frac{\|\tilde{u}_H-\tilde{u}_{H/2}\|_2}{\|\tilde{u}_{H/2}\|_2}$
		& order &  $\frac{\|\tilde{u}'_H-\tilde{u}'_{H/2}\|_2}{\|\tilde{u}'_{H/2}\|_2}$ & order &   $\frac{\|a\tilde{u}'_H-a\tilde{u}'_{H/2}\|_2}{\|a\tilde{u}'_{H/2}\|_2}$
		&order & $\kappa$ & $\frac{a_{\max}}{a_{\min}}$ \\
		\hline	
$1/2$   &2.3154E-01   &   &4.5360E-01   &   &4.5866E-01   &   &1.54   &2.67\\
$1/2^2$   &4.7958E-02   &2.27   &2.0709E-01   &1.13   &2.1327E-01   &1.10   &1.81   &2.67\\
$1/2^3$   &1.1535E-02   &2.06   &1.0093E-01   &1.04   &1.0325E-01   &1.05   &1.93   &2.67\\
$1/2^4$   &3.1247E-03   &1.88   &5.0224E-02   &1.01   &5.1273E-02   &1.01   &1.97   &2.67\\
$1/2^5$   &8.2435E-04   &1.92   &2.5085E-02   &1.00   &2.5615E-02   &1.00   &1.99   &2.67\\
$1/2^6$   &2.1331E-04   &1.95   &1.2539E-02   &1.00   &1.2807E-02   &1.00   &1.99   &2.67\\		 
		\hline
\end{tabular}}
\label{table1:example3}
\end{table}	

\begin{table}[htbp]
\caption{Performance of the $l_\infty$ norm of the error in \cref{Special:ex3} of our proposed derivative-orthogonal wavelet multiscale method  on the uniform Cartesian mesh $H$. Note that $N=2^{14}$ in \eqref{fine:xi} for estimating the $l_\infty$ norm.}
\centering
	\renewcommand\arraystretch{0.7}
\setlength{\tabcolsep}{1mm}{
	\begin{tabular}{c|c|c|c|c|c|c|c|c}
		\hline
		$H$
		&   $\|\tilde{u}_H-\tilde{u}_{H/2}\|_\infty$
		& order &  $\|\tilde{u}'_H-\tilde{u}'_{H/2}\|_\infty$ & order &   $\|a\tilde{u}'_H-a\tilde{u}'_{H/2}\|_\infty$ & order  & $\kappa$ & $\frac{a_{\max}}{a_{\min}}$ \\
		\hline	
$1/2$   &4.2314E-03   &   &2.5856E-02   &   &1.0052E-01   &   &1.54   &2.67\\
$1/2^2$   &1.0619E-03   &1.99   &1.5628E-02   &0.73   &6.3632E-02   &0.66   &1.81   &2.67\\
$1/2^3$  &2.7613E-04   &1.94   &6.6146E-03   &1.24   &2.6730E-02   &1.25   &1.93   &2.67\\
$1/2^4$   &8.7980E-05   &1.65   &3.0974E-03   &1.09   &1.2096E-02   &1.14   &1.97   &2.67\\
$1/2^5$   &2.4401E-05   &1.85   &1.5457E-03   &1.00   &6.5000E-03   &0.90   &1.99   &2.67\\
$1/2^6$   &6.4601E-06   &1.92   &7.7127E-04   &1.00   &3.6661E-03   &0.83   &1.99   &2.67\\		 
		\hline
\end{tabular}}
\label{table2:example3}
\end{table}	
\begin{figure}[htbp]
\centering
\begin{subfigure}[b]{0.23\textwidth}
	 \includegraphics[width=4.5cm,height=3.cm]{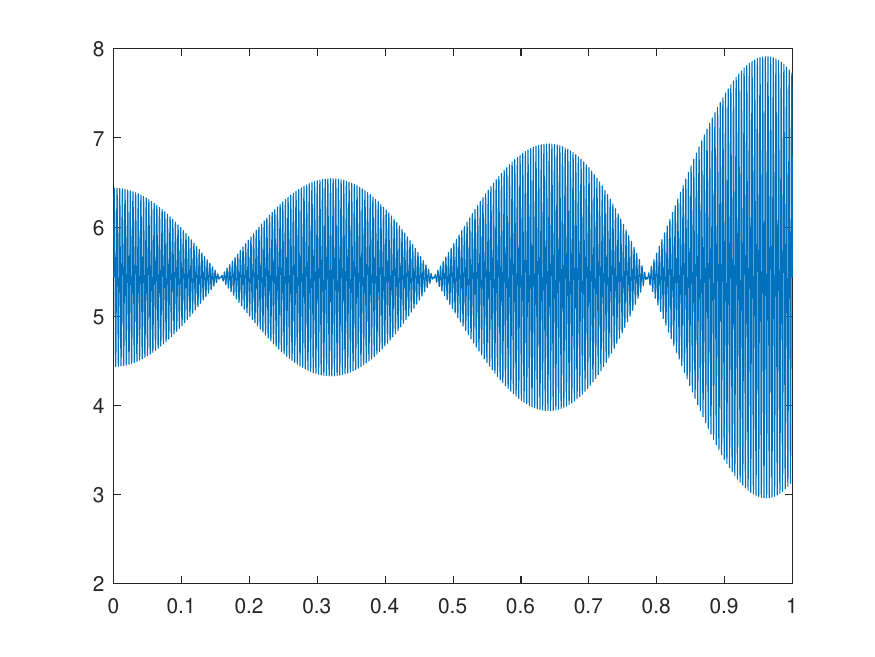}
\end{subfigure}
\begin{subfigure}[b]{0.23\textwidth}
	 \includegraphics[width=4.5cm,height=3.cm]{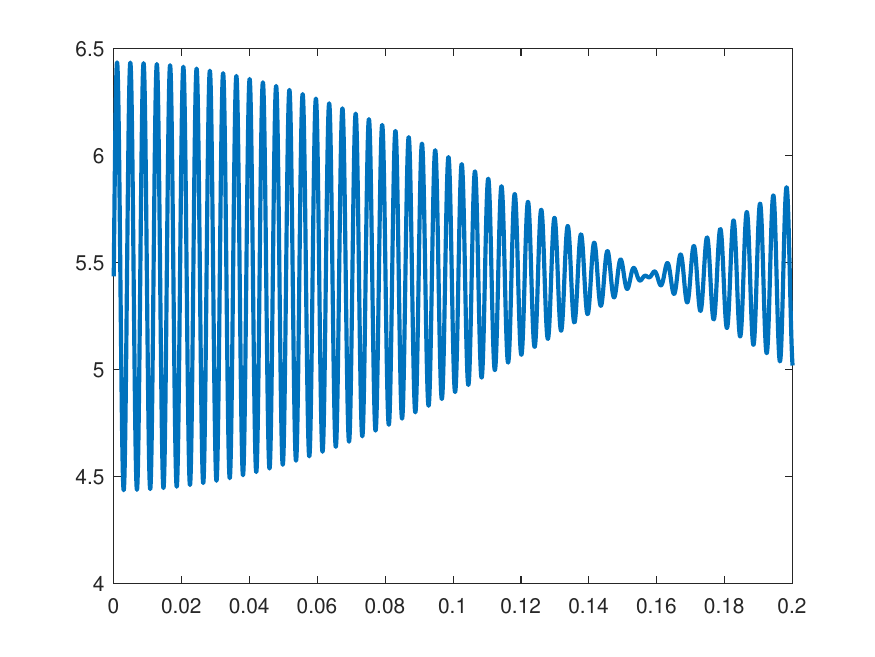}
\end{subfigure}
\begin{subfigure}[b]{0.23\textwidth}
	 \includegraphics[width=4.5cm,height=3.cm]{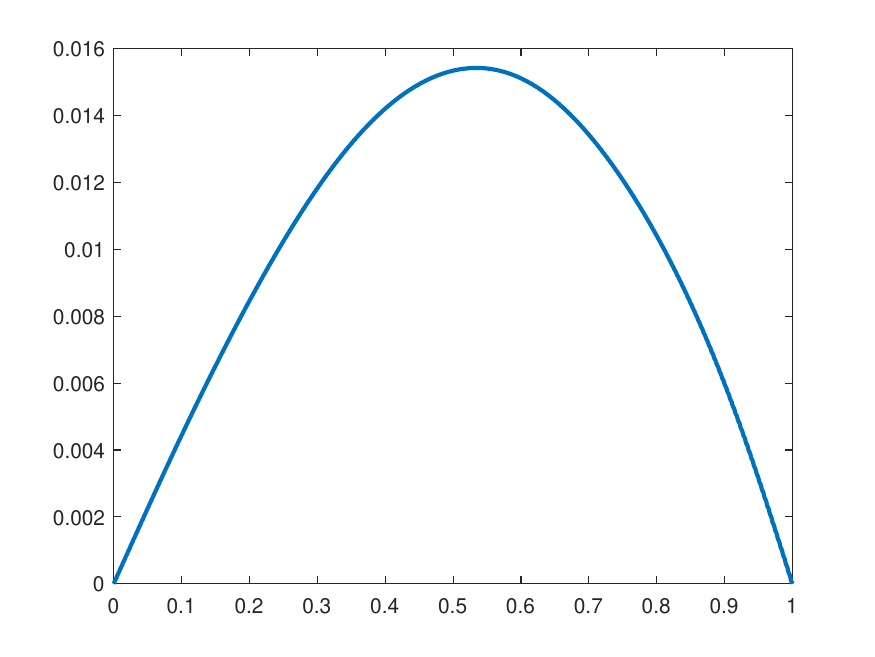}
\end{subfigure}
\begin{subfigure}[b]{0.23\textwidth}
	 \includegraphics[width=4.5cm,height=3.cm]{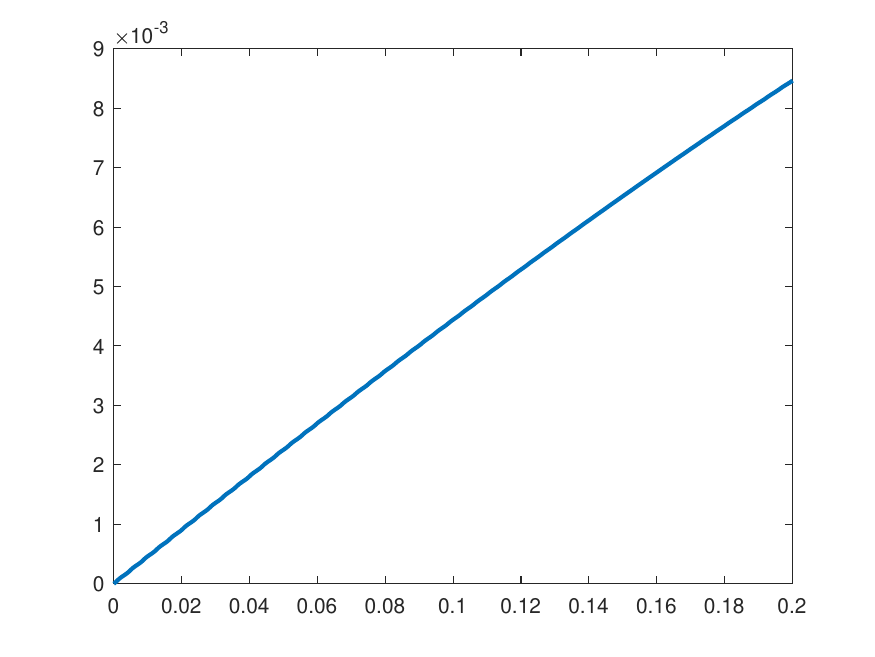}
\end{subfigure}
\begin{subfigure}[b]{0.23\textwidth}
	 \includegraphics[width=4.5cm,height=3.cm]{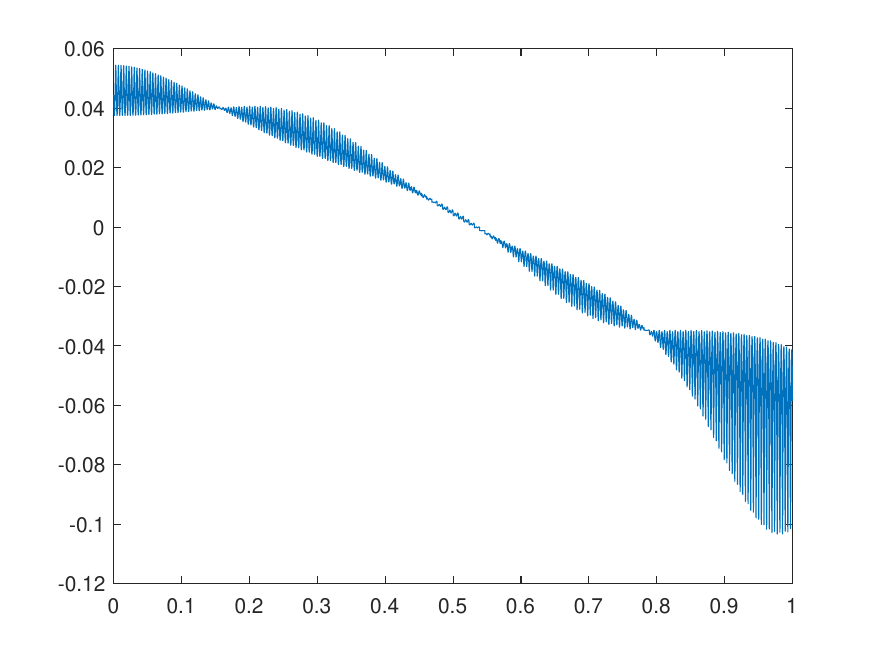}
\end{subfigure}
\begin{subfigure}[b]{0.23\textwidth}
	 \includegraphics[width=4.5cm,height=3.cm]{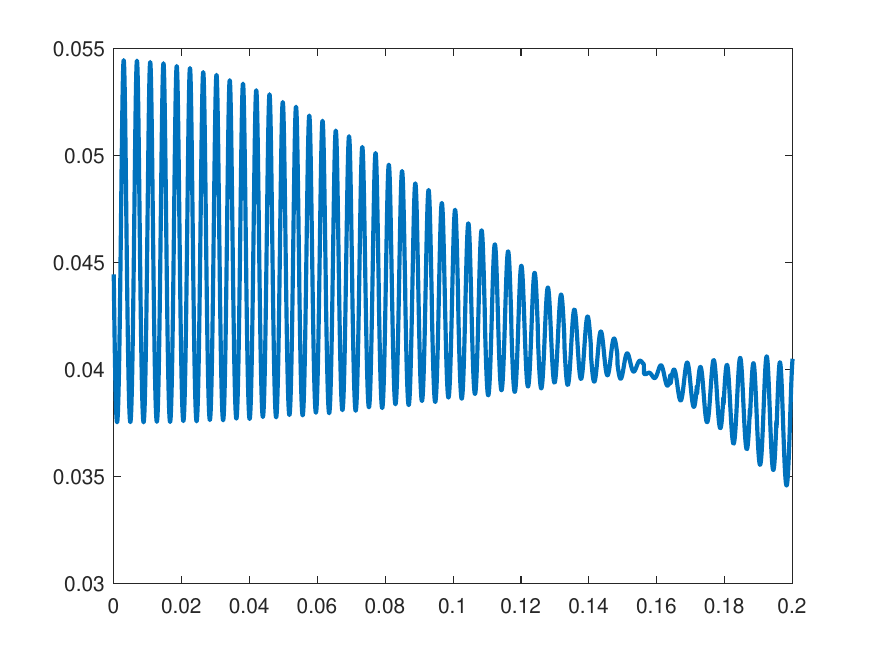}
\end{subfigure}
\begin{subfigure}[b]{0.23\textwidth}
	 \includegraphics[width=4.5cm,height=3.cm]{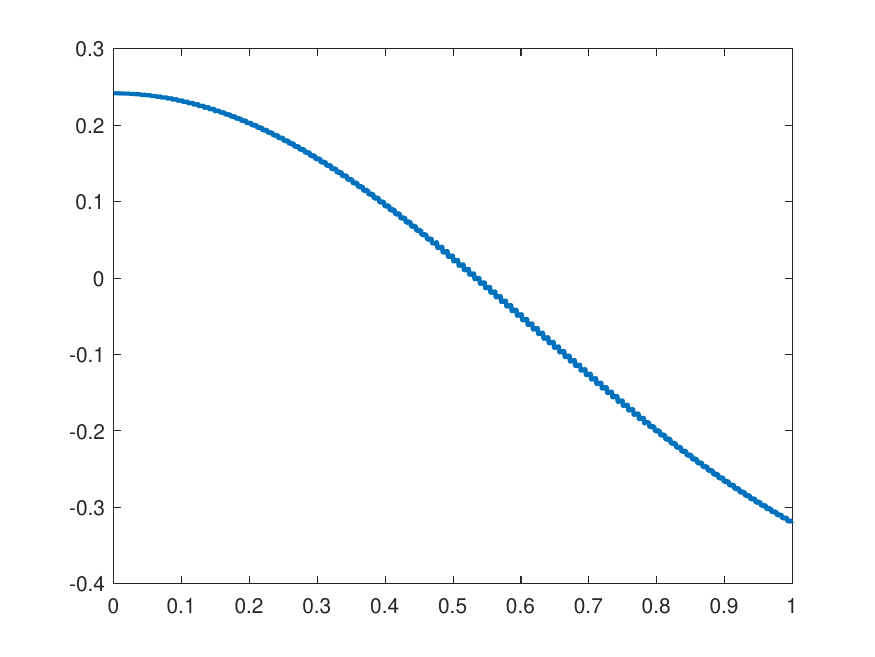}
\end{subfigure}
\begin{subfigure}[b]{0.23\textwidth}
	 \includegraphics[width=4.5cm,height=3.cm]{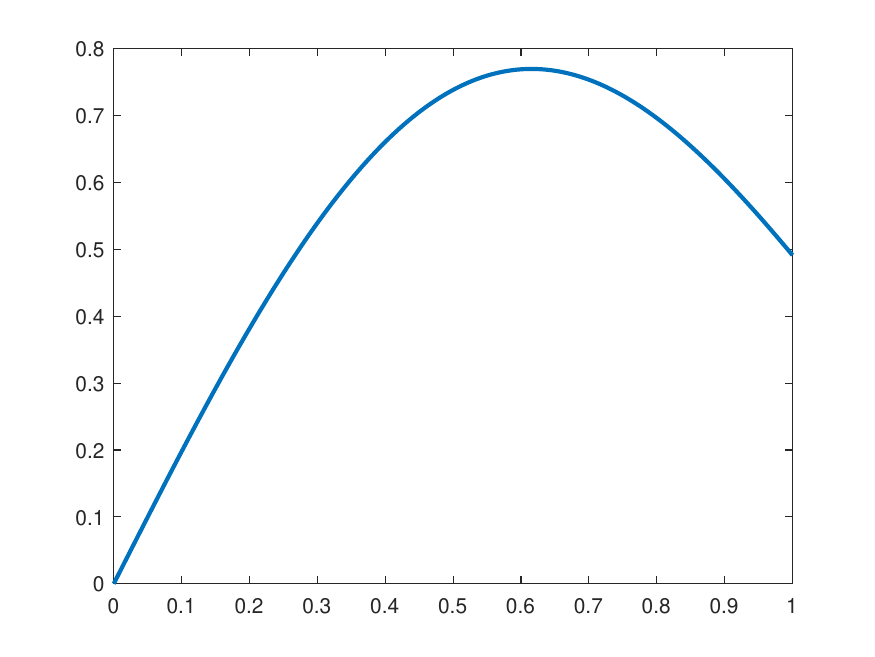}
\end{subfigure}
\begin{subfigure}[b]{0.3\textwidth}
	 \includegraphics[width=6cm,height=3cm]{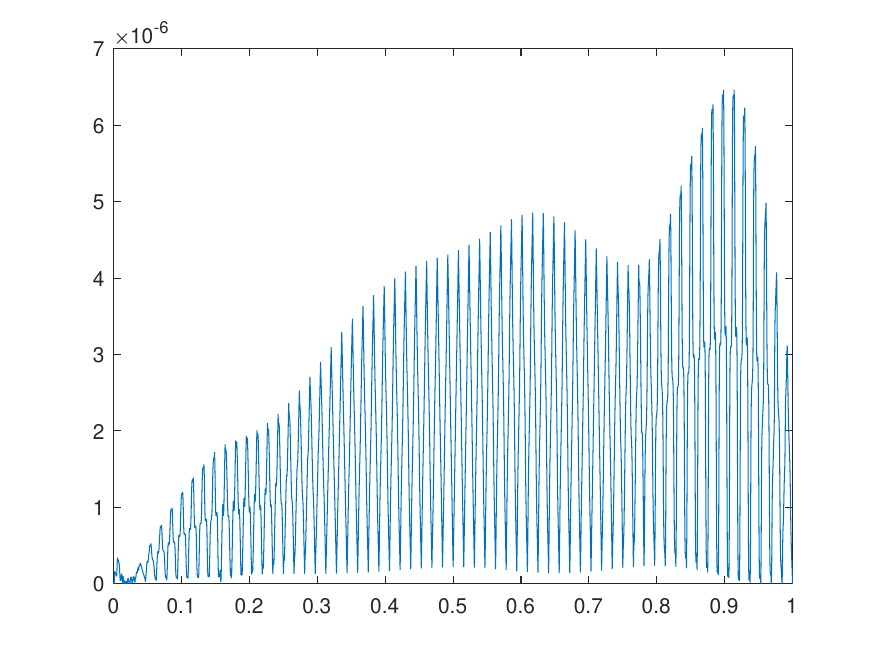}
\end{subfigure}
\begin{subfigure}[b]{0.3\textwidth}
	 \includegraphics[width=6cm,height=3cm]{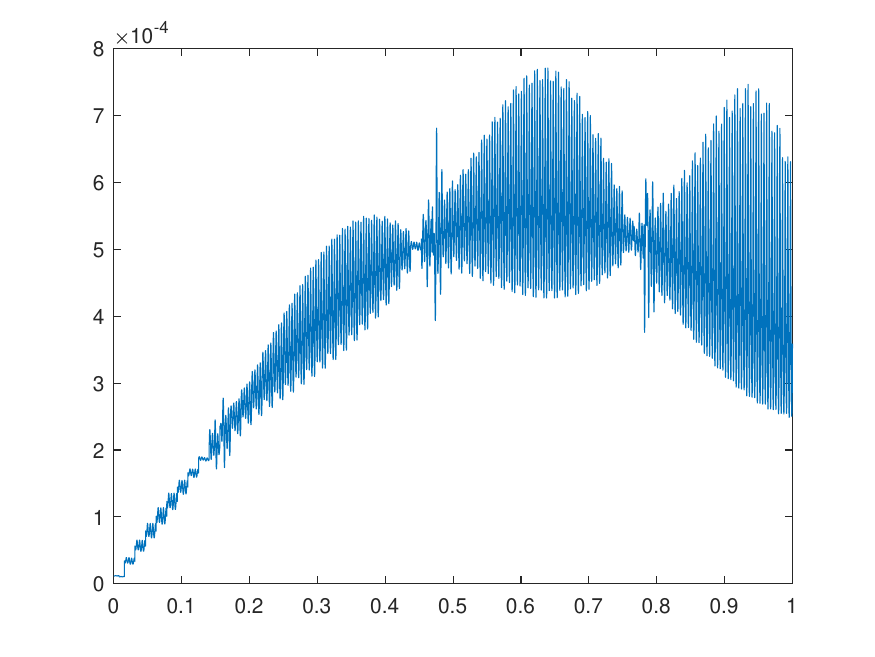}
\end{subfigure}
\begin{subfigure}[b]{0.3\textwidth}
	 \includegraphics[width=6cm,height=3cm]{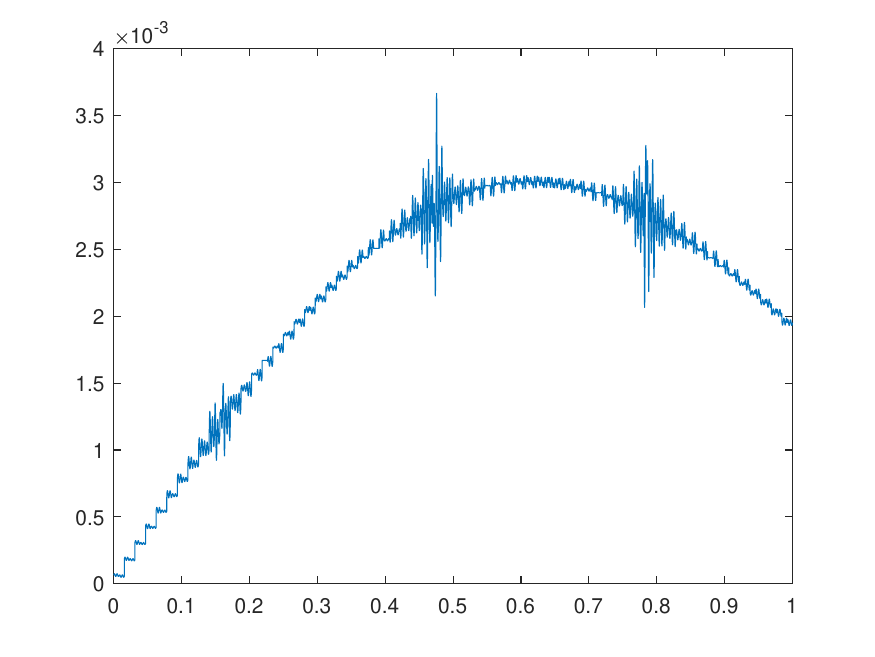}
\end{subfigure}
\caption{\cref{Special:ex3}: The top row: $a$ in [0,1] (first panel), $a$ in [0,0.2] (second panel), $\tilde{u}_{H}$ in [0,1] (third panel), and $\tilde{u}_{H}$ in [0,0.2] (fourth panel). The middle row: $\tilde{u}'_{H}$ in [0,1] (first panel), $\tilde{u}'_{H}$ in [0,0.2] (second panel), $a\tilde{u}'_{H}$ in [0,1] (third panel) with $H=\frac{1}{2^7}$, and $f$ in [0,1] (fourth panel). The bottom row: $|\tilde{u}_{H}-\tilde{u}_{H/2}|$ (first panel), $|\tilde{u}'_{H}-\tilde{u}'_{H/2}|$ (second panel), and  $|a\tilde{u}'_{H}-a\tilde{u}'_{H/2}|$  (third panel) in [0,1] with $H=\frac{1}{2^6}$ at all $x_i=i/N$ for $i=0,\dots, N$ with $N=2^{14}$ in \eqref{fine:xi}.}
\label{fig:exam3}
\end{figure}

\begin{example}\label{Special:ex4}
	\normalfont
	Consider $f(x)=\cos(4x)$ and a discontinuous diffusion coefficient $a$ with high-contrast jumps below:
	\[
	a(x)= \left( 10^4\chi_{[\frac{i-1}{2^8},\frac{i}{2^8})}
	+10^{-4}\chi_{[\frac{i}{2^8}, \frac{i+1}{2^8})} \right)\exp(x^2+1)\left(10+\cos(20x)\right) \quad \text{with} \quad i=1,3,5,\dots,2^8-1.
	\]
	Note that the exact solution $u$ of \eqref{Model:Problem} is not available.
The results on the numerical solution $\tilde{u}_H$ are presented in \cref{table1:example4,table2:example4,fig:exam4}.
\end{example}
\begin{table}[htbp]
	\caption{Performance of the $l_2$ norm of the error in \cref{Special:ex4} of our proposed derivative-orthogonal wavelet multiscale method on the uniform Cartesian mesh $H$. Note that $N=2^{14}$ in \eqref{fine:xi} for estimating the $l_2$ norm. }
	\centering
	\renewcommand\arraystretch{0.7}	
	\setlength{\tabcolsep}{1mm}{
		 \begin{tabular}{c|c|c|c|c|c|c|c|c}
			\hline
			$H$
			&   $\frac{\|\tilde{u}_H-\tilde{u}_{H/2}\|_2}{\|\tilde{u}_{H/2}\|_2}$
			& order &  $\frac{\|\tilde{u}'_H-\tilde{u}'_{H/2}\|_2}{\|\tilde{u}'_{H/2}\|_2}$ & order &   $\frac{\|a\tilde{u}'_H-a\tilde{u}'_{H/2}\|_2}{\|a\tilde{u}'_{H/2}\|_2}$
			&order & $\kappa$ & $\frac{a_{\max}}{a_{\min}}$ \\
			\hline	
$1/2$   &8.1846E-01   &   &1.6153E+00   &   &1.1281E+00   &   &1.1951E+08   &3.0665E+08\\
$1/2^2$   &1.7618E-01   &2.22   &4.9703E-01   &1.70   &3.8735E-01   &1.54   &1.2820E+08   &3.0665E+08\\
$1/2^3$   &3.8195E-02   &2.21   &2.3071E-01   &1.11   &1.8159E-01   &1.09   &1.3075E+08   &3.0665E+08\\
$1/2^4$   &8.6474E-03   &2.14   &1.1397E-01   &1.02   &8.9408E-02   &1.02   &1.3136E+08   &3.0665E+08\\
$1/2^5$   &2.1766E-03   &1.99   &5.6825E-02   &1.00   &4.4544E-02   &1.01   &1.3151E+08   &3.0665E+08\\
$1/2^6$   &5.9792E-04   &1.86   &2.8394E-02   &1.00   &2.2252E-02   &1.00   &1.3155E+08   &3.0665E+08\\
			\hline
	\end{tabular}}
	\label{table1:example4}
\end{table}	

\begin{table}[htbp]
	\caption{Performance of the $l_\infty$ norm of the error in \cref{Special:ex4} of our proposed derivative-orthogonal wavelet multiscale method  on the uniform Cartesian mesh $H$. Note that $N=2^{14}$ in \eqref{fine:xi} for estimating the $l_\infty$ norm.}
	\centering
	\renewcommand\arraystretch{0.7}	
	\setlength{\tabcolsep}{0.1mm}{
		 \begin{tabular}{c|c|c|c|c|c|c|c|c}
			\hline
			$H$
			&   $\|\tilde{u}_H-\tilde{u}_{H/2}\|_\infty$
			& order &  $\|\tilde{u}'_H-\tilde{u}'_{H/2}\|_\infty$ & order &   $\|a\tilde{u}'_H-a\tilde{u}'_{H/2}\|_\infty$ & order  & $\kappa$ & $\frac{a_{\max}}{a_{\min}}$ \\
			\hline	
	$1/2$   &3.0071E+00   &   &2.9791E+01   &   &1.3141E-01   &   &1.1951E+08   &3.0665E+08\\
	$1/2^2$   &9.5963E-01   &1.65   &2.4180E+01   &0.30   &7.0859E-02   &0.89   &1.2820E+08   &3.0665E+08\\
	$1/2^3$   &2.4780E-01   &1.95   &1.2500E+01   &0.95   &3.3683E-02   &1.07   &1.3075E+08   &3.0665E+08\\
	$1/2^4$   &6.9096E-02   &1.84   &5.7479E+00   &1.12   &1.6012E-02   &1.07   &1.3136E+08   &3.0665E+08\\
	$1/2^5$   &1.8739E-02   &1.88   &2.7866E+00   &1.04   &7.8678E-03   &1.03   &1.3151E+08   &3.0665E+08\\
	$1/2^6$   &4.8165E-03   &1.96   &1.3783E+00   &1.02   &3.9355E-03   &1.00   &1.3155E+08   &3.0665E+08\\
			\hline
	\end{tabular}}
	\label{table2:example4}
\end{table}	

\begin{figure}[htbp]
	\centering
	\begin{subfigure}[b]{0.23\textwidth}
		 \includegraphics[width=4.5cm,height=3.cm]{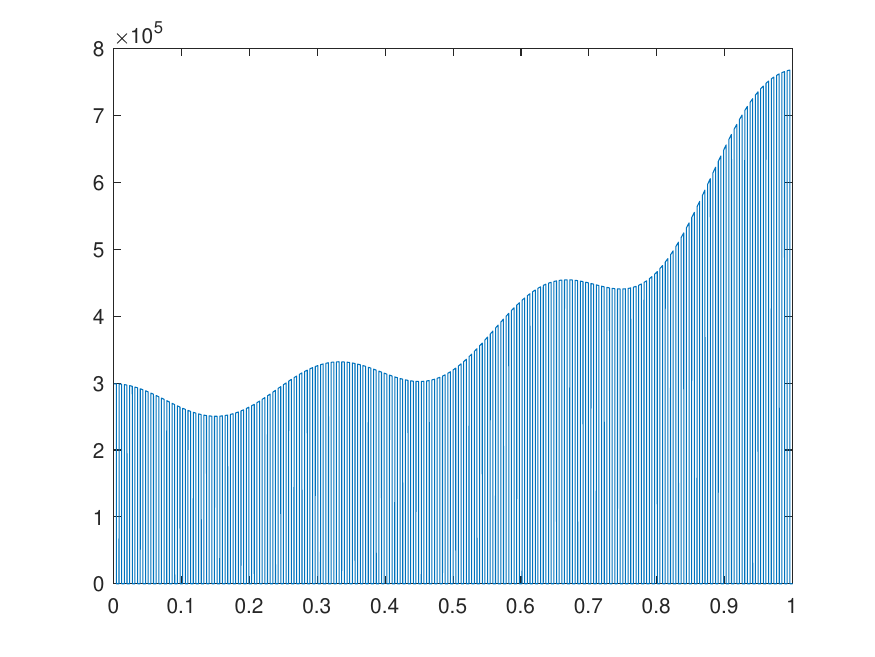}
	\end{subfigure}
	\begin{subfigure}[b]{0.23\textwidth}
		 \includegraphics[width=4.5cm,height=3.cm]{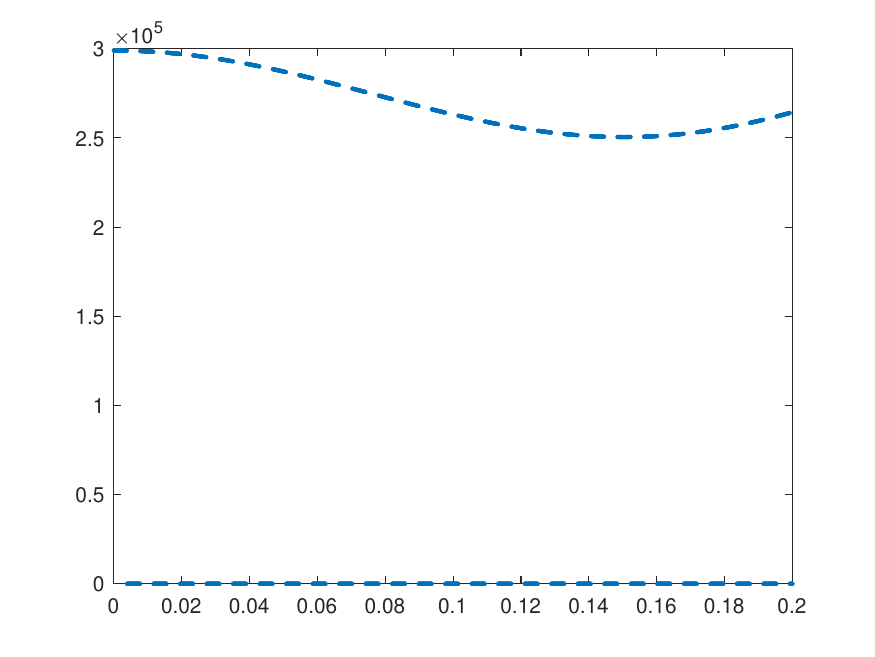}
	\end{subfigure}
	\begin{subfigure}[b]{0.23\textwidth}
		 \includegraphics[width=4.5cm,height=3.cm]{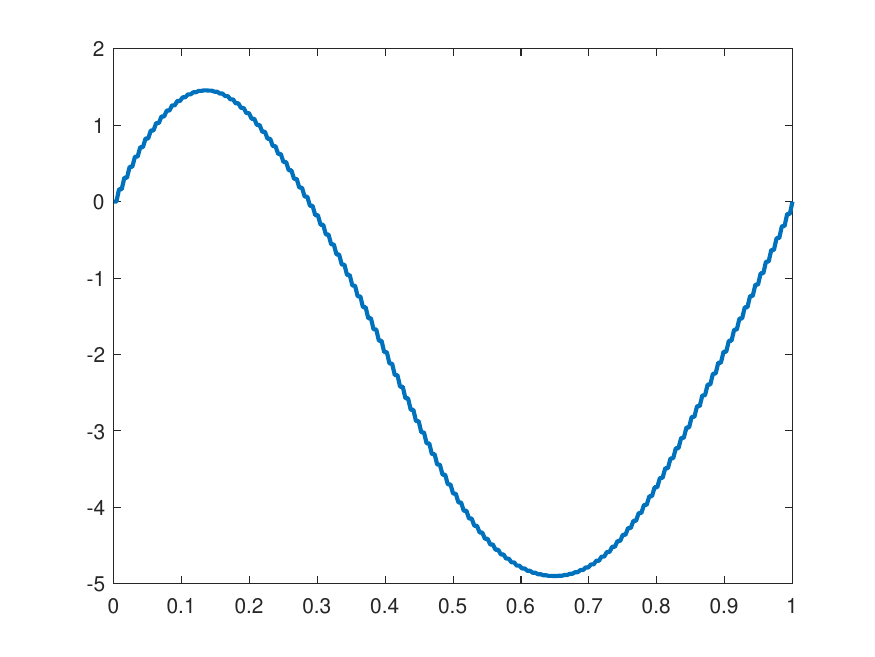}
	\end{subfigure}
	\begin{subfigure}[b]{0.23\textwidth}
		 \includegraphics[width=4.5cm,height=3.cm]{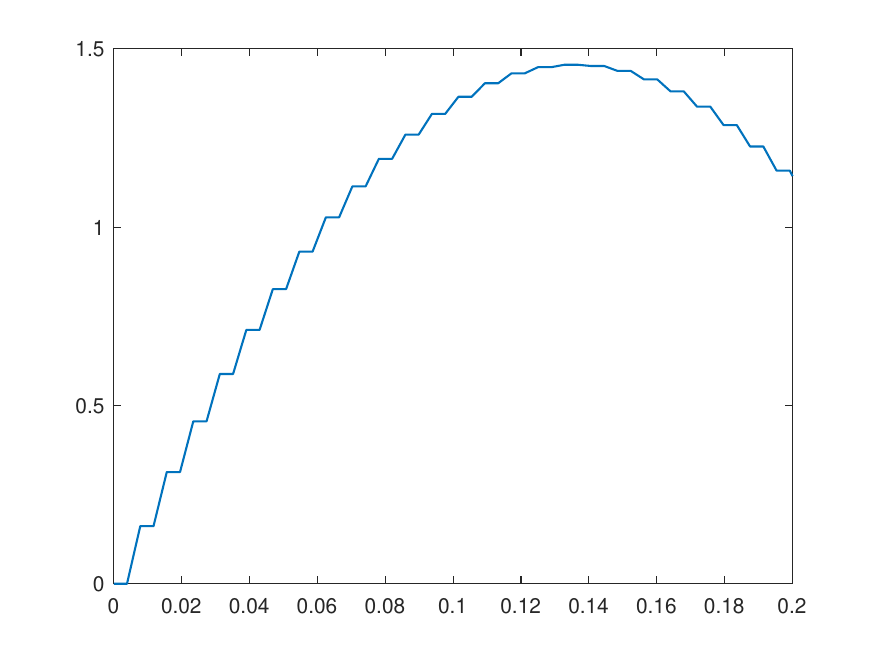}
	\end{subfigure}
	\begin{subfigure}[b]{0.23\textwidth}
		 \includegraphics[width=4.5cm,height=3.cm]{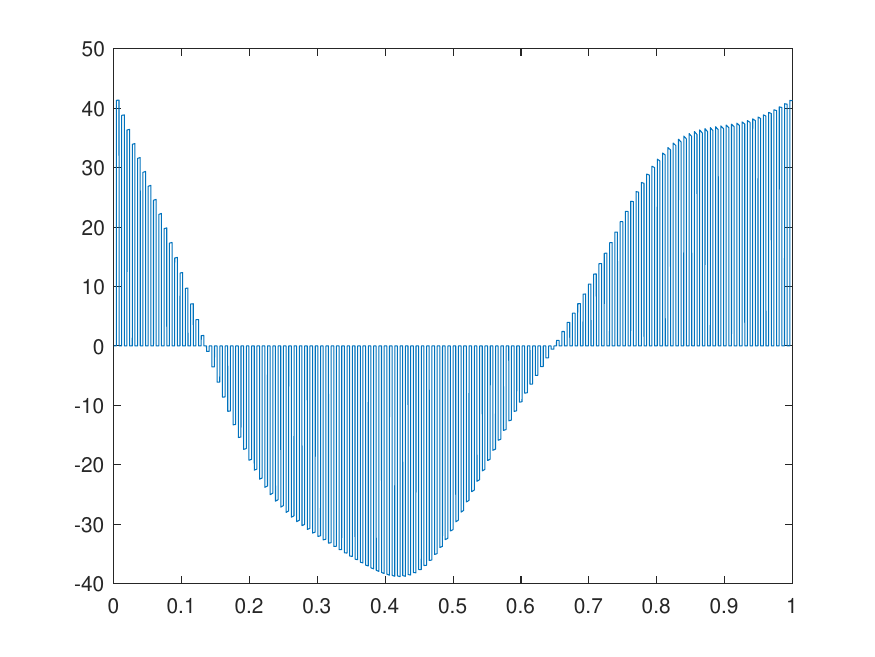}
	\end{subfigure}
	\begin{subfigure}[b]{0.23\textwidth}
		 \includegraphics[width=4.5cm,height=3.cm]{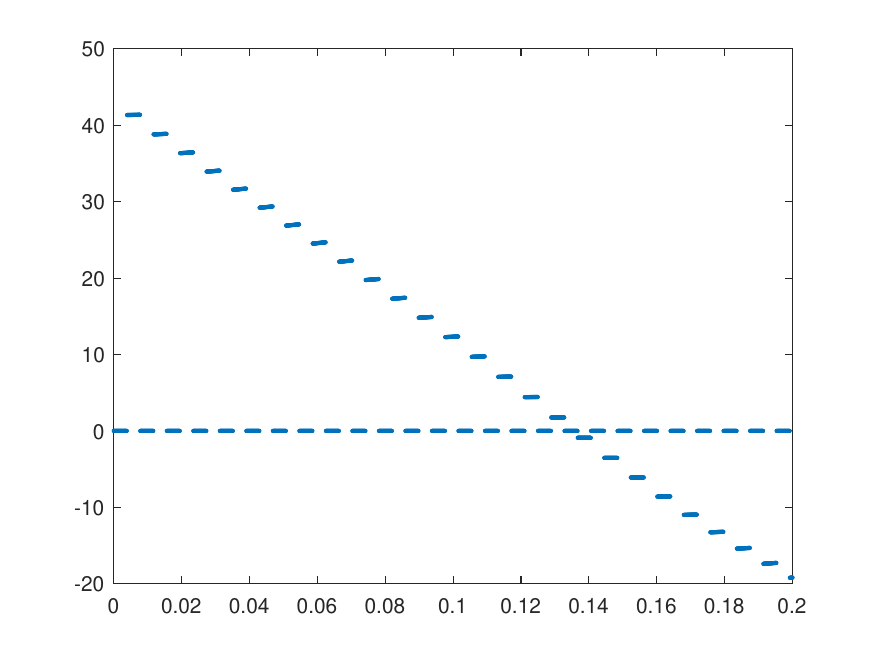}
	\end{subfigure}
	\begin{subfigure}[b]{0.23\textwidth}
		 \includegraphics[width=4.5cm,height=3.cm]{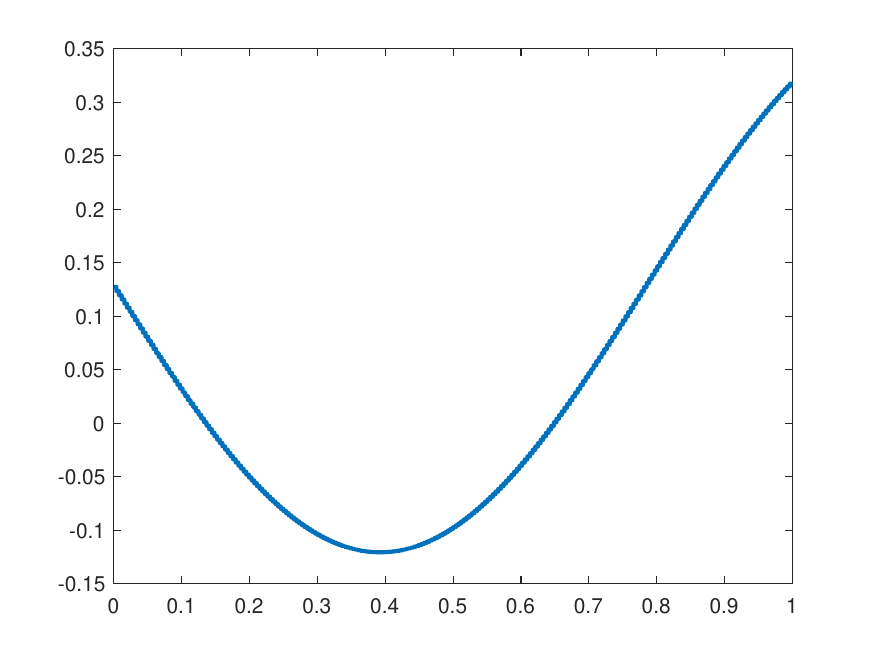}
	\end{subfigure}
	\begin{subfigure}[b]{0.23\textwidth}
		 \includegraphics[width=4.5cm,height=3.cm]{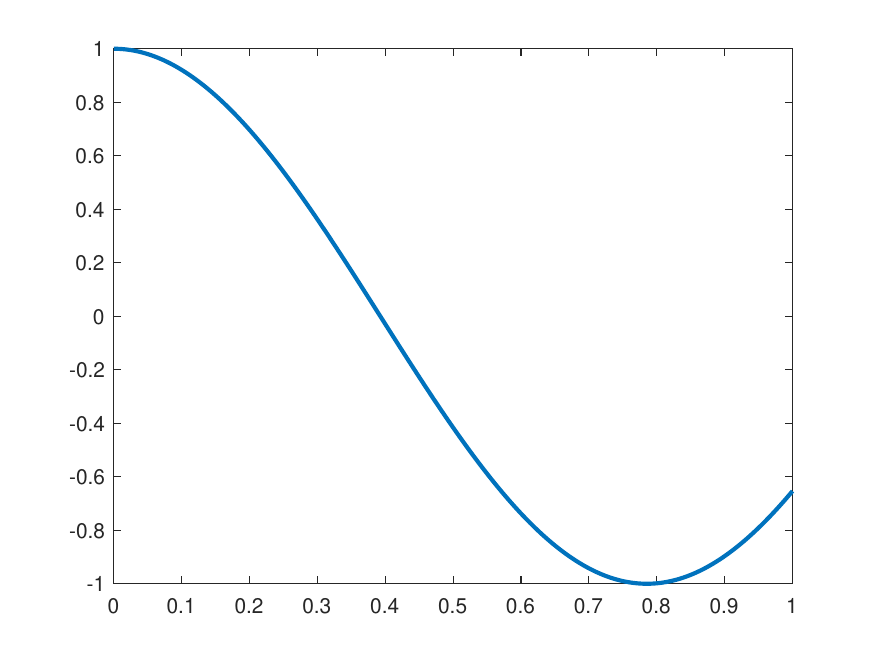}
	\end{subfigure}
	\begin{subfigure}[b]{0.3\textwidth}
		 \includegraphics[width=6cm,height=3cm]{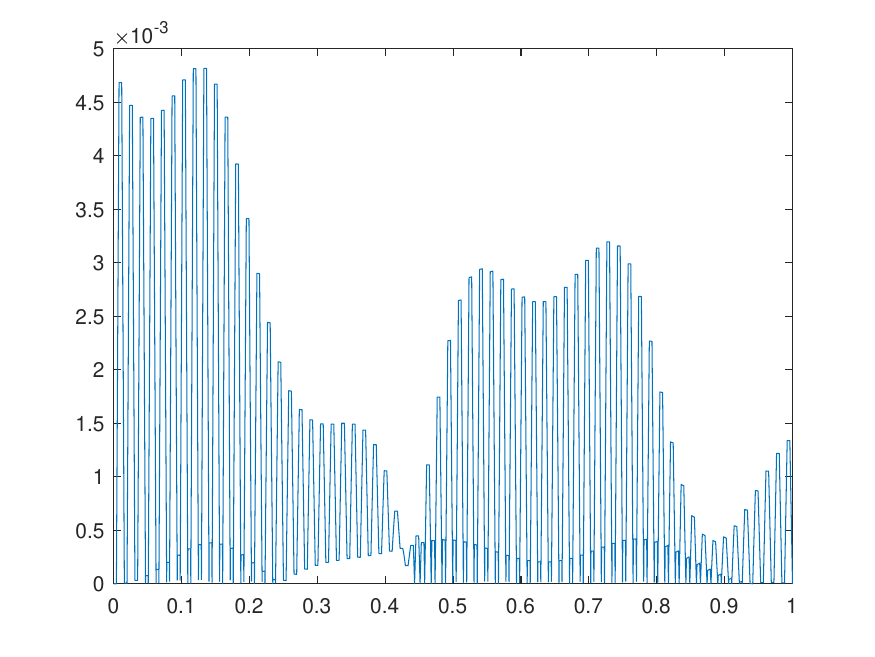}
	\end{subfigure}
	\begin{subfigure}[b]{0.3\textwidth}
		 \includegraphics[width=6cm,height=3cm]{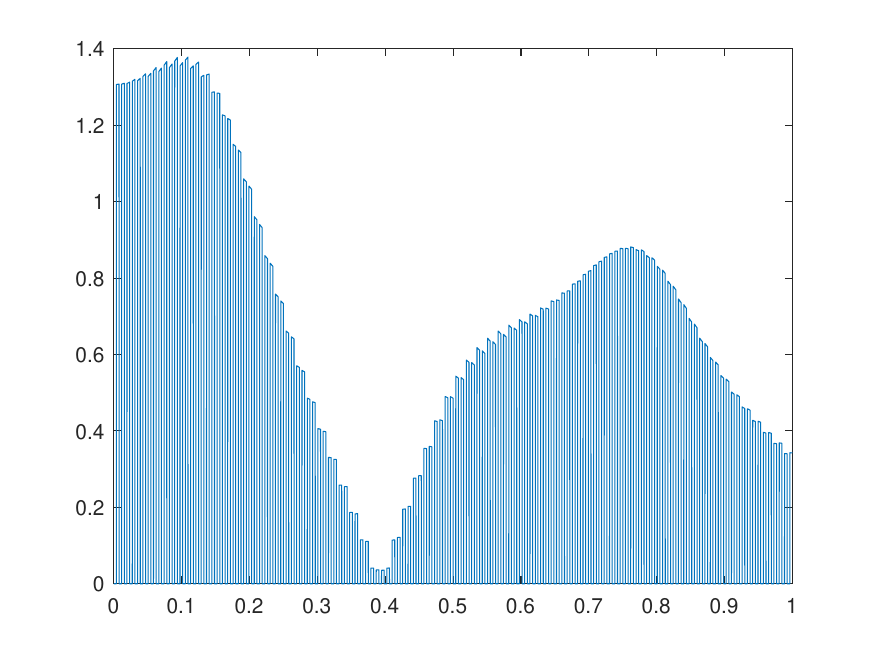}
	\end{subfigure}
	\begin{subfigure}[b]{0.3\textwidth}
		 \includegraphics[width=6cm,height=3cm]{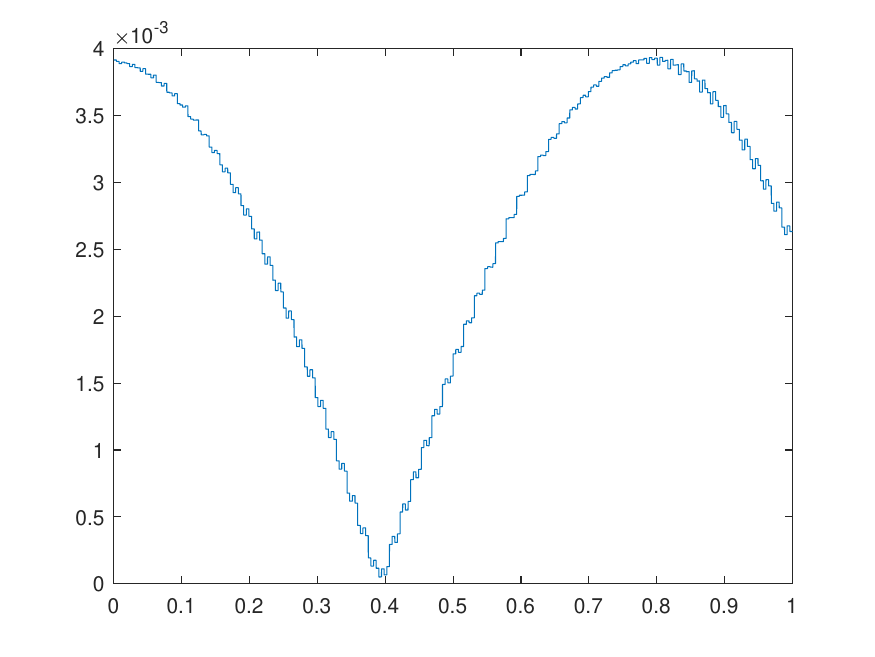}
	\end{subfigure}
	\caption{\cref{Special:ex4}: The top row: $a$ in [0,1] (first panel), $a$ in [0,0.2] (second panel), $\tilde{u}_{H}$ in [0,1] (third panel), and $\tilde{u}_{H}$ in [0,0.2] (fourth panel). The middle row: $\tilde{u}'_{H}$ in [0,1] (first panel), $\tilde{u}'_{H}$ in [0,0.2] (second panel), $a\tilde{u}'_{H}$ in [0,1] (third panel) with $H=\frac{1}{2^7}$, and $f$ in [0,1] (fourth panel). The bottom row: $|\tilde{u}_{H}-\tilde{u}_{H/2}|$ (first panel), $|\tilde{u}'_{H}-\tilde{u}'_{H/2}|$ (second panel), and  $|a\tilde{u}'_{H}-a\tilde{u}'_{H/2}|$ (third panel) in [0,1] with $H=\frac{1}{2^6}$ at all  $x_i=i/N$ for $i=0,\dots, N$ with $N=2^{14}$ in \eqref{fine:xi}.}
	\label{fig:exam4}
\end{figure}

\begin{example}\label{Special:ex6}
	\normalfont
	Consider a discontinuous source term $f(x)$ and a discontinuous diffusion coefficient $a(x)$ with high-contrast jumps below:
	\[
	\begin{split}
	& a(x)= \left( 10^4\chi_{[\frac{i-1}{2^8},\frac{i}{2^8})}
	+10^{-4}\chi_{[\frac{i}{2^8}, \frac{i+1}{2^8})} \right)\left(10+\sin(40x)\right) \quad \text{with} \quad i=1,3,5,\dots,2^8-1,\\
	& f(x)= \left( 1\chi_{[\frac{i-1}{2^8},\frac{i}{2^8})}
	+10^{-3}\chi_{[\frac{i}{2^8}, \frac{i+1}{2^8})} \right)\cos(10x) \quad \text{with} \quad i=1,3,5,\dots,2^8-1.
	\end{split}
	\]
	Note that the exact solution $u$ of \eqref{Model:Problem} is not available and $\|f\|_{L_2(\Omega)} \approx 0.511567 
$.
The results on the numerical solution $\tilde{u}_H$ are presented in \cref{table1:example6,table2:example6,fig:exam6}.
\end{example}
\begin{table}[htbp]
	\caption{Performance of the $l_2$ norm of the error in \cref{Special:ex6} of our proposed derivative-orthogonal wavelet multiscale method on the uniform Cartesian mesh $H$. Note that $N=2^{14}$ in \eqref{fine:xi} for estimating the $l_2$ norm. }
	\centering
	\renewcommand\arraystretch{0.7}	
	\setlength{\tabcolsep}{1mm}{
		\begin{tabular}{c|c|c|c|c|c|c|c|c}
			\hline
			$H$
			&   $\frac{\|\tilde{u}_H-\tilde{u}_{H/2}\|_2}{\|\tilde{u}_{H/2}\|_2}$
			& order &  $\frac{\|\tilde{u}'_H-\tilde{u}'_{H/2}\|_2}{\|\tilde{u}'_{H/2}\|_2}$ & order &   $\frac{\|a\tilde{u}'_H-a\tilde{u}'_{H/2}\|_2}{\|a\tilde{u}'_{H/2}\|_2}$
			&order & $\kappa$ & $\frac{a_{\max}}{a_{\min}}$ \\
			\hline	
$1/2$  &5.1083E+00  &  &1.0242E+01  &  &1.0776E+01  &  &1.0159E+08  &1.2222E+08 \\
$1/2^2$  &4.9138E-01  &3.38  &9.4325E-01  &3.44  &9.4667E-01  &3.51  &1.0232E+08  &1.2222E+08 \\
$1/2^3$   &1.0610E-01  &2.21  &3.5600E-01  &1.41  &3.5931E-01  &1.40  &1.0788E+08  &1.2222E+08 \\
$1/2^4$   &2.7615E-02  &1.94  &1.7106E-01  &1.06  &1.7199E-01  &1.06  &1.0933E+08  &1.2222E+08 \\
$1/2^5$   &6.6613E-03  &2.05  &8.4581E-02  &1.02  &8.5132E-02  &1.01  &1.0955E+08  &1.2222E+08 \\
$1/2^6$   &1.8051E-03  &1.88  &4.2198E-02  &1.00  &4.2466E-02  &1.00  &1.0963E+08  &1.2222E+08 \\
			\hline
	\end{tabular}}
	\label{table1:example6}
\end{table}	

\begin{table}[htbp]
	\caption{Performance of the $l_\infty$ norm of the error in \cref{Special:ex6} of our proposed derivative-orthogonal wavelet multiscale method  on the uniform Cartesian mesh $H$. Note that $N=2^{14}$ in \eqref{fine:xi} for estimating the $l_\infty$ norm.}
	\centering
		\renewcommand\arraystretch{0.7}
	\setlength{\tabcolsep}{0.1mm}{
		\begin{tabular}{c|c|c|c|c|c|c|c|c}
			\hline
			$H$
			&   $\|\tilde{u}_H-\tilde{u}_{H/2}\|_\infty$
			& order &  $\|\tilde{u}'_H-\tilde{u}'_{H/2}\|_\infty$ & order &   $\|a\tilde{u}'_H-a\tilde{u}'_{H/2}\|_\infty$ & order  & $\kappa$ & $\frac{a_{\max}}{a_{\min}}$ \\
			\hline	
$1/2$   &3.6544E+00  &  &3.2872E+01  &  &2.9586E-02  &  &1.0159E+08  &1.2222E+08 \\
$1/2^2$   &1.7343E+00  &1.08  &3.1000E+01  &0.08  &2.8427E-02  &0.06  &1.0232E+08  &1.2222E+08 \\
$1/2^3$   &4.8819E-01  &1.83  &1.6754E+01  &0.89  &1.6690E-02  &0.77  &1.0788E+08  &1.2222E+08 \\
$1/2^4$   &1.2856E-01  &1.92  &8.3965E+00  &1.00  &7.8405E-03  &1.09  &1.0933E+08  &1.2222E+08 \\
$1/2^5$   &3.6352E-02  &1.82  &4.2650E+00  &0.98  &3.9594E-03  &0.99  &1.0955E+08  &1.2222E+08 \\
$1/2^6$   &9.4345E-03  &1.95  &2.1079E+00  &1.02  &1.9999E-03  &0.99  &1.0963E+08  &1.2222E+08 \\
			\hline
	\end{tabular}}
	\label{table2:example6}
\end{table}	

\begin{figure}[htbp]
	\centering
	\begin{subfigure}[b]{0.23\textwidth}
		 \includegraphics[width=4.5cm,height=3.cm]{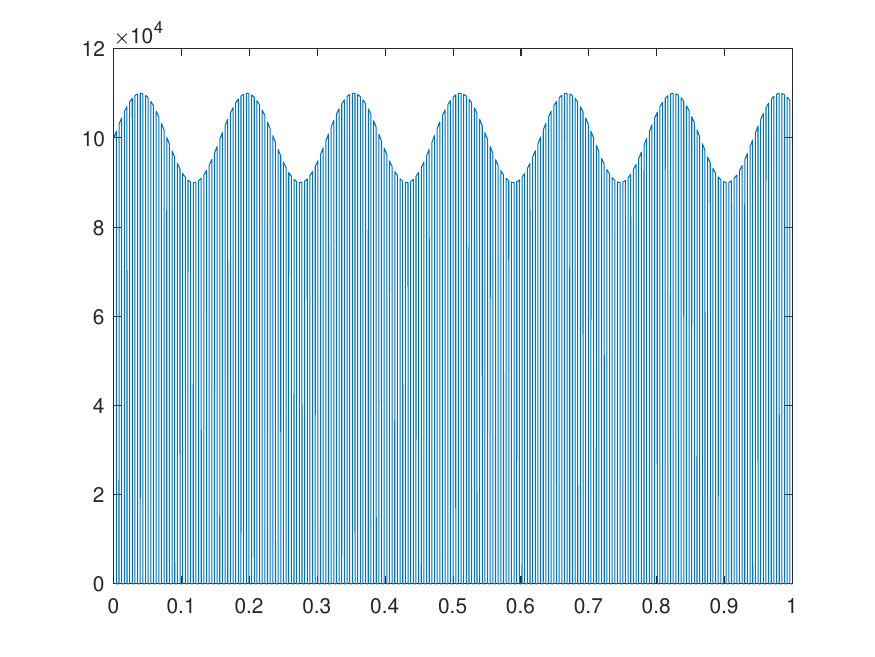}
	\end{subfigure}
	\begin{subfigure}[b]{0.23\textwidth}
		 \includegraphics[width=4.5cm,height=3.cm]{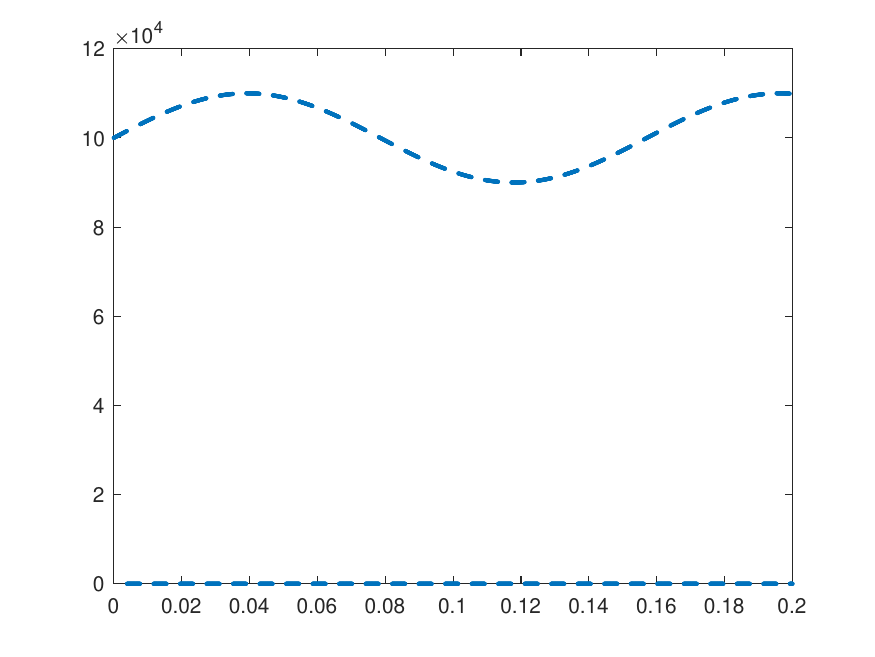}
	\end{subfigure}
	\begin{subfigure}[b]{0.23\textwidth}
		 \includegraphics[width=4.5cm,height=3.cm]{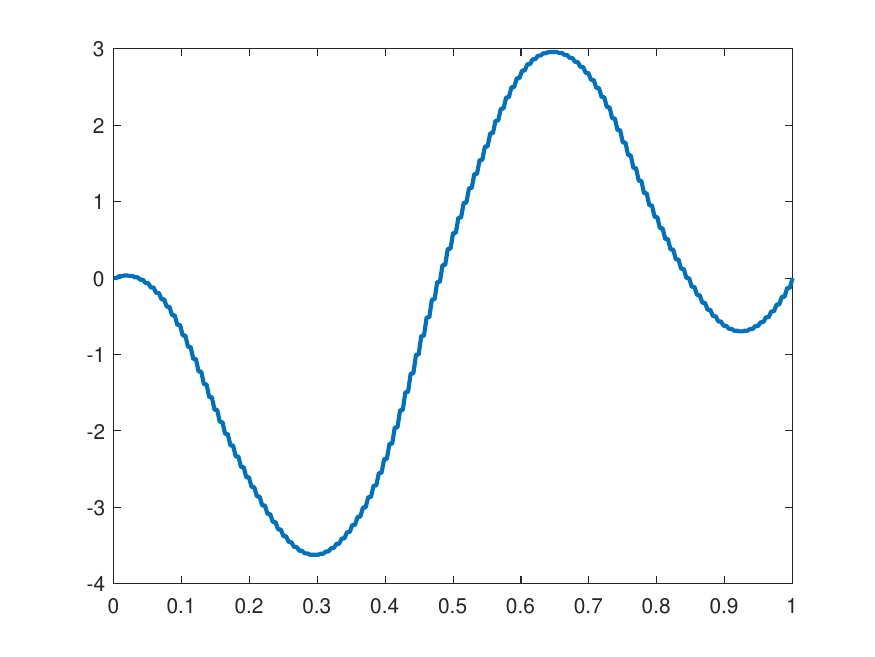}
	\end{subfigure}
	\begin{subfigure}[b]{0.23\textwidth}
		 \includegraphics[width=4.5cm,height=3.cm]{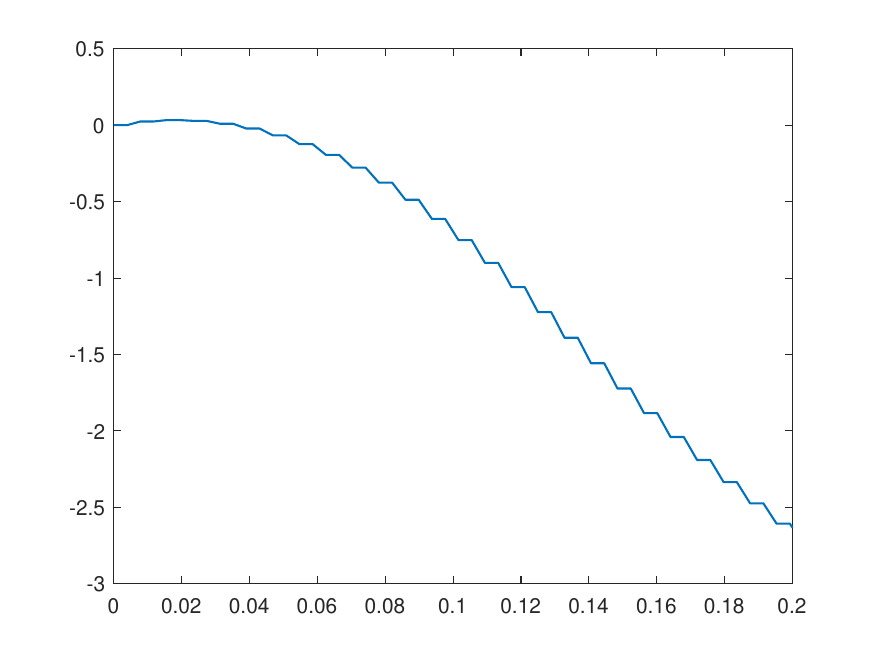}
	\end{subfigure}
	\begin{subfigure}[b]{0.23\textwidth}
		 \includegraphics[width=4.5cm,height=3.cm]{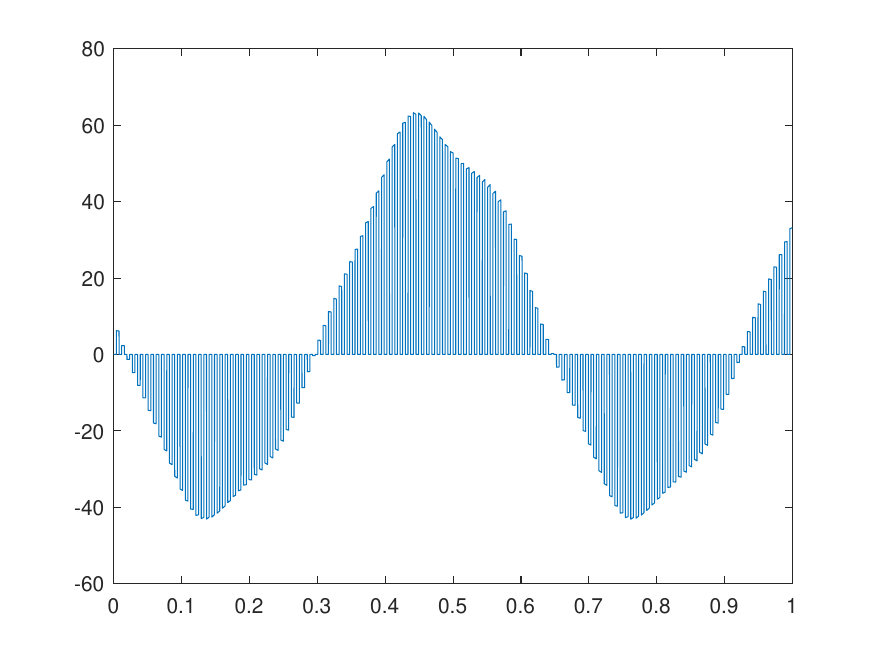}
	\end{subfigure}
	\begin{subfigure}[b]{0.23\textwidth}
		 \includegraphics[width=4.5cm,height=3.cm]{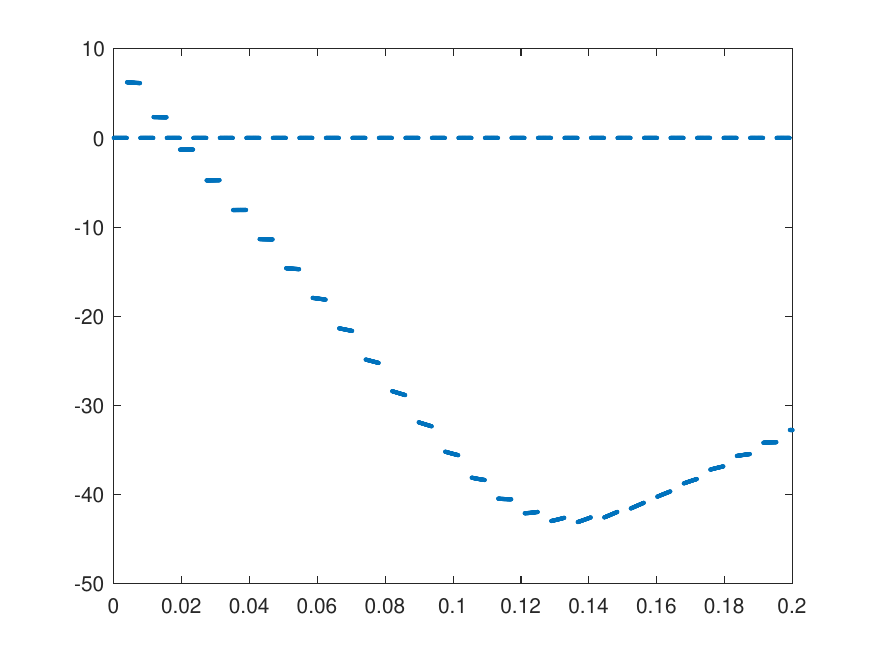}
	\end{subfigure}
	\begin{subfigure}[b]{0.23\textwidth}
		 \includegraphics[width=4.5cm,height=3.cm]{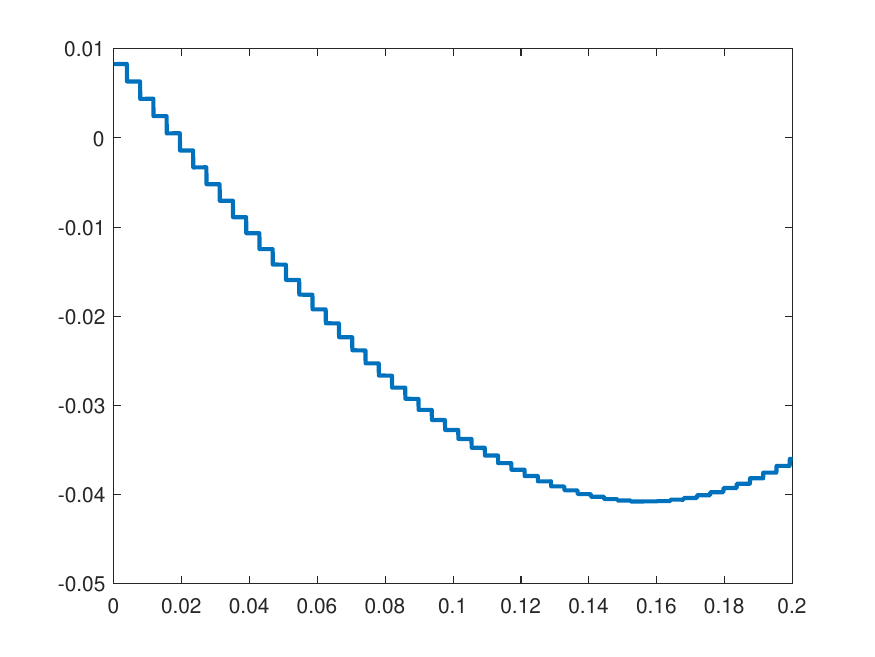}
	\end{subfigure}
	\begin{subfigure}[b]{0.23\textwidth}
		 \includegraphics[width=4.5cm,height=3.cm]{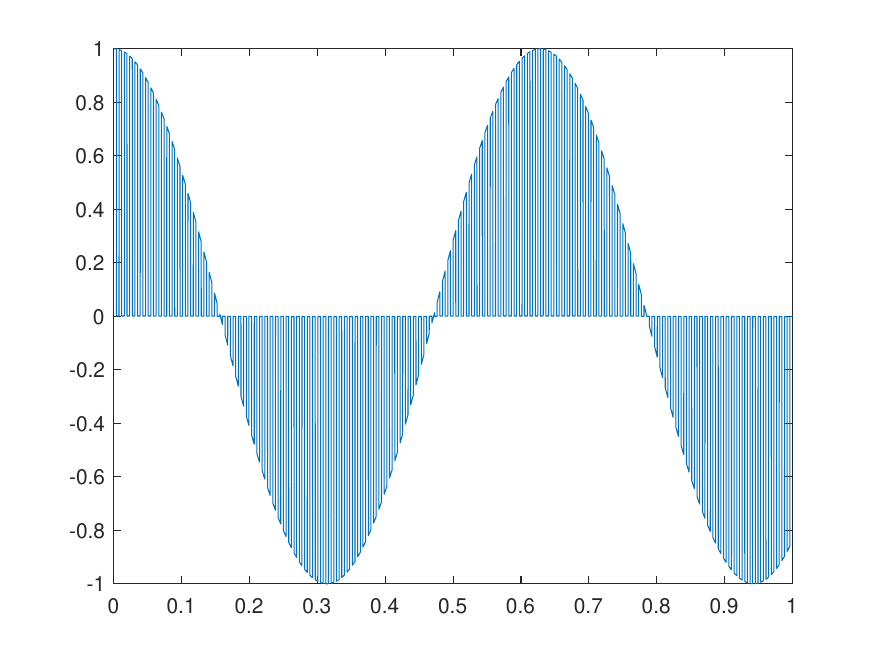}
	\end{subfigure}
	\begin{subfigure}[b]{0.23\textwidth}
		 \includegraphics[width=4.5cm,height=3.cm]{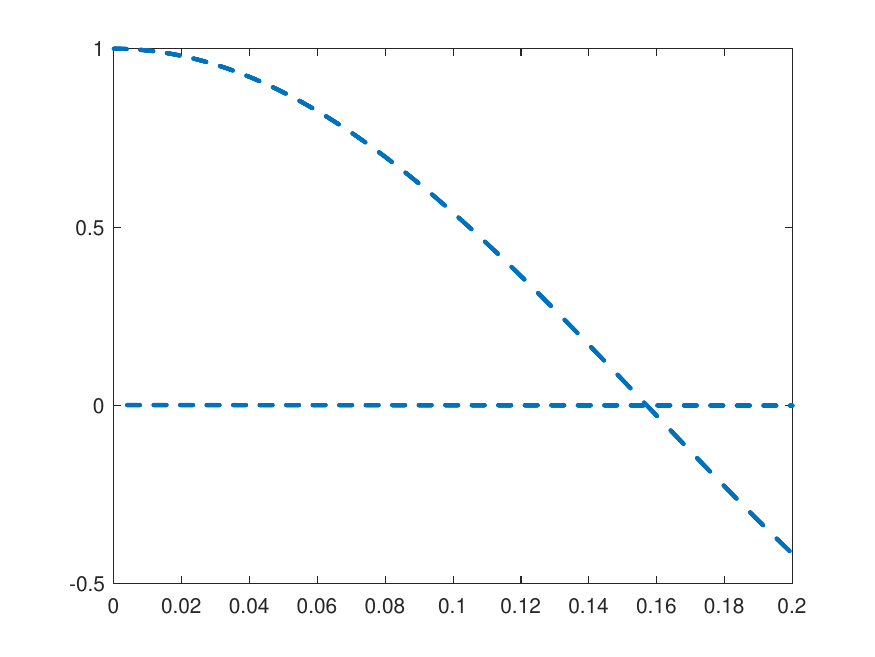}
	\end{subfigure}
	\begin{subfigure}[b]{0.23\textwidth}
		 \includegraphics[width=4.5cm,height=3.cm]{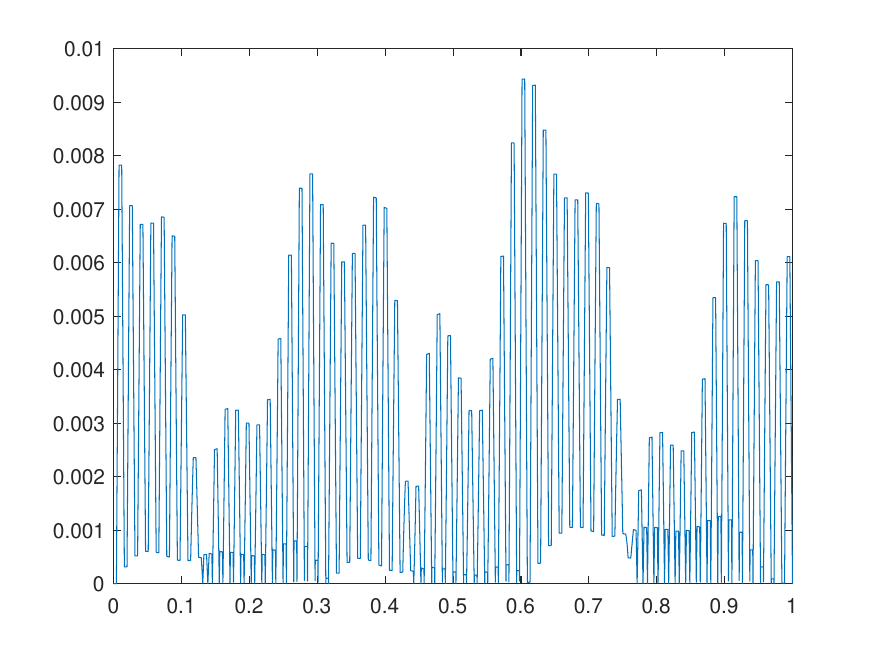}
	\end{subfigure}
	\begin{subfigure}[b]{0.23\textwidth}
		 \includegraphics[width=4.5cm,height=3.cm]{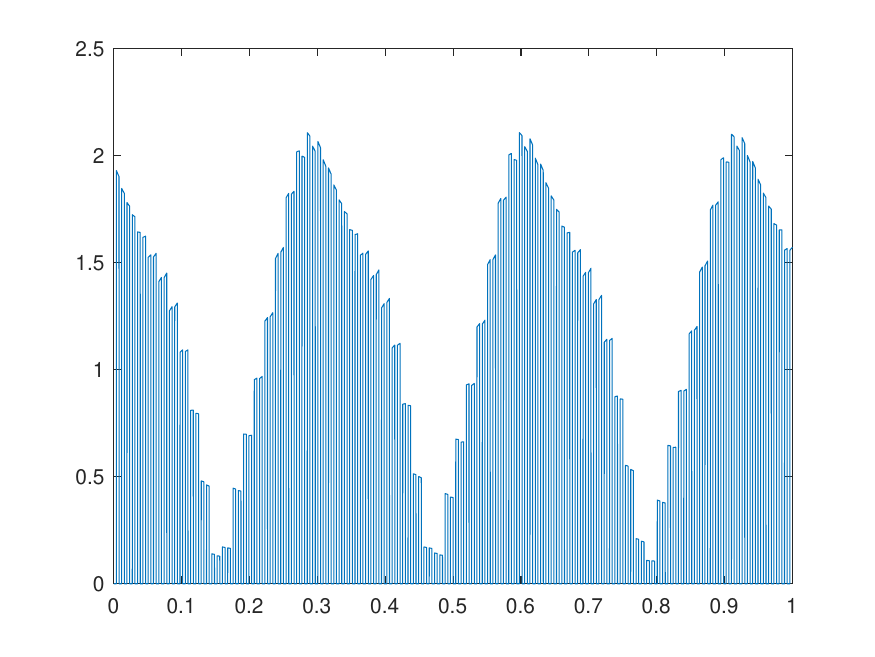}
	\end{subfigure}
	\begin{subfigure}[b]{0.23\textwidth}
		 \includegraphics[width=4.5cm,height=3.cm]{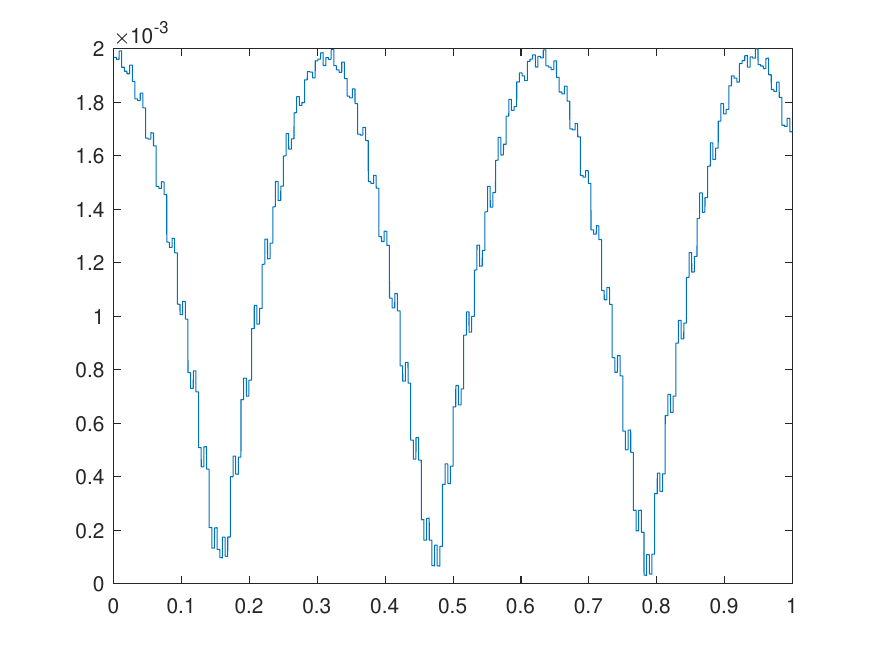}
	\end{subfigure}
	\caption{\cref{Special:ex6}: The top row: $a$ in [0,1] (first panel), $a$ in [0,0.2] (second panel), $\tilde{u}_{H}$ in [0,1] (third panel), and $\tilde{u}_{H}$ in [0,0.2] (fourth panel). The middle row: $\tilde{u}'_{H}$ in [0,1] (first panel), $\tilde{u}'_{H}$ in [0,0.2] (second panel), $a\tilde{u}'_{H}$ in [0,0.2] (third panel) with $H=\frac{1}{2^7}$, and $f$ in [0,1] (fourth panel). The bottom row: $f$ in [0,0.2] (first panel), $|\tilde{u}_{H}-\tilde{u}_{H/2}|$ (second panel), $|\tilde{u}'_{H}-\tilde{u}'_{H/2}|$ (third panel), and  $|a\tilde{u}'_{H}-a\tilde{u}'_{H/2}|$ (fourth panel) in [0,1] with $H=\frac{1}{2^6}$ at all  $x_i=i/N$ for $i=0,\dots, N$ with $N=2^{14}$ in \eqref{fine:xi}.}
	\label{fig:exam6}
\end{figure}

\section{Conclusion}\label{Special:sec:Conclu}
In this paper, we consider $-\nabla \cdot \left(a \nabla  u\right)=f$ with the zero Dirichlet boundary condition, $0<a_{\min}\le a\le  a_{\max}<\infty$, and $f\in L_2(\Omega)$ in the one-dimensional domain $\Omega=(0,1)$. To numerically compute a precise  and convincing approximated solution on the coarse mesh size $H$,  we present an innovative derivative-orthogonal wavelet multiscale method to calculate the numerical solution $\tilde{u}_H$. The main contributions of our proposed method are as follows:
\\
\textbf{Contributions in the theoretical analysis:}
\begin{itemize}
	\item We prove in \cref{TTY2} that the condition number $\kappa$ of the stiffness matrix of our proposed method satisfies that $\kappa \le \frac{ a_{\max} }{a_{\min}}$ for any coarse mesh size $H$.

	\item We establish in \cref{Thm:u:prime:L2} that the energy norm of the error of our proposed method attains the first-order
	convergence rate for any coarse mesh size $H$. Precisely, we prove that $\|\sqrt{a}(u'-(\tilde{u}_H)')\|_{L_2(\Omega)} \leq CH$, where $C=\frac{2}{\sqrt{a_{\min}}}\|f\|_{L_2(\Omega)}$.
	\item We prove in \cref{Thm:u:L2} that the $L_2$ norm of the error of our proposed method achieves the second-order
	convergence rate for any coarse mesh size $H$. Precisely, we prove that $\|u-\tilde{u}_{H}\|_{L_2(\Omega)} \leq CH^2$, where $C= \frac{4}{a_{\min}}\|f\|_{L_2(\Omega)}$.

	\item The explicit expressions of  constants $C$ in error estimates indicate that the energy and
	$L_2$-norm errors of our proposed method are independent to $a_{\max}$. Hence, our method can handle simultaneously the rough diffusion coefficient $a$ and the rough source term $f$, as long as $\|f\|_{L_2(\Omega)}$ is not huge.

\item We prove in \cref{Thm:u:infty}
that our numerical solutions $\tilde{u}_H$ have the interpolation property in \eqref{interpolation:1}.

\end{itemize}
\textbf{Contributions in the numerical computation:}
\begin{itemize}
	\item We provide
  several examples ($u$ is available/unavailable,  $a$ is continuous/discontinuous, and $\lim_{x\to 0^+}a(x)=\infty$) to measure $\frac{\|\tilde{u}_H-u\|_2}{\|u\|_2}$, $\frac{\|\tilde{u}'_H-u'\|_2}{\|u'\|_2}$, $\frac{\|a\tilde{u}'_H-au'\|_2}{\|au'\|_2}$, $\|\tilde{u}_H-u\|_\infty$, $\|\tilde{u}'_H-u'\|_\infty$,  and $\|a\tilde{u}'_H-au'\|_\infty$. The numerical results
	confirm bounded condition numbers and first-order/second-order convergence rates of the proposed method for elliptic equations with high-frequency and high-contrast coefficients.

\item Particularly,  the coefficient $a$ satisfies that $a_{\max}=\sup_{x\in \Omega} a(x)=\infty$ in \cref{Special:ex2:infty}, and  the coefficient $a$ and the  source term $f$ both exhibit discontinuous and high-frequency oscillations and  high-contrast jumps with $\|f\|_{L_2(\Omega)} \approx 0.5$ in \cref{Special:ex6}.  Our proposed method continues to perform robustly.

	\item Our derivative-orthogonal wavelet multiscale method is straightforward to be implemented, allowing for easy repetition and verification of the numerical results.
\end{itemize}
In the future, we plan to extend our proposed derivative-orthogonal wavelet multiscale method presented in this paper to the 2D case of  \eqref{Model:Problem} by modifying the regular and special basis functions.
\noindent\textbf{Acknowledgment}

Michael Neilan (neilan@pitt.edu, Department of Mathematics, University of Pittsburgh, Pittsburgh, PA 15260 USA) provided the initial idea and proof of \cref{Thm:u:infty} (the interpolation property) and recommended testing  \cref{Special:ex2:infty} with $\lim_{x\to 0^+}a(x)=\infty$ to verify the robustness of the proposed derivative-orthogonal wavelet multiscale method.

\end{document}